\documentclass[11pt]{article}

\usepackage{graphicx}
\usepackage{bm}
\usepackage{comment}
\usepackage{psfrag}
\usepackage{latexsym,amsmath,amsfonts,amscd, amsthm}
\usepackage{changebar}
\usepackage{color}
\usepackage{bm}
\usepackage{tikz}
\usepackage{multirow}
\usepackage{subfigure}
\usepackage{xcolor}

\usetikzlibrary{arrows,backgrounds,snakes}

\usepackage{chngcntr}

\newtheoremstyle{plainNoItalics}{}{}{\normalfont}{}{\bfseries}{.}{ }{}
\usepackage{amssymb}
\usepackage{bbm}

\theoremstyle{plain}
\newtheorem{thm}{Theorem}[section]
\newtheorem{lem}{Lemma}[section]

\theoremstyle{plainNoItalics}

\newtheorem{defn}[thm]{Definition}
\newtheorem{rem}[thm]{Remark}
\newtheorem{prop}[thm]{Proposition}

\newcommand{\beq}{\begin{equation}}
\newcommand{\eeq}{\end{equation}}
\newcommand{\beqa}{\begin{eqnarray}}
\newcommand{\eeqa}{\end{eqnarray}}
\newcommand{\bit}{\begin{itemize}}
\newcommand{\eit}{\end{itemize}}
\newcommand{\bedef}{\begin{defn}}
\newcommand{\edefn}{\end{defn}}
\newcommand{\bpro}{\begin{prop}}
\newcommand{\epro}{\end{prop}}

\newcommand{\df}{\partial}

\newcommand{\Dx}{\Delta x}
\newcommand{\Dt}{\Delta t}

\newcommand{\vareps}{\varepsilon}
\newcommand{\mD}{{\mathcal D}}

\newcommand{\iL}{{i-\frac{1}{2}}}

\newcommand{\testR}{{\phi}}   % test function for rho
\newcommand{\testG}{{\psi}}   % test function for g

\newcommand{\bI}{{\bf{I}}}

\newcommand{\mA}{\mathcal A}

\newcommand{\bh}{{b_{h,v}}}

\newcommand{\lgl}{\langle}
\newcommand{\rgl}{\rangle}

\counterwithin*{equation}{section}

\begin{document}
\baselineskip=1.2pc
\title{Asymptotic preserving IMEX-DG-S schemes for linear kinetic transport equations based on Schur complement }

\author{Zhichao Peng\thanks{Department of Mathematical Sciences, Rensselaer Polytechnic Institute, Troy, NY 12180, U.S.A. {\tt pengz2@rpi.edu}}
\and
 Fengyan Li\thanks{Department of Mathematical Sciences, Rensselaer Polytechnic Institute, Troy, NY 12180, U.S.A.
 {\tt lif@rpi.edu}. Research is supported by NSF grants DMS-1719942 and DMS-1913072.}}
\maketitle

\abstract{
We consider a linear kinetic transport equation under a diffusive scaling, that converges to a diffusion equation as the Knudsen number $\vareps\rightarrow 0$.  In \cite{boscarino2013implicit,peng2019stability}, to achieve the asymptotic preserving (AP) property and  unconditional stability in the diffusive regime with $\vareps\ll 1$, numerical schemes are developed based on an additional reformulation of the even-odd or micro-macro decomposed version of the equation. The key of the 
reformulation is to add a weighted diffusive term on both sides of one equation in the decomposed system. The choice of the weight function, however,  is problem-dependent and ad-hoc, and it can affect 
the performance of numerical simulations. To avoid issues related to the choice of the weight function and still obtain the AP property and unconditional stability in the diffusive regime, we propose in this paper a new family of AP schemes, termed as  IMEX-DG-S schemes,  directly solving the micro-macro decomposed system without any further reformulation. The main ingredients of the IMEX-DG-S schemes include globally stiffly accurate implicit-explicit (IMEX) Runge-Kutta (RK) temporal discretizations with a new IMEX strategy, discontinuous Galerkin (DG) spatial discretizations, discrete ordinate methods for the velocity space, and the application of the Schur complement to the algebraic form of the schemes 
 to control the overall computational cost. The AP property of the schemes is shown formally. With an energy type stability analysis applied to the first order scheme, and Fourier type stability analysis applied to the first to third order schemes,  we confirm the uniform stability of the methods with respect to $\varepsilon$ and the unconditional stability in the diffusive regime. A series of numerical examples are presented to demonstrate the performance of the new schemes.
}

%%%%%%%%%%%%%%%%%%
% main body of the article
%%%%%%%%%%%%%%%%%%
\section{Introduction}
\label{sec:introduction}

Important physical phenomena like radiative  transfer and neutron transport can be modeled by  
kinetic transport equations. 
In this work, we  consider  a linear kinetic transport equation under a diffusive scaling:
\begin{align}
\label{eq:f_model}
\vareps \partial_t f + v\partial_x f = \frac{\sigma_s}{\vareps}(\lgl f \rgl - f) - \vareps \sigma_a f.
\end{align}
Here, $f(x,v,t)$ is the 
probability distribution function of particles,
$x\in \Omega_x$ is the spatial position, $v\in \Omega_v$ is the velocity  with $\Omega_v$ being  bounded, and $t$ is  time.  $\sigma_s(x)\geq 0$ and $\sigma_a(x)\geq 0$ are the scattering and the absorption coefficients, respectively. $\lgl f \rgl =\int_{\Omega_v} f d\nu$,  
where $\nu$ is a measure of the velocity space satisfying $\int_{\Omega_v}d\nu=1$. 
$\vareps>0$ denotes the Knudsen number, which is the ratio of mean free path of particles to the characteristic length.

The linear kinetic transport equation \eqref{eq:f_model} has a multiscale nature. With the assumption $\sigma_s>0$, the kinetic transport equation will converge to its diffusion limit as $\vareps\rightarrow 0$: 
\begin{align}
\label{eq:diffusion_limit_intro}
\partial_t \rho =  \lgl v^2\rgl \partial_{x}(\partial_x\rho/\sigma_s)-\sigma_a \rho,
\end{align}
where $\rho=\lgl f \rgl$ is the macroscopic density of particles. This multiscale nature poses computational challenges: (1) a standard explicit numerical scheme has a time step restriction $\Dt\leq O(\vareps h)$ ($h$ is the mesh size) due to numerical stability, with prohibitive computational cost for small $\vareps$; (2)  an implicit scheme, though possibly being  unconditionally stable  and hence not suffering from stability issue, may still fail to capture the correct physical limit as $\vareps\rightarrow 0$ on under-resolved meshes \cite{caflisch1997uniformly,naldi1998numerical}.

Asymptotic preserving (AP) schemes \cite{jin2010asymptotic,degond2011asymptotic}, which preserve the asymptotic behavior of the physical model on the discrete level, are a well established candidate to address above challenges.  An AP scheme for \eqref{eq:f_model} converges to a scheme solving the limiting  diffusion equation \eqref{eq:diffusion_limit_intro} as $\vareps\rightarrow 0$, while being consistent and stable for a broad range of $\vareps$, even on under-resolved meshes for small $\vareps$. 
 AP schemes having explicit limiting schemes are considered in \cite{jin1998diffusive,klar1998asymptotic,jin2000uniformly,lemou2008new, jang2015high}.
 With the diffusive nature of the physical limit and the explicitness of the limiting schemes, these methods have a parabolic time step restriction $\Dt= O(h^2)$ in the diffusive regime $\vareps\ll 1$. To enhance the stability, AP schemes with  implicit limiting schemes are developed in \cite{boscarino2013implicit,peng2019stability}, and they are  demonstrated, either numerically or analytically,  to be unconditionally stable in the diffusive regime.

To achieve unconditional stability in the diffusive regime, in \cite{boscarino2013implicit}, a second reformulation is introduced to 
the even-odd decomposition \cite{lewis1984computational, jin2000uniformly} of \eqref{eq:f_model}. 
 And a similar strategy is applied  to the micro-macro decomposition \cite{liu2004boltzmann} of \eqref{eq:f_model} in \cite{peng2019stability}. Taking the methods in \cite{peng2019stability} as an example,
 a weighted diffusive term $\omega \partial_{x}(\partial_x\rho/\sigma_s)$, 
 determined by the diffusion limit, is added to both sides of one equation in the micro-macro decomposed system. Then, a suitable implicit-explicit (IMEX) Runge-Kutta (RK) time discretization is applied to the newly reformulated system. 
 Under the adopted IMEX strategy, the two added diffusive terms 
 are treated differently.
 In space, local discontinuous Galerkin (LDG) method  \cite{cockburn1998local}  is applied. The resulting IMEX-LDG schemes are  AP and can be high order,
with the limiting schemes being implicit to solve the diffusion limit.  Based on Fourier analysis, stability condition is obtained in \cite{peng2019stability} by numerically solving an eigenvalue problem, and it confirms the unconditional stability in the diffusive regime. Following an energy approach, the stability condition is rigorously established  in \cite{peng2020analysis} for the first order in time scheme applied to the model with general material properties $\sigma_s(x)$ and $\sigma_a(x)$.

In a multiscale problem, $\sigma_s(x)$ may vary spatially, and the diffusion dominant and transport dominant subregions can coexist. Despite the success of enhancing the stability in the diffusive regime, the strategy in \cite{boscarino2013implicit,peng2019stability} with an  additional reformulation also faces some issues.
 First, the choice of the numerical weight is problem-dependent, and this ad-hoc choice has an influence on the performance of numerical simulations. Second,  when the scattering coefficient $\sigma_s(x)$ varies spatially, intuitively, a spatially dependent weight function $\omega$ seems to be preferred to better capture the multiscale behavior. However, with such a spatially dependent weight, an extra non-physical assumption would be needed to maintain local conservation property, see Section \ref{sec:model} for more discussions.

To overcome these issues and still accomplish unconditional stability in the diffusive regeime, we here design a new family of IMEX discontinuous Galerkin (DG) AP schemes  
based on the Schur complement \cite{zhang2006schur}, referred to  as IMEX-DG-S methods. Our new schemes directly solve the micro-macro decomposed system without any additional reformulation, and hence they do not suffer from issues mentioned above. 
In time, we apply  
globally stiffly accurate IMEX-RK time integrators \cite{boscarino2013implicit} with a new IMEX strategy. In space, we use DG discretizations \cite{cockburn2012discontinuous} with carefully chosen  numerical fluxes. In the velocity space, a discrete ordinates method \cite{pomraning1973equations}
is utilized.  On the solver level, the key is to apply the Schur complement to the fully discrete system, and this is important for the computational cost. Indeed our proposed methods  have comparable computational complexity as the IMEX-LDG schemes in \cite{peng2019stability}.
As $\vareps\rightarrow 0$, asymptotic analysis shows that our new schemes formally converge to high order methods that involve implicit RK methods in time and  LDG methods in space solving the diffusion limit, implying the AP property of the schemes. When an initial layer exists in the solution,  our schemes no longer need special treatment for the first step as in \cite{peng2019stability}  in order to stay AP.
With an energy type stability analysis applied to the first order scheme, and Fourier type stability analysis applied to the first to third order schemes, we obtain stability conditions that confirm  the uniform stability of the methods with respect to $\varepsilon$ and the unconditional stability in the diffusive regime. A discrete energy different from that in \cite{jang2014analysis, peng2020analysis}  is used in the energy analysis. 
 Numerical examples are presented to demonstrate the performance of the IMEX-DG-S schemes and its advantage over the IMEX-LDG schemes in some test cases.

The rest of the paper is organized as follows. In Section \ref{sec:model}, we present  the micro-macro decomposition, and briefly review the additional  reformulation strategy in 
\cite{peng2019stability}
 to motivate our work.  In Section \ref{sec:method}, we define our numerical schemes, and provide the details of the Schur complement for the final matrix-vector system. In Section \ref{sec:AP}, formal asymptotic analysis is presented to confirm the AP  property. In Section \ref{sec:stability}, energy and Fourier analyses are performed to obtain stability conditions. In Section \ref{sec:numerical}, numerical results are reported to illustrate the performance of the proposed schemes, and this will be followed by conclusions in  Section \ref{sec:conclusion}.

\section{Micro-macro decomposition and motivation }
\label{sec:model}

Following the micro-macro decomposition framework \cite{liu2004boltzmann,lemou2008new}, we reformulate \eqref{eq:f_model}. We first define a scattering operator $\mathcal{L} f=\lgl f \rgl - f$, and let $\Pi$ denote the $L^2$ projection onto the null space of $\mathcal{L}$: $\textrm{Null}(\mathcal{L})=\textrm{Span}\{1\}$. 
 We then decompose $f$ orthogonally into $f=\rho +\vareps g$, where $\rho:= \Pi f= \lgl f \rgl$, and  $g:=\frac{1}{\vareps}({\bf{I}}-\Pi)f$ satisfying $\lgl g\rgl=0$. Finally  we apply  $\Pi$ and its orthogonal complement ${\bf{I}}-\Pi$ to 
 \eqref{eq:f_model}, and obtain the micro-macro decomposed system:
\begin{subequations}
\label{eq:micro_macro}
\begin{align}
\label{eq:micro_macro_rho}
&\partial_t \rho + \partial_x \lgl vg \rgl = -\sigma_a \rho,\\
\label{eq:micro_macro_g}
\vareps& \partial_t g + ({\bf{I} }-\Pi)\partial_x (vg) + \frac{1}{\vareps}v\partial_x \rho = -\frac{\sigma_s}{\vareps} g -\vareps \sigma_a g.
\end{align}
\end{subequations}  
Assume $\sigma_s(x)>0$. As the Knudsen number $\vareps\rightarrow 0$, \eqref{eq:micro_macro} formally converges to its diffusion limit
\begin{align}
\label{eq:diffusion_limit}
\sigma_s g= - v\partial_x \rho, \qquad \partial_t \rho = -\partial_x \lgl vg\rgl-\sigma_a\rho= \lgl v^2\rgl \partial_{x}(\partial_x\rho/\sigma_s)-\sigma_a \rho.
\end{align}

To define AP schemes with unconditional stability in the diffusive regime with $\vareps\ll1$, \cite{peng2019stability} applies a second reformulation to \eqref{eq:micro_macro}
 by adding the weighted diffusive term $\lgl v^2\rgl \partial_{x}(\omega\partial_x\rho/\sigma_s)$ to both sides of \eqref{eq:micro_macro_rho}:
\begin{subequations}
\label{eq:micro_macro_reformulation}
\begin{align}
\label{eq:micro_macro_rho_reformulation}
&\partial_t \rho + \partial_x \lgl v\left(g +v\omega(\partial_x\rho/\sigma_s)\;\right)\rgl = \lgl v^2\rgl \partial_{x}(\omega\partial_x\rho/\sigma_s) -\sigma_a \rho,\\
\label{eq:micro_macro_g_reformulation}
\vareps& \partial_t g + ({\bf{I} }-\Pi)\partial_x (vg) + \frac{1}{\vareps}v\partial_x \rho = -\frac{\sigma_s}{\vareps} g -\vareps \sigma_a g.
\end{align}
\end{subequations}  
Here, $\omega$ is a non-negative numerical weight function, and it satisfies $\omega\rightarrow 1$ as $\vareps\rightarrow 0$.

In \cite{peng2019stability}, globally stiffly accurate IMEX-RK time discretizations are applied, where the weighted diffusive term $\lgl v^2\rgl \partial_{x}(\omega\partial_x\rho/\sigma_s)$ on the left hand side of \eqref{eq:micro_macro_rho_reformulation} is treated explicitly and that on the right hand side is treated implicitly.  With LDG methods further applied in space, the resulting schemes  are AP,
unconditionally stable in the diffusive regime, and they also show good performance numerically. However, the choice of $\omega$ is problem-dependent, and it can affect the performance of the methods (see. e.g. Examples 2, 4, 5  in  Section \ref{sec:numerical}).
Moreover, when $\sigma_s(x)$ is not constant,  a spatially dependent weight $\omega$ would  be preferred intuitively  in order to better capture the multiscale behavior. If such weight function is used, one would need to assume $\omega\partial_x \rho/\sigma_s$ to be continuous to ensure the local conservation, and this is apparently not physical due to the weight-dependence. 
As far as we know, only weight functions that do not vary spatially have been considered in the literature. 

\section{Numerical methods}
\label{sec:method}

In this section, we will design a new family of AP schemes directly based on the micro-macro decomposition \eqref{eq:micro_macro}, aiming at achieving unconditional stability in the diffusive regime without the need for a weight function.  In the following subsections, we will present discretizations in time, in space, and in velocity. For the fully discrete schemes, Schur complement will be applied to their   algebraic systems. We end this section by extending  the IMEX strategy to some more general  kinetic transport models. Throughout this section,  periodic boundary conditions are assumed in space.

%%%%%%%%%%%%%%%%%%%%%%%%%%%%%%%%%%%%%%%%%%%%%%%%%%%%%%%%%%%%%%%%%% 
% Time discretization
%%%%%%%%%%%%%%%%%%%%%%%%%%%%%%%%%%%%%%%%%%%%%%%%%%%%%%%%%%%%%%%%%%  

\subsection{Time discretization}
\label{sec:time}
In time, we apply globally stiffly accurate IMEX-RK methods of type ARS. The first order one is defined as follows. 
 Given $\rho^n$ and $g^n$ at $t=t^n$, we seek $\rho^{n+1}$ and $g^{n+1}$ at $t^{n+1}=t^n+\Dt$, which satisfy
\begin{subequations}
\label{eq:time}
\begin{align}
\label{eq:time_rho}
&\frac{\rho^{n+1}-\rho^n}{\Dt} + \partial_x\lgl vg^{n+1}\rgl = -\sigma_a \rho^{n+1},\\
\label{eq:time_g}
&\frac{g^{n+1}-g^n}{\Dt} + \frac{1}{\vareps}({\bf{I}}-\Pi)(v\partial_x g^n)+\frac{1}{\vareps^2}v\partial_x \rho^{n+1} = -\frac{1}{\vareps^2}\sigma_s g^{n+1}-\sigma_a g^{n+1}.
\end{align}
\end{subequations}
The same IMEX strategy in \eqref{eq:time} will be also used to achieve high order temporal accuracy.  Recall that a general IMEX-RK scheme can be  
represented by a double Butcher tableau:
\beq
\label{eq: B_table}
\begin{array}{c|c}
\tilde{c} & \tilde{\mA}\\
\hline
 & {\tilde{b}}^T \end{array} \ \ \ \ \
\begin{array}{c|c}
{c} & {\mA}\\
\hline
 & {b^T} \end{array}~~.
\eeq
Here, $\tilde{\mA}=(\tilde{a}_{ij})$, $\mA=(a_{ij})$ are $s\times s$ lower triangular matrices, and $\tilde{a}_{ii}=0$, $i=1, \cdots s$. $\tilde{b}=(\tilde{b}_i)$, $b=(b_i)$, $\tilde{c}=(\tilde{c}_i),$ and $c=(c_i)$ are $s$-dimensional vectors, and $\tilde{c}_i=\sum_{j=1}^{i-1}\tilde{a}_{ij}$, $c_i=\sum_{j=1}^i a_{ij}$. 
An IMEX-RK scheme is {\it{globally stiffly accurate}} \cite{boscarino2013implicit} if 
$$ c_s = \tilde{c}_s=1,\;\text{and}\; b_j= a_{sj},\; \tilde{b}_j=\tilde{a}_{sj},\; \forall j=1,\dots,s,$$ and it is of type ARS \cite{ascher1997implicit} if
\begin{equation}
\label{eq:A:assume}
\mA=\left[
\begin{array}{cc}
0 & 0\\
0 & \hat{\mA}
\end{array}
\right],\quad \textrm{where} \; \hat{\mA}\; \textrm{is invertible}.
\end{equation}
For the second and third order accuracy, we use ARS(2,2,2) and  ARS(4,4,3), respectively \cite{ascher1997implicit,boscarino2013implicit}.
%, with their formulas given in Appendix \ref{sec:appendix}. 

%%%%%%%%%%%%%%%%%%%%%%%%%%%%%%%%%%%%%%%%%%%%%%%%%%%%%%%%%%%%%%%%%%  
% Spatial discretization
%%%%%%%%%%%%%%%%%%%%%%%%%%%%%%%%%%%%%%%%%%%%%%%%%%%%%%%%%%%%%%%%%%  
\subsection{Space discretization}
\label{sec:space}

Let $\Omega_x=[x_L,x_R]$ be the computational domain, and $\Omega_h = \left\{ I_i=[x_{i-\frac{1}{2}},x_{i+\frac{1}{2}}],i=1,\dots, N\right\}$ be a partition of $\Omega_x$. Let $x_i=({x_{i-\frac{1}{2} }+x_{i+\frac{1}{2}} })/{2}$, $h_i = x_{i+\frac{1}{2} }-x_{i-\frac{1}{2}}$, $h=\max_i{h_i}$. Given a nonnegative integer $k$, define the discrete space
$U_h^k=\{u\in L^2(\Omega_x): u|_{I_i}\in P^k(I_i),\forall 1\leq i\leq N\} $, where $P^k(I_i)$ denotes the space of polynomials with degree at most $k$ on $I_i$. Define $u^{\pm}_{i+\frac{1}{2}}=\lim_{\Delta x\rightarrow 0^{\pm}} u(x_{i+\frac{1}{2} }+\Dx)$ and the jump  $[u]_{i+\frac{1}{2}}=u^+_{i+\frac{1}{2} }-u^-_{i-\frac{1}{2} }$, $\forall i$.

We apply the following DG spatial discretization to the semi-discrete method in \eqref{eq:time}. Given numerical solutions $\rho_h^n$ and $g_h^n\in U_h^k$, we look for $\rho_h^{n+1}, g_h^{n+1}\in U_h^{k}$ satisfying, $\forall \testR,\testG \in U_h^k$
\begin{subequations}
\label{eq:space}
\begin{align}
\label{eq:space_rho}
&(\frac{\rho_h^{n+1}-\rho_h^n}{\Dt},\testR) + l_h\left(\lgl vg_h^{n+1}\rgl,\testR\right) = -\left(\sigma_a \rho_h^{n+1},\testR\right),\\
\label{eq:space_g}
&(\frac{g_h^{n+1}-g_h^n}{\Dt},\testG) + \frac{1}{\vareps}\widetilde{b}_{h,v}(g_h^n,\testG)-\frac{1}{\vareps^2}vd_h(\rho_h^{n+1},\testG) = -\frac{1}{\vareps^2}(\sigma_s g_h^{n+1},\testG)-(\sigma_a g_h^{n+1},\testG).
\end{align}
\end{subequations}
Here, $(\cdot,\cdot)$ is the standard $L^2$ inner product of $L^2(\Omega_x)$. Bilinear forms $d_h(\cdot,\cdot)$, $l_h(\cdot,\cdot)$,  $\widetilde{b}_{h,v}(\cdot,\cdot)$ are defined as
\begin{subequations}
\label{eq:bilinear}
\begin{align}
\label{eq:dh}
d_h(\rho_h, \testG)&=\sum_i\int_{I_i} \rho_h \df_x\testG dx + \sum_i (\breve{\rho_h})_{\iL}[\testG]_\iL , \\
\label{eq:lh}
l_h(\lgl vg_h\rgl,\testR)&=-\sum_i \int_{I_i} \lgl vg_h\rgl \partial_x\testR dx-\sum_i \widehat{\lgl vg_h\rgl}_\iL [\testR]_\iL,\\
\label{eq:bh}
\widetilde{b}_{h,v}(g_h,\testG)&=((\bI-\Pi)\mD^{up}_h(g_h; v), \testG) =(\mD^{up}_h(g_h; v) - \langle\mD^{up}_h(g_h; v)\rangle, \testG),
\end{align}
\end{subequations}
where $\mD^{up}_h(g_h;v)\in U_h^k$ is an upwind approximation to $v\partial_x g$ for any given velocity $v$:
\begin{align}
(\mD^{up}_h(g_h; v), \testG)
=-\sum_i\left(\int_{I_i} vg_h \df_x\testG dx\right) - \sum_i \widetilde{(vg_h)}_\iL[\testG]_\iL, \quad \forall \psi\in U_h^k.
\label{eq:mD}
\end{align}
$(\widetilde{vg_h})_\iL$, $\widehat{\lgl vg_h\rgl}_\iL$ and $(\breve{\rho_h})_\iL$ are numerical fluxes and chosen as:
\begin{subequations}
\label{eq:flux}
\begin{align}
\label{eq_upwind}
&\textrm{upwind:}\quad(\widetilde{vg_h})_\iL:=
\left\{
\begin{array}{ll}
({vg_h})_\iL^-,&\mbox{if}\; v\geq0,\\
({vg_h})_\iL^+,&\mbox{if}\; v<0,
\end{array}
\right.
\;\;\\
&\textrm{alternating:}\quad
(\breve{\rho_h})_\iL=(\rho_h)_\iL^-,\;\;\widehat{\lgl vg_h\rgl}_\iL = {\lgl vg_h\rgl}_\iL^+.
\end{align}
\end{subequations}
Based on the Riesz representation, we can further find two well-defined bounded linear operators 
$\mD_h^-,\mD_h^+: U_h^k\rightarrow U_h^k$ such that
$$ (\mD_h^-\testR,\testG) = -d_h(\testR,\testG),\;\; (\mD_h^+\testR,\testG) = l_h(\testR,\testG),\;\;\forall \testR,\testG\in U_h^k.$$
$\mD_h^{\pm}$ can be seen as discrete derivative operators, and the scheme \eqref{eq:space} can be rewritten as:
\begin{subequations}
\label{eq:space_derivative}
\begin{align}
\label{eq:space_rho_d}
&\frac{\rho_h^{n+1}-\rho_h^n}{\Dt} + \mD_h^+\lgl vg_h^{n+1}\rgl = -\pi_h(\sigma_a \rho_h^{n+1}),\\
\label{eq:space_g_d}
&\frac{g_h^{n+1}-g_h^n}{\Dt} + \frac{1}{\vareps}(\bI-\Pi)\mD^{up}_h(g^n_h; v)+\frac{v}{\vareps^2}\mD_h^-\rho_h^{n+1} = -\frac{1}{\vareps^2}\pi_h(\sigma_s g_h^{n+1})-\pi_h(\sigma_a g_h^{n+1}),.
\end{align}
\end{subequations}
with $\pi_h$ being the $L^2$ projection onto $U_h^k$.

The DG spatial discretization can be coupled directly with  high order IMEX-RK time integrators. At $t=0$,  $\rho_h^0$ and $g_h^0$ are initialized by $L^2$ projection, namely, $\rho_h^0=\pi_h(\rho(x,0))$ and  $g_h^0=\pi_h(g(x,v,0))$.  The following lemma summarizes a property of the bilinear forms $l_h$ and $d_h$, and it is  important in stability analysis and can be easily verified. 

%%%%%%%%%%%%%%%%%%%%%%%%%%%%%%
% Lemma: adjoint relation of discrete spatial operator
%%%%%%%%%%%%%%%%%%%%%%%%%%%%%%
\begin{lem}\label{lem:adjoint_relation}
With periodic boundary conditions in space, there hold
\begin{align}
l_h(\testR,\testG) = d_h(\testG,\testR),\;\; \forall \testR,\testG\in U_h^k, \quad\text{and}\quad\mD_h^+=-(\mD_h^-)^\top,
\end{align}
where the superscript $^\top$ to an operator denotes its adjoint.
\end{lem}

%%%%%%%%%%%%%%%%%%%%%%%%%%%%%%%%%%%%%%%%%%%%%%%%%%%%%%%%%%%%%%%%%%  
%Velocity discretization
%%%%%%%%%%%%%%%%%%%%%%%%%%%%%%%%%%%%%%%%%%%%%%%%%%%%%%%%%%%%%%%%%%
\subsection{Velocity discretization}
\label{sec:velocity}

In velocity variable, we will apply the discrete ordinates method \cite{pomraning1973equations}. 
Let $\{v_l\}_{l=1}^{N_v}$ denote a set of quadrature points as collocation points in the velocity space $\Omega_v$ and $\{\omega_l\}_{l=1}^{N_v}$ be the corresponding quadrature weights. An integral in velocity will be 
approximated by 
\begin{align}
\lgl \eta(v) \rgl = \int_{\Omega_v} \eta(v) d\nu\approx \sum_{l=1}^{N_v}\omega_l \eta(v_l)\triangleq:\lgl \eta(v)\rgl_h,\;
\end{align} 
Particularly, we choose $\{\omega_l\}_{l=1}^{N_v}$ and $\{v_l\}_{l=1}^{N_v}$  satisfying 
\begin{align}
\label{eq:v_square_integral}
\lgl v^2 \rgl = \lgl v^2\rgl_h. 
\end{align}
 This requirement is essential for our fully discrete schemes to capture the correct diffusion limit as $\vareps\rightarrow 0$.

%%%%%%%%%%%%%%%%%%%%%%%%%%%%%%%%%%%%%%%%%%%%%%%%%%%%%%%%%%%%%%%%%%  
%Fully discrete 
%%%%%%%%%%%%%%%%%%%%%%%%%%%%%%%%%%%%%%%%%%%%%%%%%%%%%%%%%%%%%%%%%%
\subsection{Fully  discrete schemes}
\label{sec:fully}

By combining the temporal, spatial, and velocity discretizations described above, we are now ready to present the fully discrete schemes:
 Given $\rho_h^n\in U_h^k$, $\{g_{h,l}^n\}_{l=1}^{N_v}\in U_h^k$, we look for $\rho_h^{n+1}\in U_h^k$, $\{g_{h,l}^{n+1}\}_{l=1}^{N_v}\in U_h^k$, satisfying for $i=1,\dots, s$, $l=1, \dots, N_v$ 
\begin{subequations}
\label{eq:full_high}
\begin{align}
(\rho_h^{n,(i)},\testR) &= (\rho_h^n,\testR) -\Dt \sum_{j=1}^i a_{ij}\left( (\mD_h^+\lgl vg_h^{n,(j)}\rgl_h,\testR)+(\sigma_a\rho_h^{n,(j)},\testR)\right), \forall \testR \in U_h^k \\
\vareps^2(g_{h,l}^{n,(i)},\testG) &= \vareps^2(g_{h,l}^n,\testG) -\vareps\Dt \sum_{j=1}^{i-1} \tilde{a}_{ij}\left( 
(\mD^{up}_h(g^{n,(j)}_{h}; v_l), \testG)
-(\lgl\mD^{up}_h(g^{n,(j)}_{h};v)\rgl_h,\testG) \right)\notag\\
				 		-&\Dt\sum_{j=1}^i a_{ij}\left( v_l(\mD_h^- \rho_h^{n,(j)},\testG)+(\sigma_s g_{h,l}^{n,(j)},\testG)+\vareps^2(\sigma_a g_{h,l}^{n,(j)},\testG)\right), \forall\testG\in U_h^k \\
\rho_h^{n+1} &= \rho_h^{n,(s)},\quad g_{h,l}^{n+1} = g_{h,l}^{n,(s)}. \label{eq:full_high_last_stage}
%\label{eq:time_high_last_stage}
\end{align}
\end{subequations}
We here have used  $g_{h,l}^n:=g_h^n(\cdot,v_l)$, $g_{h,l}^{n,(i)}:=g_h^{n,(i)}(\cdot,v_l)$, and 
$\lgl \mathcal{G}(g_h^{n,(j)})\rgl_h =\sum_{l=1}^{N_v}\omega_l \mathcal{G}(g_{h,l}^{n,(j)})$ with $\mathcal{G}: L^2(\Omega_v)\rightarrow L^2(\Omega_v)$. And the intermediate functions $\rho_h^{n,(i)}, g_{h,l}^{n,(i)}$ are also in $U_h^k$.
 
Particularly, the first order in time scheme is: $\forall \testR,\testG\in U_h^k$,  %$\forall n\geq 0$, 
$l=1,\dots,N_v$,
\begin{subequations}
\label{eq:full_first}
\begin{align}
\label{eq:full_first_rho}
&(\frac{\rho_h^{n+1}-\rho_h^n}{\Dt},\testR) + l_h\left(\lgl vg_h^{n+1}\rgl_h,\testR\right) = -\left(\sigma_a \rho_h^{n+1},\testR\right),\\
\label{eq:full_first_g}
&(\frac{g_{h,l}^{n+1}-g_{h,l}^n}{\Dt},\testG) + \frac{1}{\vareps}b_{h,v}(g_{h,l}^n,\testG)-\frac{1}{\vareps^2}vd_h(\rho_h^{n+1},\testG) = -\frac{1}{\vareps^2}(\sigma_s g_{h,l}^{n+1},\testG)-(\sigma_a g_{h,l}^{n+1},\testG),
\end{align}
\end{subequations}
where
\begin{align}
b_{h,v}(g_{h,l}^n,\testG) = (\mD^{up}_h(g^{n}_{h}; v_l), \testG) -(\lgl\mD^{up}_h(g^{n}_{h};v)\rgl_h,\testG).
\end{align}

From here on, we will use IMEX$p$-DG-S to refer to the fully discrete scheme with $p$-th order IMEX-RK time integrator, and use IMEX$p$-DG$k$-S with the discrete space $U_h^{k-1}$ in the spatial discretization. Here S stands for the Schur complement, which will be discussed in next subsection. Finally one can obtain the following property of the numerical solution following a similar proof of Lemma 3.1 in \cite{jang2014analysis},
\begin{equation}
\lgl g_h^{n}\rgl_h=0, \quad \forall n\geq 0.
\label{eq:g:prop}
\end{equation}

%%%%%%%%%%%%%%%%%%%%%%%%%%%%%%%%%%%%%%%%%%%%%%%%%%%%%%%%%%%%%%%%%%  
% Matrix-Vector form
%%%%%%%%%%%%%%%%%%%%%%%%%%%%%%%%%%%%%%%%%%%%%%%%%%%%%%%%%%%%%%%%%%  
\subsection{Matrix-vector formulation and Schur complement}
% of IMEX1 scheme}
\label{sec:matrix-vector}

To implement the proposed schemes, we will further apply Schur complement at the algebraic level. With this, our methods will have comparable computational complexity as the IMEX-LDG schemes in  \cite{peng2019stability,peng2020analysis}.  Next we use  the first order in time IMEX1-DG-S scheme to illustrate. Similar discussion can go to the high order in time schemes.

%First, we present the matrix-vector formulation of the fully discrete IMEX1 scheme \eqref{eq:full_first}. Then, we apply the Schur complement to the resulting linear system. Even though both $\partial_x \lgl vg^{n+1}\rgl$
%and $v \partial_x\rho^{n+1}$ are implicit, one just needs to invert one discrete diffusion operator applying to $\rho$ on each time step. Similar strategy can be applied to high order in time schemes.

We start with the matrix-vector formulation of the IMEX1-DG-S scheme  \eqref{eq:full_first}. Let $\{e_l(x) \}_{l=1}^m$ be a basis of the discrete space $U_h^k$. Define ${\bf{e}}=(e_1(x),\dots,e_m(x) )^T$. Then the numerical solutions can be expanded as
$$\rho_h^n(x)=\sum_{i=1}^m \rho_i^n e_i(x)=({\boldsymbol{\rho} }^n)^T{\bf{e}},\;\; g_{h,j}^n(x)=\sum_{i=1}^m g^n_{j,i} e_i(x)=({\bf{g}}^n_j)^T{\bf{e}},$$
where ${\boldsymbol{\rho} }^n=(\rho_1^n,\dots,\rho_m^n)^T$ and ${\bf{g}}^n_j=(g_{j,1}^n,\dots,g_{j,m}^n)^T$.

Define the mass matrix $(M)_{ij}=(e_j,e_i)$ and stiff matrices $(D^+)_{ij} = (\mD_h^+ e_j,e_i)$, $(D^-)_{ij}=(\mD_h^- e_j,e_i)$. Also define $(\Sigma_s)_{ij}=(\sigma_s e_j,e_i)$ and $(\Sigma_a)_{ij}=(\sigma_a e_j, e_i)$. The fully discrete IMEX1-DG-S scheme \eqref{eq:full_first} can be written into its matrix-vector form:
\begin{subequations}
\label{eq:matrix-vector}
\begin{align}
\mathcal{L}\left({\boldsymbol{\rho} }^{n+1},{\bf{g}}_1^{n+1},{\bf{g}}_2^{n+1},\dots{\bf{g}}_{N_v}^{n+1}\right)^T
=\left({\bf{b}}_0^n,{\bf{b}}_1^n,{\bf{b}}_2^n,\dots,{\bf{b}}_{N_v}^n\right)^T,
\quad\text{and,}\\[0.15\baselineskip]
\mathcal{L}=
\left(
\begin{matrix}
M + \Dt\Sigma_a	  	&\Dt \omega_1 v_1 D^+ 		& \Dt \omega_2 v_2 D^+ 		&\dots 	& \Dt\omega_{N_v}v_{N_v}D^+\\
v_1 \Dt D^- 			& \Theta	& 0 						&\dots 	& 0 \\
v_2 \Dt D^- 			& 0						& \Theta	&\dots 	& 0 \\
\vdots	  			& \vdots					& \vdots 					&\ddots	&\vdots\\
v_{N_v}\Dt D^-			& 0						& 0 						&\dots  	&\Theta
\end{matrix}
\right).
\end{align}
\end{subequations}
Here $\Theta = \vareps^2 (M +\Dt\Sigma_a)+ \Dt\Sigma_s$, and  ${\bf{b}}_j^n$, $\forall j=0,\dots, N_v$, are vectors determined by the data on time level $n$. Given that the mass matrix $M$ is symmetric positive definite (SPD) and $\sigma_s\geq 0$, $\sigma_a\geq 0$, $\Theta$ is SPD hence invertible.  Following the standard procedure of the Schur complement \cite{zhang2006schur}, we first express $ {\bf{g}}_j^{n+1}$ in terms of ${\bf{b}}_j^n$ and ${\boldsymbol{\rho}}^{n+1}$, namely,
\begin{align}
\label{eq:schur_step1}
{\bf{g}}_j^{n+1} = \Theta^{-1}\left( {\bf{b}}_j^n - v_j \Dt D^-{\boldsymbol{\rho}}^{n+1}\right), \quad\forall j=1,\dots, N_v.
\end{align}
With the local nature of the DG discrete space $U_h^k$, its basis functions can be chosen such that $M$, $\Sigma_s$ and $\Sigma_a$ are block diagonal. As a result, $\Theta$ can be inverted locally on each element, with an $O(N)$ total cost. 

Next substitute \eqref{eq:schur_step1} into the first row of \eqref{eq:matrix-vector} and utilize $\lgl v^2\rgl = \lgl v^2\rgl_h=\sum_j \omega_j v_j^2$, one obtains
\begin{align}
\mathcal{H}{\boldsymbol{\rho}}^{n+1} = \widetilde{{\bf{b}}}_0^n,
\end{align}
with 
\begin{align}
\mathcal{H}=M+\Dt\Sigma_a - \lgl v^2\rgl \Dt^2 D^+\Theta^{-1}D^-,
\end{align}
and $\widetilde{{\bf{b}}}_0^n$ depends on the solution  on time level $n$. 
For each time step,  we need to invert  $\mathcal{H}$. Based on Lemma \ref{lem:adjoint_relation}, 
$-\mD_h^+$ is the adjoint operator of $\mD_h^-$. This leads to $-D^+=(D^-)^T$, therefore $\mathcal{H}$ is SPD. Indeed $\mathcal{H}$ is a discrete version of $1+\Dt\sigma_a-\lgl v^2\rgl\Dt^2\partial_x\Big((\vareps^2(1+\Dt\sigma_a)+\Dt\sigma_s)^{-1}\partial_x\Big)$, a diffusive operator with the absorption effect.
With the nice property such as being SPD, $\mathcal{H}$ is much easier to invert numerically than the matrix $\mathcal{L}$ in \eqref{eq:matrix-vector}.

For high order IMEX-RK schemes, the Schur complement can be applied similarly. 
On each inner stage, a discrete diffusive operator with the absorption effect needs to be inverted. 
Particularly, in the double Butcher tableaus of either ARS(2,2,2) or ARS(4,4,3), the diagonal entries of the matrix from the implicit part are exactly the same. Hence, for each time step, exactly the same matrix is inverted (numerically) for all inner stages. 

\begin{rem} With a similar derivation, one can show  that the IMEX1-LDG scheme in \cite{peng2019stability, peng2020analysis} needs to invert  $\widetilde{\mathcal{H}}=M+\Dt\Sigma_a - \omega\lgl v^2\rgl \Dt D^+\Sigma_s^{-1}D^-$ for each step, where $\omega\rightarrow1$ as $\varepsilon\rightarrow 0$. With both $\Sigma_s$ and $\Theta$ being block diagonal, the  computational cost of the IMEX1-DG-S scheme  is comparable with that of IMEX1-LDG schemes in \cite{peng2019stability, peng2020analysis}. The same comment also goes to higher order methods in both families. Note that as $\vareps\rightarrow 0$,   $\widetilde{\mathcal{H}}$ and ${\mathcal {H}}$ approach the same operator.
\end{rem}
\begin{rem}
For the discretization of  the velocity space, one can alternatively apply the $P_N$ method \cite{pomraning1973equations}, which expands $f$ in terms of  orthogonal polynomials in the velocity variable $v$. 
If applying $P_N$ method as well as our spatial and temporal  discretizations,  based on the Schur complement, one still just needs to invert one discrete diffusive operator for one inner RK stage. The key to verify this is to use the commuting  property $(\lgl \testG(v,x)\rgl,\testR(x) )=\lgl (\testG(v,x),\testR(x) )\rgl$. The schemes with the $P_N$ method in velocity  are not explored here.
\end{rem}

%%%%%%%%%%%%%%%%%%%%%%%%%%%%%%%%%%%%%%%%%%%%%%%%%%%%%%%%%%%%%%%%%%  
% General linear kinetic transport system
%%%%%%%%%%%%%%%%%%%%%%%%%%%%%%%%%%%%%%%%%%%%%%%%%%%%%%%%%%%%%%%%%%  
\subsection{More general linear kinetic transport equations}
\label{sec:g:model}

Though not considered in this paper, we want to point out that our temporal strategy works for more general linear kinetic transport equations, for example, the case when the scattering effect is anisotropic in the velocity space.  Consider a more general linear kinetic transport equation: 
\begin{align}
\vareps \partial_t f+v\partial_x f = \frac{1}{\vareps}\mathcal{Q}f,
\label{eq:general-model}
\end{align}
where $\mathcal{Q}$ is a collision operator. 
As in \cite{lemou2008new}, we assume that there exists an equilibrium state $E$ independent of $t$ and $x$ satisfying $E\geq0$, $\lgl E\rgl=1$ and $\lgl vE\rgl=0$. The collision operator $\mathcal{Q}$ satisfies the following:
\begin{enumerate}
\item $\mathcal{Q}$ is a linear operator in the velocity space,  independent of $f$, and local in $x$;
\item $\mathcal{Q}$ is non-positive self-adjoint;
\item $\textrm{Null}(\mathcal{Q})=\textrm{Span}\{E\}=\{f=\rho E=\lgl f\rgl E\}$.
\end{enumerate} 

Following \cite{lemou2008new}, we apply micro-macro decomposition.
% to \eqref{eq:general-model}. 
Define an orthogonal projection $\Pi: L^2(\Omega_v;E^{-1}dv)\rightarrow \textrm{Null}(\mathcal{Q})$, that is $\Pi f=\rho E$. Rewrite $f$ as $f=\Pi f+({\bf{I}}-\Pi)f=\rho E+\vareps g$. The micro-macro decomposed system of \eqref{eq:general-model} is
\begin{subequations}
\label{eq:micro-macro-general}
\begin{align}
&\partial_t \rho+\partial_x\lgl vg\rgl =0,
\label{eq:micro-macro-general1}
\\
&\partial_t g+\frac{1}{\vareps}({\bf{I} }-\Pi)(v\partial_xg)+\frac{1}{\vareps^2}vE\partial_x\rho=\frac{1}{\vareps^2}\mathcal{Q}g.
\label{eq:micro-macro-general2}
\end{align}
\end{subequations}
Under the assumption on $\mathcal {Q}$, as $\vareps\rightarrow 0$ we formally obtain the diffusion limit:
\begin{subequations}
\label{eq:general-diffusion-limit}
\begin{align}
&g = \mathcal{Q}^{-1}(vE)\partial_x \rho,
\label{eq:general-diffusion-limit1}\\
&\partial_t\rho + \partial_x (\lgl v \mathcal{Q}^{-1}(vE)\rgl\partial_x\rho)=0.
\label{eq:general-diffusion-limit2}
\end{align}
\end{subequations}
Apply the same time discretization as \eqref{eq:time}, we have
\begin{subequations}
\label{eq:time-general}
\begin{align}
\label{eq:time-rho-general}
&\frac{\rho^{n+1}-\rho^n}{\Dt} + \partial_x\lgl vg^{n+1}\rgl = 0,\\
\label{eq:time-g-general}
&\frac{g^{n+1}-g^n}{\Dt} + \frac{1}{\vareps}({\bf{I}}-\Pi)(v\partial_x g^n)+\frac{1}{\vareps^2}vE\partial_x \rho^{n+1} = \frac{1}{\vareps^2}\mathcal{Q} g^{n+1}.
\end{align}
\end{subequations}
The spatial derivatives can be further replaced by discrete derivatives as in 
Section \ref{sec:space}.
 On the solver level,  we apply  Schur complement as below. 
At each time step, we first express 
\begin{equation}
g^{n+1}= (\vareps^2 -\Dt\mathcal{Q})^{-1}\left(-vE\Dt\partial_x\rho^{n+1}+b^n\right),
\label{eq:g:S:a}
\end{equation}
 where $b^n$ is determined by the data on time level $n$.  
 We then  substitute \eqref{eq:g:S:a} into \eqref{eq:time-rho-general}, and obtain $\big(1-\Dt\partial_{x}(\kappa_{\Dt}(x)\partial_x)\big)\rho^{n+1}=\widetilde{b}^n$, where 
$\kappa_{\Dt}(x)=\Dt\lgl v(\vareps^2 -\Dt\mathcal{Q})^{-1}(vE)\rgl$ and $\widetilde{b}^n$ depends on the solution  on time level $n$. To obtain $\rho^{n+1}$,  a  discrete  diffusion
 operator is inverted.  If $\Dt$ is fixed in time, $\kappa_{\Dt}(x)$ can be pre-computed locally.

\section{AP property}\label{sec:AP}
We formally analyze the asymptotic behavior of the proposed schemes in \eqref{eq:full_high} and show they are AP.  Assume the initial data $\rho(x,0)$ and $g(x,v,0)$ are uniformly bounded with respect to $\vareps$. Then, the initialization through $L^2$ projection leads to uniform boundedness of  $\rho_h^0$ and $g_h^0$. Using mathematical induction and boundedness of the discrete operator $\mD_h^{\pm}$ and $\mD_h^{up}$, we formally obtain that as $\vareps\rightarrow 0$, $\forall \testR, \testG\in U_h^k$, $\forall n\geq 0$,
\begin{subequations}
\label{eq:limit_scheme}
\begin{align}
&(\rho_h^{n,(i)},\testR)= (\rho_h^n,\testR) - \Dt \sum_{j=1}^i a_{ij}\left((\mD_h^+\lgl vg^{n,(j)}_h\rgl_h,\testR)+(\sigma_a\rho_h^{n,(j)},\testR)\right),\;\;i=1,\dots,s,\\
&(\sigma_s g_{h,l}^{n,(i)},\testG) = -v_l(\mD_h^- \rho_h^{n,(i)},\testG),\;\; l=1,\dots,N_v, \;\;i=1,\dots,s,\label{eq:limit_scheme_g}\\
&\rho_h^{n+1} = \rho_h^{n,(s)},\quad\;\; g_{h,l}^{n+1} = g_{h,l}^{n,(s)},\;\; l=1,\dots,N_v.
\end{align}
\end{subequations}

Multiply $\omega_l v_l$ on both sides of \eqref{eq:limit_scheme_g} and sum up with respect to $l$, we get 
\begin{align}
(\sigma_s\lgl vg^{n,(i)}_{h}\rgl_h,\testR) &= \sum_{l=1}^{N_v}\omega_lv_l(\sigma_s g_{h,l}^{n,(i)},\testR)
=-\sum_{l=1}^{N_v}(\omega_l v_l^2 \mD_h^- \rho_h^{n,(i)},\testG)
\notag\\
& = -\lgl v^2\rgl_h( \mD_h^- \rho_h^{n,(i)},\testG)= -\lgl v^2\rgl( \mD_h^- \rho_h^{n,(i)},\testG).\label{eq:limit_scheme_quadrature}
\end{align}
Substitute \eqref{eq:limit_scheme_quadrature} into \eqref{eq:limit_scheme}, then the limiting scheme can be rewritten as:  $\forall \testR, \testG\in U_h^k$, $\forall n\geq0$
\begin{subequations}
\label{eq:limit_scheme_post}
\begin{align}
&(\rho_h^{n,(i)},\testR)= (\rho_h^n,\testR) - \Dt \sum_{j=1}^i a_{ij}\left((\mD_h^+\lgl vg^{n,(j)}_h\rgl_h,\testR)+(\sigma_a\rho_h^{n,(j)},\testR)\right),\;\;i=1,\dots,s,\label{eq:limit_scheme_post1}\\
&( \sigma_s\lgl vg^{n,(i)}_{h}\rgl_h,\testG)  =  -\lgl v^2\rgl( \mD_h^- \rho_h^{n,(i)},\testG),\;\;i=1,\dots,s,\label{eq:limit_scheme_post2}\\
&(\sigma_s g_{h}^{n,(i)},\testG) = -v_l(\mD_h^- \rho_h^{n,(i)},\testG),\;\; l=1,\dots,N_v, \;\;i=1,\dots,s,\label{eq:limit_scheme_post3}\\
&\rho_h^{n+1} = \rho_h^{n,(s)},\quad g_{h,l}^{n+1} = g_{h,l}^{n,(s)},\;\; l=1,\dots,N_v.\label{eq:limit_scheme_post4}
\end{align}
\end{subequations}
In \eqref{eq:limit_scheme_post1} and\eqref{eq:limit_scheme_post2}, $\lgl vg_h^{n,(i)}\rgl_h$ actually provides an approximation to
 $\lgl v^2\rgl \sigma_s^{-1}(x)\partial_x \rho$.
Hence, \eqref{eq:limit_scheme_post1}, \eqref{eq:limit_scheme_post2} and \eqref{eq:limit_scheme_post4} define a high order  implicit RK LDG  
scheme solving the diffusion limit \eqref{eq:diffusion_limit}, whose time discretization is determined by the implicit part of the IMEX-RK scheme.
Moreover, in \eqref{eq:limit_scheme_post3}, the local equilibrium 
$\sigma_s g=-v\partial_x\rho$ is preserved on the discrete level at each RK inner stage. Therefore, we formally verify the AP property of the proposed schemes.

\begin{rem}
 Though our analysis above does not require the initial data $f(x,v,0)=\rho(x,0)+\vareps g(x,v,0)$ to be close to the local equilibrium $\sigma_s g(x,v,0)=-v\partial_x \rho(x,0)$, it does not cover the worst scenario $g(x,v,0)=\frac{1}{\vareps}\left(f(x,v,0)-\rho(x,0) \right)= O(\frac{1}{\vareps})$. For this case, with formal analysis similar to \cite{peng2019stability}, one can show that the limiting scheme is an $O(\Dt)$ perturbation to \eqref{eq:limit_scheme}, regardless the temporal accuracy. Hence, the limiting scheme is a first order in time scheme solving the diffusion limit. This implies that our schemes stay  AP, and indeed they are  {\it{strong AP}} \cite{jin2010asymptotic}.  When the temporal accuracy is higher than one, with the $O(\Dt)$ perturbation, our AP schemes  suffer from order reduction for the case of $g(x,v,0)=O(\frac{1}{\vareps})$. To recover the full $p$-th order temporal accuracy as designed, one can adopt the strategy proposed in \cite{peng2019stability} and alter the first time step size into $\Dt_1=\Dt^p$,  where $\Dt$ is the time step size for later steps, predicted by stability analysis.
\end{rem}
\section{Stability}
\label{sec:stability}
In this section, numerical stability analysis will be carried out. An energy approach will be applied to the first order IMEX1-DG1-S scheme in Section \ref{sec:energy}, and Fourier analysis will then be applied to the first to the third order schemes, namely IMEX$k$-DG$k$-S scheme, $k=1, 2, 3$ in Section \ref{sec:fourier}. 
The analysis shows that our schemes are uniformly stable with respect to $\vareps$ and unconditionally stable in the diffusive regime.   Throughout this Section, we assume periodic boundary conditions in $x$, and $\sigma_s(x)\geq \sigma_m>0, \;\forall x\in \Omega_x$.

%%%%%%%%%%%%%%%%%%%%%%%%%%%%%%%%%%%%%%%%%%%%%%%%%
% Energy analysis
%%%%%%%%%%%%%%%%%%%%%%%%%%%%%%%%%%%%%%%%%%%%%%%%% 
\subsection{Energy analysis for IMEX1-DG1-S scheme}
\label{sec:energy}

In this section, 
we will present an energy approach for stability analysis of
the IMEX1-DG1-S scheme \eqref{eq:full_first}.
% for bounded velocity domain $\Omega_v$.
 The mesh is assumed to be  regular, namely,  there exists $\delta$ such that  $h_i/h\geq\delta,\; \forall i$, during the mesh refinement. We use $||\cdot||$ to denote the standard $L^2$ norm for $L^2(\Omega_x)$, and let  $|||g|||=\sqrt{\lgl(g,g)\rgl_h}$ and $|||g|||_s=\sqrt{\lgl (\sigma_s g,g)\rgl_h}$. For stability, we first define  a  $\mu$-dependent discrete energy $E_{\mu,h}$ with $\mu\in[0,1]$ as a parameter.  To guarantee $E_{\mu,h}$ non-increasing, we obtain $\mu$-dependent stability conditions. The results are further optimized with respect to $\mu$. The energy type stability analysis for higher order in time schemes is left to our future investigation.

\begin{defn}
\label{def:stab}
Given $\mu\in[0,1]$, we define a discrete energy for our schemes,
\begin{align}
\label{eq:mu:ene:def}
E_{\mu,h}^n=||\rho_h^n||^2+\vareps^2 |||g_h^{n}|||^2+(1-\mu)\Dt|||g_h^n|||^2_s.
\end{align} 
The scheme is $\mu$-stable, if $E_{\mu,h}^{n+1}\leq E_{\mu,h}^n,\;\forall n\geq0$. If there exists $\mu\in[0,1]$ such that the scheme is $\mu$-stable, then the scheme is stable. If the scheme is stable (resp. $\mu$-stable) for arbitrary 
$\Dt>0$, then it is unconditionally stable (resp. $\mu$-stable).
\end{defn}

\begin{rem}
The $\mu$-dependent discrete energy $E_{\mu,h}^n$ in \eqref{eq:mu:ene:def} is quite different from that in  \cite{jang2014analysis, peng2020analysis}. Particularly,  the discrete energy in \cite{jang2014analysis, peng2020analysis} involves $\rho_h$ and $g_h$ from different time levels.
\end{rem}

%%%%%%%%%%%%%%%%%%%%%%
% $\mu$-stability for IMEX1-DG1-S
%%%%%%%%%%%%%%%%%%%%%%
\begin{thm}[{\bf$\mu$-stability}]
\label{thm:mu_stable}
Given $\mu\in[0,1]$, the IMEX1-DG1-S scheme is unconditionally $\mu$-stable, if 
\begin{align}
\frac{\vareps}{\sigma_m h} \leq \lambda_0(\mu):=\frac{(1-\mu)\delta}{2||v||_\infty}.
\label{eq:uncond-stable-region}
\end{align}
Otherwise, it is $\mu$-stable under the time step condition
\begin{align}
\Dt\leq\tau_0(\mu):=\frac{2\vareps^2 h}{2\vareps ||v||_\infty/\delta-(1-\mu)\sigma_mh}.\label{eq:CFL}
\end{align}
Here $\delta$ is the mesh regularity parameter.
\end{thm}

%%%%%%%%%%%%%%%%%%%%%%%%%%
% proof for $\mu$-stability for IMEX1-DG1 -S
%%%%%%%%%%%%%%%%%%%%%%%%%%

\begin{proof}
%[{\bf Proof of Theorem  \ref{thm:mu_stable}. }]
%%%%%%%%%%%%%%

Take $\testR=\rho_h^{n+1}$ in \eqref{eq:full_first_rho}, and take $\testG=\vareps^2 g_h^{n+1}$ in \eqref{eq:full_first_g}. %substitute integral operator $\lgl \cdot\rgl$ with its discrete counterpart $\lgl\cdot \rgl_h$. 
Sum up \eqref{eq:full_first_g} for different collocation points $v_l$ with the corresponding weight $\omega_l$, we have 
\begin{subequations}
\begin{align}
&\frac{1}{\Dt}(\rho_h^{n+1}-\rho_h^n,\rho_h^{n+1})+l_h(\lgl vg_h^{n+1}\rgl_h,\rho_h^{n+1})\notag\\
=&\frac{1}{2\Dt}(||\rho_h^{n+1}||^2-||\rho_h^n||^2+||\rho_h^{n+1}-\rho_h^n||^2)+l_h(\lgl vg_h^{n+1}\rgl_h,\rho_h^{n+1}) 
=-(\sigma_a\rho_h^{n+1},\rho_h^{n+1}),
\label{eq:proof1-1}\\
&\frac{\vareps^2}{\Dt}\lgl(g_h^{n+1}-g_h^n,g_h^{n+1})\rgl_h+\vareps \lgl b_{h,v}(g_h^n,g_h^{n+1})\rgl_h-\lgl vd_h(\rho_h^{n+1}, g_h^{n+1})\rgl_h\notag\\
=&\frac{\vareps^2}{2\Dt}(|||g_h^{n+1}|||^2-|||g_h^n|||^2+|||g_h^{n+1}-g_h^n|||^2)+\vareps\lgl b_{h,v}(g_h^n,g_h^{n+1})\rgl_h
-d_h(\rho_h^{n+1}, \lgl vg_h^{n+1}\rgl_h)\notag\\
=&  -|||g_h^{n+1}|||_s^2-\varepsilon^2\lgl (\sigma_a g_h^{n+1},g_h^{n+1})\rgl_h.
\label{eq:proof1-2}
\end{align}
\end{subequations}

Summing up \eqref{eq:proof1-1} and \eqref{eq:proof1-2}, with Lemma \ref{lem:adjoint_relation}, we obtain
\begin{align}
&\frac{1}{2\Dt}( ||\rho_h^{n+1}||^2+\vareps^2|||g_h^{n+1}|||^2-||\rho_h^n||^2-\vareps^2|||g_h^n|||^2)
+\frac{1}{2\Dt}(||\rho_h^{n+1}-\rho_h^n||^2+\vareps^2|||g_h^{n+1}-g_h^n|||^2)
\notag\\
&+(\sigma_a\rho_h^{n+1},\rho_h^{n+1})+\varepsilon^2\lgl (\sigma_a g_h^{n+1},g_h^{n+1})\rgl_h+|||g_h^{n+1}|||_s^2\notag \\
&+\vareps
\lgl b_{h,v}(g_h^n-g_h^{n+1},g_h^{n+1})\rgl_h+\vareps\lgl b_{h,v}(g_h^{n+1},g_h^{n+1})\rgl_h
=0.
\label{eq:proof1}
\end{align}
Similar to \cite{peng2020analysis}, we split $|||g_h^{n+1}|||_s^2$ into 
\begin{align}
\label{eq:proof_split}
|||g_h^{n+1}|||_s^2
=\mu|||g_h^{n+1}|||_s^2+ (1-\mu)\Big(\frac{1}{2}|||g_h^{n+1}|||_s^2-\frac{1}{2}|||g_h^{n}|||_s^2
+\frac{1}{4}|||g_h^{n+1}-g_h^{n} |||_s^2 + \frac{1}{4}|||g_h^{n+1}+g_h^{n} |||_s^2\Big).
 \end{align}
With the piecewise constant in the discrete space, 
$\partial_x g_h^{n+1}=0$, and $|u(x_{i\pm\frac{1}{2}})|= h_i^{-1/2}||u||_{L^2(I_i)}$, $\forall u\in P^0(I_i)$. 
Following similar steps as in \cite{jang2014analysis} (such as its equation (3.22) and (3.24)), using the property of the solution in \eqref{eq:g:prop} and Young's inequality,
we obtain
\begin{align}
&\lgl \bh(g_h^{n+1},g_h^{n+1})\rgl_h =\left\lgl\sum_i\frac{|v|}{2}[g_h^{n+1}]_\iL^2 \right\rgl_h,\label{eq:upwind}\\
&\left| \lgl \bh(g_h^{n+1}-g_h^{n} ,g_h^{n+1}) \rgl_h \right|\leq \eta |||g_h^{n+1}-g_h^{n}|||^2+\frac{1 }{\eta\delta h}\sum_i \left\lgl (\frac{|v|}{2}[g_h^n]_\iL)^2\right \rgl_h.
\label{eq:upwind_jump}
\end {align}
Here, $\eta$ is a positive parameter, which will be determined later.

Substitute \eqref{eq:proof_split}-\eqref{eq:upwind_jump} into \eqref{eq:proof1}, and utilize $\sigma_a\geq 0$, we get
\begin{align}
&\frac{1}{2\Dt}( E_{\mu,h}^{n+1}-E_{\mu,h}^n)+\frac{1}{2\Dt}||\rho_h^{n+1}-\rho_h^n||^2
+(\frac{\vareps^2}{2\Dt}+\frac{1-\mu}{4}\sigma_m - \vareps\eta)|||g_h^{n+1}-g_h^n|||^2\notag\\
&+\frac{1-\mu}{4}|||g_h^{n+1}+g_h^n|||_s^2+\mu|||g_h|||_s^2+\vareps(1-\frac{||v||_\infty} {2\eta \delta h})\lgl \sum_{i}\lgl\frac{|v|}{2}[g_h^{n+1}]^2\rgl_h
\leq 0.
\end{align}

In order to guarantee $E_{\mu,h}^{n+1}\leq E_{\mu,h}^n$, we require 
\begin{subequations}
\begin{align}
&\frac{\vareps^2}{2\Dt}+\frac{1-\mu}{4}\sigma_m - \varepsilon\eta\geq 0,\label{eq:proof_l3.a}\\
&1-\frac{||v||_\infty}{2\eta\delta h}\geq 0.\label{eq:proof_l3}
\end{align}
\end{subequations}
We choose $\eta=\frac{\vareps}{2\Dt}+\frac{1-\mu}{4\vareps}\sigma_m$, so \eqref{eq:proof_l3.a} holds, and 
the inequality in \eqref{eq:proof_l3} becomes
\begin{align}
\frac{\vareps}{\Dt}\geq\frac{2\vareps||v||_\infty/\delta-(1-\mu)\sigma_mh }{2\vareps h}.\label{eq:proof_final}
\end{align}
When $\frac{\vareps}{\sigma_m h}\leq\frac{(1-\mu)\delta}{2||v||_\infty}$, \eqref{eq:proof_final} holds for arbitrary $\Dt>0$, hence the method is unconditionally stable. Otherwise, we need $\Dt$ to satisfy \eqref{eq:CFL} to have the conditional $\mu$-stability.
%%%%%%%%%%%%%%
\end{proof}

Next we will optimize the results in Theorem \ref{thm:mu_stable} in $\mu$ to maximize the unconditionally stable region and also the allowable time step size when the scheme
 is conditionally stable.  
%%%%%%%%%%%%%%%%%
% stability for IMEX1-DG1-S
%%%%%%%%%%%%%%%%
\begin{thm}[{\bf stability}]
%\textbf{(stability)}
\label{thm:stable}
The IMEX1-DG1-S scheme is unconditionally stable, if 
\begin{align}
\frac{\vareps}{\sigma_m h} \leq \frac{\delta}{2||v||_\infty}.
\end{align}
Otherwise, it is stable under the time step condition
\begin{align}
\Dt\leq\frac{2\vareps^2 h}{2\vareps ||v||_\infty/\delta-\sigma_mh}.
\label{eq:CFL}
\end{align}
\end{thm}
%%%%%%%%%%%%%%%%%%%%%%%%%%%%%%%%%%%%%%%%%%%%%%%%%
% Optimization
%%%%%%%%%%%%%%%%%%%%%%%%%%%%%%%%%%%%%%%%%%%%%%%%%
\begin{proof}
%[\bf Proof of Theorem  \ref{thm:stable}. ]
Based on the definition of $\mu$-stability and stability in Definition \ref{def:stab}, the results in Theorem \ref{thm:mu_stable} further imply that the IMEX1-DG1-S scheme is unconditionally stable if 
\begin{align}
\frac{\varepsilon}{\sigma_m h}\leq\max_{\mu\in[0,1]}{\lambda_0(\mu) }= \max_{\mu\in[0,1]} \left(\frac{(1-\mu)\delta}{2||v||_\infty}\right)=\frac{\delta}{2||v||_\infty}.
\end{align}
When $\frac{\vareps}{\sigma_m h}>\frac{\delta}{2||v||_\infty}$,  the scheme is conditionally stable under the following time step restriction
\begin{align}
\Dt\leq\max_{\mu\in[0,1]}\tau_0(\mu)=\max_{\mu\in[0,1]}\left(\frac{2\vareps^2 h}{2\vareps ||v||_\infty/\delta-(1-\mu)\sigma_mh}\right)
=\frac{2\vareps^2 h}{2\vareps ||v||_\infty/\delta-\sigma_mh}.
\end{align}
%%%%%%%%%
\end{proof}

%%%%%%%%%%%%%%%%%%%%
% remark: stability for sigma_s = 0
%%%%%%%%%%%%%%%%%%%%
\begin{rem}
For a multiscale problem, it is possible to have subregions with $\sigma_s=0$ where the problem is purely transport. In this case,  $\sigma_m=0$, and 
our proofs above still hold. Specifically, the IMEX1-DG1-S scheme is always conditionally stable under the time step condition 
$\Dt\leq\frac{2\vareps^2 h\delta}{2\vareps ||v||_\infty}=\frac{\vareps h\delta}{||v||_\infty}$, and the unconditional stability is not expected.
\end{rem}

%%%%%%%%%%%%%%%%%%%%%%%%%%%%%%%%%%%%%%%%%%%%%%%%%
% Fourier analysis
%%%%%%%%%%%%%%%%%%%%%%%%%%%%%%%%%%%%%%%%%%%%%%%%% 
\subsection{Fourier Analysis for IMEX$k$-DG$k$-S scheme, $k=1, 2, 3$}\label{sec:fourier}

In this section, Fourier analysis is performed for the IMEX$k$-DG$k$-S scheme, $k=1, 2, 3$, when the schemes are applied to the one-group transport equation in slab geometry with $\Omega_v=[-1,1]$. Related, 
$\lgl f\rgl= 
%\int_{\Omega_v}F d\nu=
 \frac{1}{2}\int_{-1}^1 f(v)dv$, with $dv$ the standard  Lebesgue measure.
$16$ Gaussian quadrature points together with the respective quadrature weights are applied to discretize the velocity space. 
As typical for Fourier analysis, it is assumed that the mesh is uniform and $\sigma_s(x)=\sigma_m>0, \; \forall x\in\Omega_x$. Motivated by that the  stability result for the IMEX1-DG1-S scheme in Section \ref{sec:energy} does not depend on $\sigma_a$, we further assume $\sigma_a=0$.
% \pzc{We also note that only $\sigma_m$ appeared in the stability condition for the IMEX1-DG1-S scheme.} 
 %$\int_{\Omega_v}d\nu = \frac{1}{2}\int_{\Omega_v}dv\frac{1}{2}\int_{\Omega_v}dv$. $16$ Gaussian quadrature points is applied to the discretization of the velocity space. 
 %Under this setting, Fourier (von Neumann) analysis is carried out.
Similar to \cite{peng2019stability}, %to perform the Fourier analysis, 
we first identify an invariant scaling structure of the amplification matrix. Then, by numerically solving an eigenvalue problem, we obtain the stability condition for IMEX$k$-DG$k$-S scheme, $k=1,2,3$.

\medskip
\noindent {\textbf{Setup of the Fourier analysis:} } 
We will use the IMEX1-DG$k$-S scheme as an example to demonstrate the setup. 
On the element $I_m$, the numerical solutions can be expanded as
\begin{align}
\rho_h^n(x)=\sum_{l=0}^{k-1}\rho^n_{ml}\phi_{l}^m(x),\quad g_{h,j}^n(x)=\sum_{l=0}^{k-1}g^n_{j,ml}\phi^{m}_l(x),\;\forall x\in I_m
\end{align}
where $\phi_l^{m}(x)=\phi_l\left( \frac{x-x_m}{h_m/2}\right)$, $\phi_l(x)$ is the $l$-th order  Legendre polynomial on $[-1, 1]$. Let ${\boldsymbol{\rho} }_m^n=(\rho_{m0}^n,\dots,\rho_{m\;k-1}^n)^T$ and ${\bf{g} }_{jm}^n=(g^n_{j,m0},\dots,g^n_{j,m\;k-1})^T$.

Take the Fourier anstaz ${\boldsymbol{\rho}}_m^n=\exp(\mathcal{I}\kappa x_m){\boldsymbol{\widehat{\rho}}}^n$ and ${\bf{g}}_{jm}^n=\exp(\mathcal{I}\kappa x_m){\bf{\widehat{g}}}_j^n$ (with $\mathcal I^2=-1$), and 
%where $x_m=mh$.
plug  them into the IMEX1-DGk-S scheme, we obtain
\begin{align}
\label{eq:fourier_matrix_vector}
&\underbrace{
\left(
\begin{matrix}
h\widehat{M} 		  	&\Dt \omega_1 v_1 \widehat{D}^+ 		& \Dt \omega_2 v_2 \widehat{D}^+ 		&\dots 	& \Dt\omega_{N_v}v_{N_v}\widehat{D}^+\\
v_1 \Dt \widehat{D}^- 	& h(\vareps^2+\sigma_m\Dt) \widehat{M} 	& 0 						&\dots 	& 0 \\
v_2 \Dt \widehat{D}^- 	& 0						&  h(\vareps^2+\sigma_m\Dt) \widehat{M}	&\dots 	& 0 \\
\vdots	  	& \vdots					& \vdots 					&\ddots	&\vdots\\
v_{N_v}\Dt \widehat{D}^-	& 0						& 0 						&\dots  	&  h(\vareps^2+\sigma_m\Dt) \widehat{M}
\end{matrix}
\right) }_{G_L}
\left(
\begin{matrix}
{\boldsymbol{\widehat{\rho}} }^{n+1} \\
{\bf{\widehat{g} }}_1^{n+1}\\
{\bf{\widehat{g} }}_2^{n+1}\\
\vdots\\
{\bf{\widehat{g}}}_{N_v}^{n+1}
\end{matrix}
\right)
\notag\\
=&\underbrace{\left(
\begin{matrix}
h\widehat{M} 		  	&0 		& 0	&\dots 	& 0\\
0	& \vareps^2 h\widehat{M}+\vareps\Dt \widehat{U}_1 	& 0 						&\dots 	& 0 \\
0	& 0						& \vareps^2 h\widehat{M}+\vareps\Dt \widehat{U}_2 &\dots 	& 0 \\
\vdots	  	& \vdots					& \vdots 					&\ddots	&\vdots\\
0	& 0						& 0 						&\dots  	&  \vareps^2 h\widehat{M}+\vareps\Dt \widehat{U}_{N_v} 
\end{matrix}
\right)}_{G_R}
\left(
\begin{matrix}
{\boldsymbol{\widehat{\rho}} }^{n} \\
{\bf{\widehat{g} }}_1^{n}\\
{\bf{\widehat{g} }}_2^{n}\\
\vdots\\
{\bf{\widehat{g}}}_{N_v}^{n}
\end{matrix}
\right).
\end{align}
Here, $\widehat{M}$, $\widehat{D}^{+}$, $\widehat{D}^-$ and $\widehat{U}$ are $k\times k$ matrices, and they are defined as follows.
\begin{subequations}
\begin{align}
&(\widehat{M})_{ij} = \frac{1}{2}\int_{-1}^{1}\phi_{i}(x)\phi_j(x) dx,\\
&(\widehat{D}^{-}(\xi)\;)_{ij}=- \int_{-1}^{1} \phi_j(x)\partial_x \phi_i(x)dx+ \phi_j(1) \phi_i(1)
-\exp({-\mathcal{I} \xi})\phi_j(1) \phi_i(-1),\\
&(\widehat{D}^{+}(\xi)\;)_{ij}=- \int_{-1}^{1} \phi_j(x)\partial_x \phi_i(x)dx+\exp({\mathcal{I} \xi}) \phi_j(-1) \phi_i(1)
-\phi_j(-1) \phi_i(-1),\\
%&(\widehat{U}_k(\xi)\;)_{ij}=(\widetilde{U}_k(\xi)\;)_{ij}-\sum_{k=1}^{N_v}\omega_k(\widetilde{U}_k(\xi)\;)_{ij},
&(\widehat{U}_{l}(\xi)\;)_{ij}=
\begin{cases}
v_{l}(\widehat{D}^{-}(\xi)\;)_{ij} -\sum_{l'=1}^{N_v}\omega_{l'} v_{l'}\big(\mathbbm{1}_{\{v_{l'}\geq0\}}(v_{l'})(\widehat{D}^{-}(\xi)\;)_{ij}+\mathbbm{1}_{\{v_{l'}<0\}}(v_{l'})(\widehat{D}^{+}(\xi)\;)_{ij}\big) ,\;\text{if}\;v_l\geq 0,\\
v_{l}(\widehat{D}^{+}(\xi)\;)_{ij}-\sum_{l'=1}^{N_v}\omega_{l'} v_{l'}\big(\mathbbm{1}_{\{v_{l'}\geq0\}}(v_{l'})(\widehat{D}^{-}(\xi)\;)_{ij}+\mathbbm{1}_{\{v_{l'}<0\}}(v_{l'})(\widehat{D}^{+}(\xi)\;)_{ij}\big),\;\text{if}\;v_l<0,
\end{cases}
\end{align}
\end{subequations}
where
$\xi=\kappa h$ is the discrete wave number, and $\mathbbm{1}_{S}(y)%=\begin{cases}
%1,\;y\in S,\\
%0,\;y\not\in S,
%\end{cases}
$
is the indicator function of the set $S$. 
Define block matrices 
\begin{align}
&{\bm{D} }^{-}=\left(v_1\widehat{D}^{-},\dots,v_{N_v}\widehat{D}^{-}\right)^T\in\mathbb{R}^{kN_v\times k},\;
{\bm{D} }^{+}=\left(\omega_1v_1\widehat{D}^{+},\dots,\omega_{N_v}v_{N_v}\widehat{D}^{+}\right)\in\mathbb{R}^{k\times kN_v},\notag\\
&{\bm{M} } = \textrm{diag}(\widehat{M},\dots,\widehat{M})\in\mathbb{R}^{kN_v\times kN_v},\;
{\bm{U} }= \textrm{diag}(\widehat{U}_1,\dots,\widehat{U}_{N_v})\in\mathbb{R}^{kN_v\times kN_v}.
\end{align}
Then, $G_L$ and $G_R$ can be rewritten as 
\begin{align}
G_L=\left(\begin{matrix}
h\widehat{M} & \Dt{\bm{D}}^+ \\
\Dt{\bm{D}}^-& h(\vareps^2+\sigma_m\Dt){\bm{M}}
\end{matrix}
\right)\;\;
\text{and}\;\;
G_R=\left(\begin{matrix}
h\widehat{M} & 0 \\
0& \vareps^2h{\bm{M}}+\vareps\Dt{\bm{U}}
\end{matrix}
\right),
\end{align}
With the amplification matrix as $\bm{G}^{(1,k)}=\bm{G}^{(1,k)}(\vareps,\sigma_m,h,\Dt; \xi)=G_L^{-1}G_R$ and ${\bf{V} }^{n}=\left({\boldsymbol{\widehat{\rho}} }^{n},{\bf{\widehat{g} }}_1^{n},{\bf{\widehat{g} }}_2^{n},\dots{\bf{\widehat{g}}}_{N_v}^{n}\right)^T$, \eqref{eq:fourier_matrix_vector} becomes ${\bf{V}}^{n+1}=G^{(1,k)}{\bf{V} }^n$. Similarly, the amplification matrix ${\bm{G}}^{(p,k)}$ of the  IMEX$p$-DG$k$-S scheme
can be derived. 
%Here, the dependence of ${\bm{G}}^{(p)}$ and ${\bm{G}}^{(p)}$ on $k$ is suppressed for simplicity.
To study the numerical stability, we will adopt the following principle. 
\begin{itemize}
\item[]
\noindent {\bf  Principle for Numerical Stability \cite{peng2019stability}:} {\em
For any given $\vareps, h, \Dt$, let the eigenvalues of $\bm{G}^{(p,k)}$ be $\lambda_i(\xi)$, $i=1,\dots, 2k$. Our scheme is said to be stable, if for all $\xi\in[-\pi, \pi]$, it satisfies either
	\begin{align}
(*)  \hspace{0.5in} &	\max_{i=1,\dots,2k}\{|\lambda_i(\xi)| \}<1,\quad\mbox{or}
	\label{fourier:stable_condition1}\\
(*) \hspace{0.5in}	&\max_{i=1,\dots,2k}\left\{|\lambda_i(\xi)| \right\}=1 \quad\text{and}\quad\bm{G}^{(p,k)}\quad\text{is diagonalizable}.
	\label{fourier:stable_condition2}
	\end{align}
	}
\end{itemize}
This principle is a necessary condition to guarantee the standard $L^2$ energy non-increasing. Before presenting the stability results, we first show an intrinsic scaling structure of the amplification matrices.

%%%%%%%%%%%%%%%%%%%%%%%%%%%%%%%%%%%%%%%%%%%%%%%%%%%%%%%%%%%%%%%%%%%%%%%%%%%%

%%%%%%%%%%%%%%%%%%%%%%%%%%%%%%%%%%%%%%
% Main findings
%%%%%%%%%%%%%%%%%%%%%%%%%%%%%%%%%%%%%%
\begin{thm}\label{thm:fourier_invariant}
For any given $k\geq 1$  and $p=1, 2, 3$, the amplification matrix  ${\bm{G}^{(p,k)} }(\vareps,\sigma_m,h,\Dt;\xi)$ of the IMEX$p$-DG$k$-S  method  is similar to some matrix $\widehat{ \bm{G}}^{(p,k)}(\frac{\vareps}{\sigma_mh},\frac{\Dt}{\vareps h};\xi)$. In other words, the eigenvalues of ${\bm{G}^{(p,k)} }(\vareps,\sigma_m,h,\Dt;\xi)$ depend on $\vareps, h, \Dt, \sigma_m$ only  in terms of $\frac{\vareps}{\sigma_m h}$ and $\frac{\Dt}{\vareps h}$,  or equivalently, only in terms of $\frac{\vareps}{\sigma_mh}$ and $\frac{\vareps^2}{\sigma_m\Dt}=\frac{\vareps/(\sigma_m h)} {\Dt/(\vareps h)}$.
\end{thm}
\begin{proof}
We start  with $p=1$. 
With  $J_m$ as the $m\times m$ identity matrix, one gets
\begin{align}
&{\bm{G}^{(1,k)} }=G_L^{-1}G_R\notag\\
&=
\left(
\left(
\begin{matrix}
\frac{\sigma_m}{\vareps}J_k & 0 \\
0& \frac{1}{\vareps^2 h}J_{kN_v}
\end{matrix}
\right)
\left(\begin{matrix}
h\widehat{M} & \Dt{\bm{D}}^+ \\
\Dt{\bm{D}}^-& h(\vareps^2+\sigma_m\Dt){\bm{M}}
\end{matrix}
\right)
\right)^{-1}
\left(
\begin{matrix}
\frac{\sigma_m}{\vareps}J_k & 0 \\
0& \frac{1}{\vareps^2 h}J_{kN_v}
\end{matrix}
\right)
\left(\begin{matrix}
h\widehat{M} & 0 \\
0& \vareps^2h{\bm{M}}+\vareps\Dt{\bm{U}}
\end{matrix}
\right)\notag\\
&=
\left(\begin{matrix}
\frac{\sigma_m h}{\vareps}\widehat{M} & \frac{\sigma_m\Dt}{\vareps}{\bm{D}}^+ \\
\frac{\Dt}{\vareps^2 h}{\bm{D}}^-& (1+\frac{\sigma_m\Dt}{\vareps^2 }){\bm{M}}
\end{matrix}
\right)^{-1}
\left(\begin{matrix}
\frac{\sigma_m h}{\vareps}\widehat{M} & 0 \\
0& {\bm{M}}+\frac{\Dt}{\vareps h}{\bm{U}}
\end{matrix}
\right).
\end{align}
Using the relations of
%\begin{subequations}
\begin{align*}
\left(\begin{matrix}
\sigma_mhJ_k & 0\\
0 & J_{kN_v}
\end{matrix}
\right)^{-1}
\left(\begin{matrix}
\frac{\sigma_m h}{\vareps}\widehat{M} & \frac{\sigma_m\Dt}{\vareps}{\bm{D}}^+ \\
\frac{\Dt}{\vareps^2 h}{\bm{D}}^-& (1+\frac{\sigma_m\Dt}{\vareps^2 }){\bm{M}}
\end{matrix}
\right)
\left(\begin{matrix}
\sigma_mhJ_k & 0\\
0 & J_{kN_v}
\end{matrix}
\right)
&=\left(\begin{matrix}
\frac{\sigma_m h}{\vareps}\widehat{M} & \frac{\Dt}{\vareps h}{\bm{D}}^+ \\
\frac{\sigma_m\Dt}{\vareps^2 }{\bm{D}}^-& (1+\frac{\sigma_m\Dt}{\vareps^2 }){\bm{M}}
\end{matrix}
\right),\\
%\text{and}\quad 
\left(\begin{matrix}
\sigma_mhJ_k & 0\\
0 & J_{kN_v}
\end{matrix}
\right)^{-1}
\left(\begin{matrix}
\frac{\sigma_m h}{\vareps}\widehat{M} & 0 \\
0& {\bm{M}}+\frac{\Dt}{\vareps h}{\bm{U}}
\end{matrix}
\right)
\left(\begin{matrix}
\sigma_mhJ_k & 0\\
0 & J_{kN_v}
\end{matrix}
\right)
&=\left(\begin{matrix}
\frac{\sigma_m h}{\vareps}\widehat{M} & 0 \\
0& {\bm{M}}+\frac{\Dt}{\vareps h}{\bm{U}}
\end{matrix}
\right),
\end{align*}
%\end{subequations}
we obtain
\begin{align}
&\left(\begin{matrix}
\sigma_m h J_k & 0\\
0 & J_{kN_v}
\end{matrix}
\right)^{-1}{\bm{G}}^{(1,k)}
\left(\begin{matrix}
\sigma_m h J_k & 0\\
0 & J_{kN_v}
\end{matrix}
\right)
%\notag\\
=\left(\begin{matrix}
\frac{\sigma_m h}{\vareps}\widehat{M} & \frac{\Dt}{\vareps h}{\bm{D}}^+ \\
\frac{\sigma_m\Dt}{\vareps^2 }{\bm{D}}^-& (1+\frac{\sigma_m\Dt}{\vareps^2 }){\bm{M}}
\end{matrix}
\right)^{-1}
\left(\begin{matrix}
\frac{\sigma_m h}{\vareps}\widehat{M} & 0 \\
0& {\bm{M}}+\frac{\Dt}{\vareps h}{\bm{U}}
\end{matrix}
\right)
\notag\\
&=\left(\begin{matrix}
\frac{\sigma_m h}{\vareps}\widehat{M} & \frac{\Dt}{\vareps h}{\bm{D}}^+ \\
\frac{\sigma_m h}{\vareps}\cdot\frac{\Dt}{\vareps h }{\bm{D}}^-& (1+\frac{\sigma_m h}{\vareps}\cdot\frac{\Dt}{\vareps h }){\bm{M}}
\end{matrix}
\right)^{-1}
\left(\begin{matrix}
\frac{\sigma_m h}{\vareps}\widehat{M} & 0 \\
0& {\bm{M}}+\frac{\Dt}{\vareps h}{\bm{U}}
\end{matrix}
\right)
=
\widehat{ \bm{G}}^{(1,k)}(\frac{\vareps}{\sigma_mh},\frac{\Dt}{\vareps h};\xi).
\end{align}
This implies that ${\bm{G}}^{(1,k)}$ is similar to $\widehat{ \bm{G}}^{(1,k)}(\frac{\vareps}{\sigma_mh},\frac{\Dt}{\vareps h};\xi)$.

The proof can be generalized to $p=2,3$ through the mathematical induction. To see this, let
${\bf V}^{n,(0)}={\bf V}^n$, ${\bf{V} }^{n,(l)}=\left({\boldsymbol{\widehat{\rho}} }^{n,(l)},{\bf{\widehat{g} }}_1^{n,(l)},{\bf{\widehat{g} }}_2^{n,(l)},\dots{\bf{\widehat{g}}}_{N_v}^{n,(l)}\right)^T$, $l=1, \dots, s$, we have
$${\bf V}^{n,(l)}=\sum_{q=0}^{l-1} {\bm{G}}^{(p,k)}_{lq}(\vareps,\sigma_m,h,\Dt;\xi) {\bf V}^{n,(q)},\; l=1,\dots,s,\;\;\text{and}\;\; {\bf V}^{n+1}={\bf V}^{n,(s)}.$$
With similar argument as for $p=1$, one can find a ${\bm{\widehat{G}}}^{(p,k)}_{lq}(\frac{\vareps}{\sigma_m h},\frac{\Dt}{\vareps h};\xi)$ such that $\forall\; l=1,\dots, s$,
$$
\left(\begin{matrix}
\sigma_m h J_k & 0\\
0 & J_{kN_v}
\end{matrix}
\right)^{-1}
{\bm{G}}^{(p,k)}_{lq}(\vareps,\sigma_m,h,\Dt;\xi)\left(\begin{matrix}
\sigma_m h J_k & 0\\
0 & J_{kN_v}
\end{matrix}
\right)={\bm{\widehat{G}}}^{(p,k)}_{lq}(\frac{\vareps}{\sigma_m h},\frac{\Dt}{\vareps h};\xi),\;\;q=0,\dots, l-1.
$$
For every ${\bm{G}}^{(p,k)}_{lq}(\vareps,\sigma_m,h,\Dt;\xi)$, exactly the same similar transformation is performed, hence, ${\bm{G}}^{(p,k)}$ is similar to some $\widehat{ \bm{G}}^{(p,k)}(\frac{\vareps}{\sigma_mh},\frac{\Dt}{\vareps h};\xi)$.
%%%%%%%%%%%%%%%%%%%%%%%%%%%%%%%%%%%%
\end{proof}

%%%%%%%%%%%%%%%%%%%%%%%%%%%%%%%%%%%%%%
% Results of numerical Fourier analysis
%%%%%%%%%%%%%%%%%%%%%%%%%%%%%%%%%%%%%%
%\medskip
\noindent  \textbf{Fourier analysis results:} 
Based on Theorem \ref{thm:fourier_invariant} and the principle for numerical stability, the numerical stability results shall only depend on $\frac{\vareps}{\sigma_m h}$ and $\frac{\Dt}{\vareps h}$.  Set $\alpha=\log_{10}(\frac{\vareps}{\sigma_m h})$ and $\beta=\log_{10}(\frac{\Dt}{\vareps h})$.  For the IMEX$k$-DG$k$-S scheme, $k=1, 2, 3$, we numerically compute eigenvalues of the amplification matrix by uniformly sampling the discrete wave number $\xi\in
[-\pi,\pi]$ with spacing $\frac{2\pi}{100}$, $\alpha\in[-5,5]$ and $\beta\in[-5,4]$  with $\frac{1}{20}$ spacing. The stability results are presented in Figure \ref{fig:fourier}, with the white region being stable, and the black region being unstable. The main observations are summarized as follows, with $k=1, 2, 3$:
\begin{enumerate}
\item[1.)] For some $\alpha_k$, the IMEX$k$-DG$k$-S scheme is unconditionally stable when $\alpha<\alpha_k$, i.e. when  $\frac{\vareps}{\sigma_m h}<C_k$. This %numerically
confirms the unconditional stability of the proposed schemes in the diffusive regime.
\item[2.)] In the transport dominant regime with $\vareps/(\sigma_m h)=O(1)$,
%($h\ll \frac{\vareps}{\sigma_m}$), 
the stability region for the IMEX$k$-DG$k$-S is under a straight line $\beta=\beta_k$. In other words, in the transport dominant regime, the scheme is conditionally stable under a standard  hyperbolic type CFL condition
$$
\beta = \log_{10}{\frac{\Dt}{\vareps h}}\leq\beta_k\Leftrightarrow \Dt\leq \widehat{C}_k\vareps h.
$$
\item[3.)] The IMEX$k$-DG$k$-S scheme is stable under the condition  $\beta\leq \mathcal{F}_k(\alpha)$, with some function $\mathcal{F}_k$. Based on this, we can further derive the stability condition $\Dt\leq \widetilde{\mathcal{F}}_k(\vareps,h,\sigma_m)$. The time step condition for the IMEX$k$-DG$k$-S schemes with $k=2,3$ in Section \ref{sec:numerical} is actually obtained through such  procedure.
\item[4.)] The Fourier results for IMEX1-DG$1$-S scheme match well with the energy analysis results.
\end{enumerate}

We want to mention that the stability properties of the IMEX-DG-S schemes are qualitatively  similar to that for the IMEX-LDG schemes in \cite{peng2019stability} with the numerical weight $\omega=\exp(-\frac{\varepsilon}{\sigma_m h})$.

%%%%%%%%%%%%%%%%%%%
\begin{figure}
  \centering 
   \subfigure[IMEX$1$-DG$1$-S]
    { \includegraphics[width=0.31\linewidth]{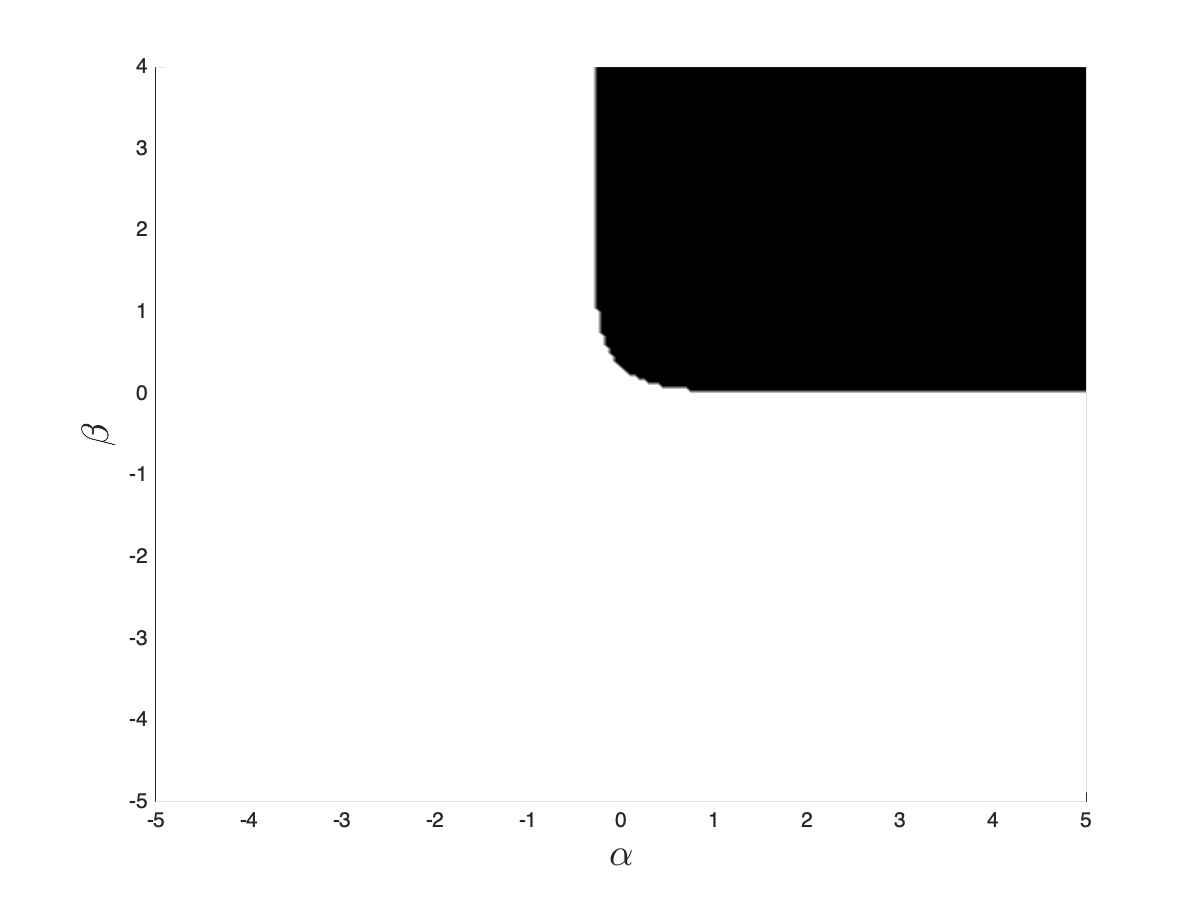} }
    \subfigure[IMEX$2$-DG$2$-S ]
  {    \includegraphics[width=0.31\linewidth]{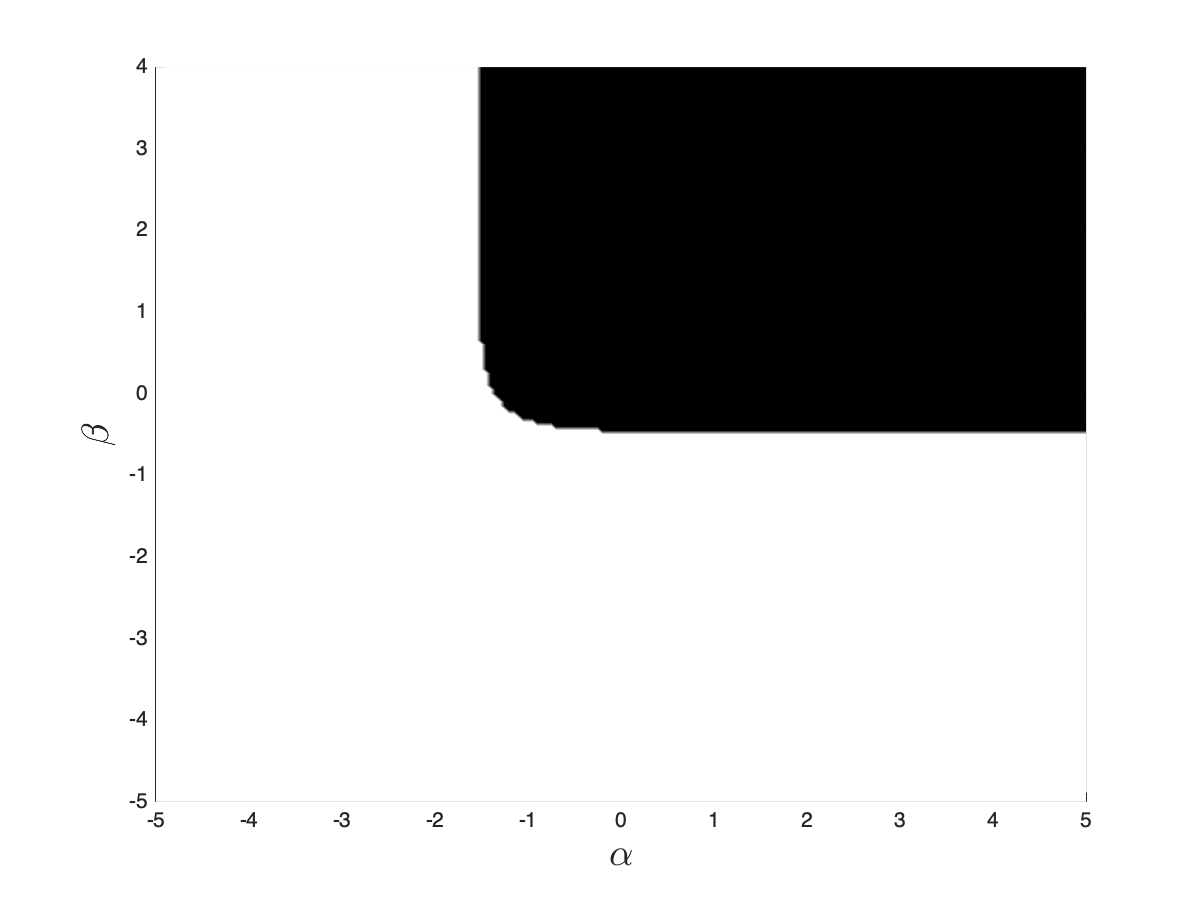}}   
  \subfigure[IMEX$3$-DG$3$-S]
 {      \includegraphics[width=0.31\linewidth]{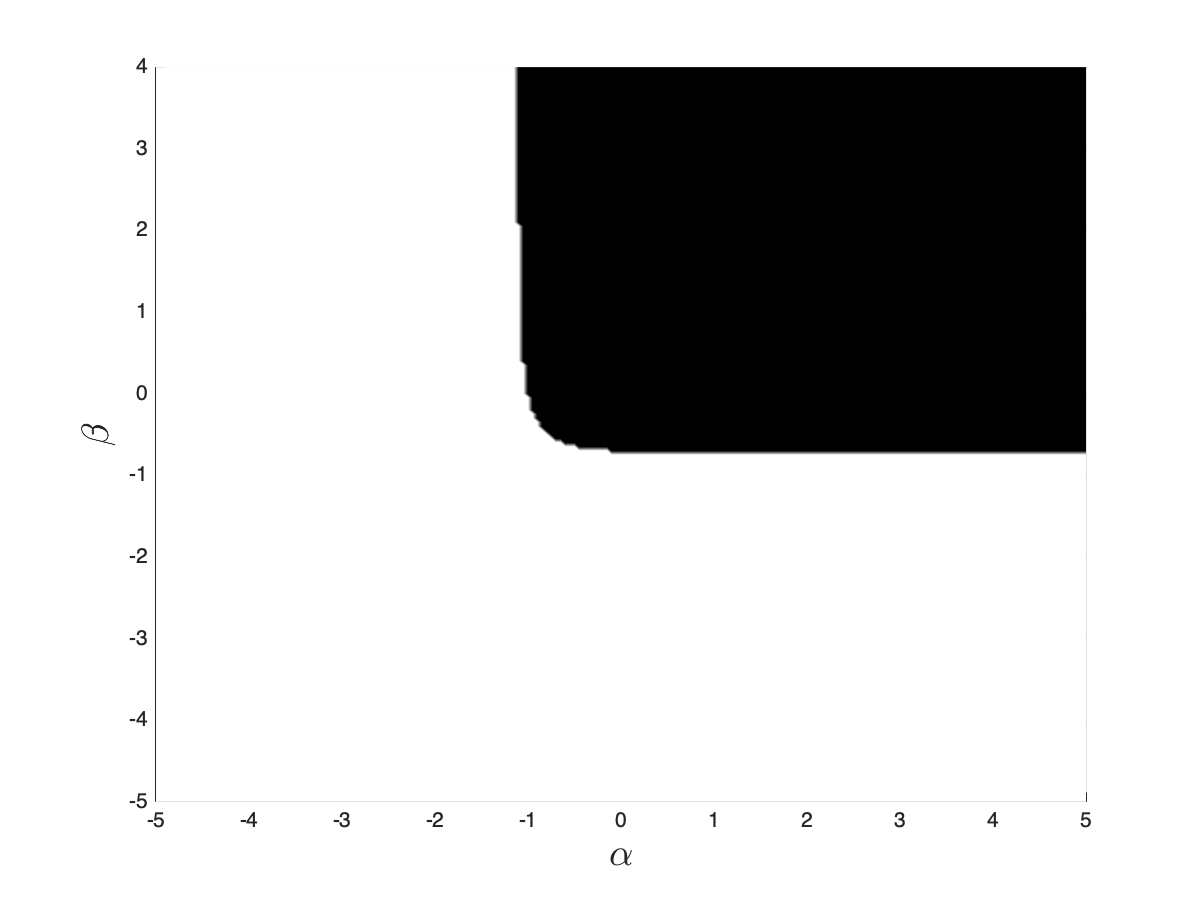}      }   
  \caption{Stability regions of the IMEX$k$-DG$k$-S methods, $k=1, 2, 3$. White: stable; black: unstable. $\alpha=\log_{10}(\frac{\vareps}{\sigma_m h})$ and $\beta=\log_{10}(\frac{\Dt}{\vareps h}).$}
  \label{fig:fourier}
\end{figure}
%%%%%%%%%%%%%%%%%%%

%%%%%%%%%%%%%%%%%%%%%%%%%%%%%%%%%%%%%%%%%%%%%%%%
% Set up
%%%%%%%%%%%%%%%%%%%%%%%%%%%%%%%%%%%%%%%%%%%%%%%%
\section{Numerical tests}
\label{sec:numerical}

In this section, we will demonstrate the performance of the IMEX$k$-DG$k$-S scheme, $k=1, 2, 3$. Two models will be considered. One is the telegraph equation with $\Omega_v=\{-1,1\}$ and $\lgl f \rgl = \frac{1}{2}(f|_{v=1}+f|_{v=-1})$, the other is the one-group transport equation in slab geometry with $\Omega_v=[-1,1]$ and $\lgl f\rgl = \frac{1}{2}\int_{-1}^1 fdv$. For the latter, we discretize the velocity space 
%$\Omega_v=[-1,1]$ 
with $16$ Gaussian quadrature points. The meshes in space are uniform unless otherwise specified.

Based on the energy and Fourier analysis in Section \ref{sec:stability}, for the one-group transport equation in slab geometry, the time step size $\Dt$ for the IMEX$k$-DG$k$-S scheme is chosen as $\Dt_{CFLk}$, 
\begin{subequations}
\label{eqn:CFL}
\begin{align}
\mbox{IMEX1-DG1-S}: \hspace{0.2in} &  \Delta t_{CFL1}=\left\{
\begin{array}{ll} 0.75 h,& \vareps\leq  0.5h, \\
\min(0.75 h ,  \frac{\vareps^2 h }{\vareps-0.5h}),&  \vareps>  0.5h,
\end{array}
\right.\\
\mbox{IMEX2-DG2-S}: \hspace{0.2in} &  \Delta t_{CFL2}=\left\{
\begin{array}{ll} 0.75 h,& \vareps\leq  0.025h,\\
\min(0.75h,\frac{\vareps^2 h/\sqrt{10} }{\vareps-0.025h}), & \vareps\geq 0.025h,
\end{array}
\right.\label{eqn:CFL:2}\\
\mbox{IMEX3-DG3-S}: \hspace{0.2in} &  \Delta t_{CFL3}=\left\{
\begin{array}{ll} 0.75 h,& \vareps\leq  0.05h,\\
\min(0.75 h ,  \frac{0.1\vareps^2 h }{\vareps-0.05h}),&  \vareps> 0.05h.
\end{array}
\right.
\end{align}
\end{subequations}
When the schemes are unconditionally stable, $\Dt=0.75 h$ is used to ensure good resolution. The time step conditions in \eqref{eqn:CFL} also work well for the telegraph equation.
We want to mention that, when boundary conditions are Dirichlet (see next subsection), due to the numerical boundary treatment, a smaller time step size is taken for the second order IMEX2-DG2-S scheme in the diffusive regime.
%, but the time step restrictions for IMEX1-DG1-S and IMEX3-DG3-S are not changed. 
For the linear solver, we apply the Schur complement discussed in Section \ref{sec:matrix-vector} and GMRES \cite{saad1986gmres} solver, which is implemented under the framework of C++ library PETSC \cite{balay2019petsc}. 

%%%%%%%%%%%%%%%%%%%%%%%%%%%%%%%%%%%%%%%%%%%%%%%%
% Numerical boundary 
%%%%%%%%%%%%%%%%%%%%%%%%%%%%%%%%%%%%%%%%%%%%%%%%
\subsection{Numerical boundary condition}\label{sec:numerical_bc}

For some numerical tests, the following inflow (also Dirichlet) boundary conditions are given:
$$ f(x_L,v,t) = f_L(v,t),\;\;v\geq0\quad\text{and}\quad f(x_R,v,t)=f_R(v,t),\;\;v\leq 0.$$ 
These conditions are insufficient to define the boundary conditions for $\rho=\lgl f\rgl$ and $g$, hence numerical treatments are needed. We here adopt a close-loop strategy similar to \cite{jang2015high,peng2019stability}.
For simplicity, we present the strategy using the one-group transport equation in slab geometry with the velocity space being continuous. In implementation, we substitute integrals in $\Omega_v$ with their discrete counterparts. 

\medskip
\noindent\textbf{Main idea:} Our numerical boundary treatment is based on the following idea. At the left boundary, we set
\begin{subequations}
\label{eq:close-loop}
\begin{align}
&\rho_L(t)+\vareps g_L(v, t)=f_L(v,t), \quad v\geq0 \quad\text{(inflow)},
\label{eq:close-loop1}\\
&\rho_L(t)+\vareps g_L(v,t)=\rho_h(x_{\frac{1}{2}}^{+}, t)+\vareps g_h(x_{\frac{1}{2} }^{+},v, t), \quad v<0\quad\text{(outflow)},
\label{eq:close-loop2}\\
&\lgl g_L(v, t)\rgl=0.
\label{eq:close-loop3}
\end{align}
\end{subequations}
Integrate \eqref{eq:close-loop1} in $v$ from $0$ to $1$ and \eqref{eq:close-loop2} from $-1$ to $0$, and sum them up, we get
\begin{subequations}
\begin{align}
&\rho_L=\rho_L(t) = \frac{1}{2}\Big(\rho_h(x_{\frac{1}{2}}^{+}, t)+\int_{0}^{1}f_L(v,t)dv+\vareps\int_{-1}^0 g_h(x_{\frac{1}{2} }^{+},v, t)dv\Big),\\
&g_L=g_L(v,t) = \begin{cases} 
		\frac{1}{\vareps}( f_L(v,t) - \rho_L(t) ),\;\;v\geq 0,\\
		\frac{1}{\vareps}(\rho_h(x_{\frac{1}{2}}^{+}, t)+\vareps g_h(x_{\frac{1}{2} }^{+},v, t)-\rho_L(v,t) ),\;\;v<0.
		\end{cases}
\end{align}
\end{subequations}
With a similar idea, we obtain at  the right boundary
\begin{subequations}
\begin{align}
&\rho_R=\rho_R(t) = \frac{1}{2}\Big(\rho_h(x_{N+\frac{1}{2}}^{-}, t)+\int_{-1}^{0}f_R(v,t)dv+\vareps\int_{0}^1 g_h(x_{N+\frac{1}{2} }^{-},v, t)dv\Big),\\
&g_R=g_R(v,t) = \begin{cases} 
		\frac{1}{\vareps}(\rho_h(x_{N+\frac{1}{2}}^{-}, t)+\vareps g_h(x_{N+\frac{1}{2} }^{-},v, t)-\rho_R(v,t) ),\;\;v\geq 0,\\
		\frac{1}{\vareps}( f_R(v,t) - \rho_R(t) ),\;\;v<0.
		\end{cases}
\end{align}
\end{subequations}

\noindent\textbf{Numerical flux:} The boundary strategies are imposed through numerical fluxes. On boundaries, we modify numerical fluxes in \eqref{eq:flux} to be
\begin{subequations}
\begin{align}
&(\breve{\rho_h})_{\frac{1}{2}}:=\rho_L,\;\;\widehat{\lgl vg_h\rgl}_\frac{1}{2}:= {\lgl vg_h\rgl}_{\frac{1}{2} }^+,\;\;\\
&(\breve{\rho_h})_{N+\frac{1}{2}}:=\rho_R,\;\;\widehat{\lgl vg_h\rgl}_{N+\frac{1}{2} }:= {\lgl vg_h\rgl}_{N+\frac{1}{2} }^-+c_R(\rho_R-\rho_{N+\frac{1}{2} }^-) ,
\\
&(\widetilde{vg_h})_\frac{1}{2}:=
\left\{
\begin{array}{ll}
{vg_L},&\mbox{if}\; v\geq0\\
({vg_h})_\frac{1}{2}^+,&\mbox{if}\; v<0
\end{array}
\right.
,\quad
(\widetilde{vg_h})_{N+\frac{1}{2} }:=
\left\{
\begin{array}{ll}
({vg_h})_{N+\frac{1}{2}}^-,&\mbox{if}\; v\geq0\\
vg_R,&\mbox{if}\; v<0
\end{array}
\right. . 
\end{align}
\end{subequations}
The penalty term $c_R(\rho_R-\rho_{N+\frac{1}{2} }^-)$ is added to maintain the accuracy of the schemes, and one can refer to \cite{castillo2002optimal,li2017optimal} for details on the role of this penalty term. In our simulation, we take $c_R=1$.

We want to mention that, due to this numerical boundary treatment, the IMEX$2$-DG$2$-S scheme is no longer unconditionally stable in the diffusive regime. We modify the time step condition in \eqref{eqn:CFL} with  $\Dt_{CFL2}=0.1h$ when $\vareps\leq 0.025h$. Note that this time step condition can still be larger than 
a parabolic time step condition $\Dt= O(h^2)$.
%%%%%%%%%%%%%%%%%%%%%%%%%%%%%%%%%%%%%%%%%%%%%%%%
% Numerical Results
%%%%%%%%%%%%%%%%%%%%%%%%%%%%%%%%%%%%%%%%%%%%%%%%
\subsection{Numerical examples}
%%%%%%%%%%%%%%%%%%%%%%%%%%%%%%%%%%%%%%%%%%%%%%%%
% Accuracy test
%%%%%%%%%%%%%%%%%%%%%%%%%%%%%%%%%%%%%%%%%%%%%%%%

\textbf{Example 1: smooth example \cite{peng2019stability}.}  We consider the one-group transport equation in slab geometry with a smooth example  on $\Omega_x=[0,2\pi]$. The initial conditions are 
$$
\rho(x,0)=\sin(x),\qquad
g(x,v,0)=-v \cos(x),
$$
with periodic boundary conditions, and $\sigma_s=1$, $\sigma_a=0$. $\Omega_x$ is partitioned with a uniform mesh and we define $N=\frac{(x_R-x_L)}{h}$. Numerical errors in $L_\infty$-norm and convergence orders are obtained by Richardson extrapolation:
\begin{subequations}
\begin{align} 
&E^{\rho}_{N}=|| \rho_h(x,T)-\rho_{\frac{h}{2}}(x,T) ||_{L_{\infty} (\Omega_x)},\quad\text{and},\quad O^{\rho}_{N}=log_{2}(E^{\rho}_N/E^{\rho}_{2N}),\\
&E^{g}_{N}=\max_{j=1,\dots N_v}|| g_h(x,v_j,T)-g_{\frac{h}{2}}(x,v_j,T) ||_{L_{\infty} (\Omega_x)},\quad\text{and},\quad O^{g}_{N}=log_{2}(E^{g}_N/E^{g}_{2N}).
\end{align}
\end{subequations}
Numerical results at $T=1$ with $\vareps=0.5$, $\vareps=10^{-2}$ and $\vareps=10^{-6}$ are shown in Tables \ref{tab:accuracy1}-\ref{tab:accuracy3}.
%, Table \ref{tab:accuracy2} and Table 
 We observe that the IMEX$k$-DG$k$-S scheme, $k=1, 2, 3$, has the optimal $k$-th order accuracy that seems to be uniform in 
$\vareps$.
%, $k=1,2,3$.

%%%%%%%%%%%%%%%%%%%%%%%%%%%%%%%%%%%%%%%%%%%%%%%%
\begin{table}[htbp]
  \centering
  \caption{Errors and convergence orders for the example 1, IMEX1-DG1-S}
  \vspace{0.1in}
    \begin{tabular}{|c|c|c|c|c|c|c|c|c|c|c|}
    \hline
	$\vareps$& $N$ & $E^{\rho}_N$  & order  &$E_N^g$ &order \\
    \hline
     \multirow{5}{*}{0.5} & 10  &1.921E-02& - & 1.909E-02&- \\
     					       & 20 & 8.709E-02&1.14 & 8.390E-03 & 1.19\\
					       & 40 & 3.540E-02 & 1.30 & 3.699E-03& 1.18\\
     					       & 80 & 1.619E-03& 1.13 &  1.744E-03& 1.08\\
					       & 160 &7.737E-04& 1.07 & 8.469E-04&1.04\\ \hline
     \multirow{5}{*}{$10^{-2}$} & 10  &1.029E-02& - & 1.467E-02 &-\\
     					       & 20 & 4.371E-03& 1.23 & 6.742E-03 &1.12\\
					       & 40 & 2.014E-03 & 1.12 & 3.204E-03& 1.07\\
     					       & 80 & 9.648E-04& 1.06 & 1.553E-03& 1.04\\
					       & 160 &5.265E-04& 0.87 & 7.986E-04&0.96\\ \hline
     \multirow{5}{*}{$10^{-6}$} & 10  & 1.011E-02 & - &  1.459E-02& - \\
     					       & 20 & 4.306E-03 & 1.23 & 6.709E-03  & 1.12 \\
					       & 40 &  1.988E-03 & 1.12 &3.189E-03&  1.07 \\
     					       & 80 & 9.520E-04& 1.06 & 1.546E-03& 1.04\\
					       & 160 &4.657E-04& 1.03 & 7.618E-04 &1.02  \\ \hline
					       
          \end{tabular}
          \label{tab:accuracy1}
\end{table}
%%%%%%%%%%%%%%%%%%%%%%%%%%%%%%%%%%%%%%%%%%%%%%%%

%%%%%%%%%%%%%%%%%%%%%%%%%%%%%%%%%%%%%%%%%%%%%%%%
\begin{table}[htbp]
  \centering
  \caption{Errors and convergence orders for the example 1, IMEX2-DG2-S}
    \vspace{0.1in}
    \begin{tabular}{|c|c|c|c|c|c|c|c|c|c|c|}
    \hline
	$\vareps$& $N$ & $E^{\rho}_N$  & order  &$E_N^g$ &order \\
    \hline
     \multirow{5}{*}{0.5} & 10  &3.505E-02& - & 3.911E-02 &- \\
     					       & 20 & 8.916E-03&1.97 & 9.991E-03 &1.97\\
					       & 40 & 2.205E-03 & 2.02 & 2.590E-03& 1.95\\
     					       & 80 & 5.479E-04& 2.01 &  6.563E-04& 1.98\\
					       & 160 &1.365E-04& 2.01 & 1.650E-04&1.99\\ \hline
     \multirow{5}{*}{$10^{-2}$} & 10  &3.519E-02& - & 4.215E-02 &-\\
     					       & 20 & 8.763E-03& 2.01 & 8.869E-03 &2.25\\
					       & 40 & 2.206E-03 &2.00 & 2.283E-03& 1.96\\
     					       & 80 & 5.523E-04& 2.00 & 5.906E-04& 1.95\\
					       & 160 &1.381E-04&2.00 & 1.536E-04&1.94\\ \hline
     \multirow{5}{*}{$10^{-6}$} & 10  & 3.518E-02 & - &  3.482E-02&- - \\
     					       & 20 & 8.726E-03 & 2.01 & 8.629E-03  & 2.01 \\
					       & 40 & 2.195E-03 & 1.99 &2.172E-03&  1.99 \\
     					       & 80 & 5.494E-04& 2.00 & 5.435E-04& 2.00\\
					       & 160 &1.374E-04& 2.00 & 1.360E-04 &2.00  \\ \hline
					       
          \end{tabular}
          \label{tab:accuracy2}
\end{table}
%%%%%%%%%%%%%%%%%%%%%%%%%%%%%%%%%%%%%%%%%%%%%%%%

%%%%%%%%%%%%%%%%%%%%%%%%%%%%%%%%%%%%%%%%%%%%%%%%
\begin{table}[htbp]
  \centering
  \caption{Errors and convergence orders for the example 1, IMEX3-DG3-S}
    \vspace{0.1in}
    \begin{tabular}{|c|c|c|c|c|c|c|c|c|c|c|}
    \hline
	$\vareps$& $N$ & $E^{\rho}_N$  & order  &$E_N^g$ &order \\
    \hline
     \multirow{5}{*}{0.5} 	       & 10  &2.588E-03& - & 2.676E-03 &- \\
     					       & 20 & 3.215E-04&  3.01 & 4.103E-04 &2.71\\
					       & 40 & 4.028E-05 & 3.00 & 6.495E-05& 2.66\\
     					       & 80 & 5.036E-06&  3.00 & 9.198E-06& 2.82\\
					       & 160 &6.303E-07& 3.00 & 1.22E-06&   2.91\\ \hline
     \multirow{5}{*}{$10^{-2}$}   & 10  &2.510E-03& - & 2.543E-03 &-\\
     					       & 20 & 3.214E-04& 2.97 & 3.724E-04 &2.77\\
					       & 40 & 4.039E-05 & 2.99 & 1.109E-04& 1.75\\
     					       & 80 & 5.061E-06& 3.00 & 5.292E-06& 4.39\\
					       & 160 &6.328E-07& 3.00 & 6.659E-07&2.99\\ 
					       & 320 &7.910E-08& 3.00 & 8.355E-08&2.99\\ 
					       \hline
     \multirow{5}{*}{$10^{-6}$}   & 10  & 2.505E-03 & - &  2.554E-03&- \\
     					       & 20 & 3.211E-04 & 2.96 & 3.174E-04  & 3.01 \\
					       & 40 & 4.041E-05 & 2.99 &3.998E-05&  3.00 \\
     					       & 80 & 5.060E-06& 3.00 & 5.007E-06& 3.00\\
					       & 160 &6.327E-07& 3.00 &  6.269E-07 &3.00  \\ \hline
					       
          \end{tabular}
          \label{tab:accuracy3}
\end{table}
%%%%%%%%%%%%%%%%%%%%%%%%%%%%%%%%%%%%%%%%%%%%%%%%

%%%%%%%%%%%%%%%%%%%%%%%%%%%%%%%%%%%%%%%%%%%%%%%%
% Two materials problem
%%%%%%%%%%%%%%%%%%%%%%%%%%%%%%%%%%%%%%%%%%%%%%%%
\smallskip
\textbf{Example 2: two-material problem \cite{larsen1989asymptotic,lemou2008new}.} We consider a two-material problem on $\Omega_x=[0, 11]$ and $\Omega_v=[-1, 1]$ with isotropic inflow boundary conditions. The setup is as follows.
\begin{align}
&\sigma_s=0,\;\;\sigma_a=1,\;\;\text{if}\; x\in[0,1],\notag\\
&\sigma_s=100,\;\;\sigma_a=0,\;\;\text{if}\;x\in[1,11],\notag\\
&f_L(v,t) = 5,\;\; f_R(v,t)=0,\;\;f(x,v,0) = 0,\;\;
\end{align}
and $\vareps=1$. 
%$\Omega_v=[-1,1]$. 
We examine the numerical solutions at a shorter time 
$T=1.5$ and also the steady state solution obtained at $T=20000$. This problem has a pure absorbing region with the length of one mean-free path on the left and a pure scattering region with the length of $1000$ mean-free path on the right. Subregiones with different scales coexist. 
From the left boundary, an isotropic inflow enters the computational region, and it instantly becomes anisotropic. An interior layer is formed near the interface between the absorbing region and the scattering region. 

We use a non-uniform mesh with $h=h^{(1)}=\frac{1}{20}$ on $[0,1]$ and $h=h^{(2)}=\frac{1}{2}$ on $[1,11]$, and the time step $\Dt$ is determined by  \eqref{eqn:CFL} using $h=h^{(1)}$. With this example, we also want to compare the performance of  the IMEX$k$-DG$k$-S schemes proposed here and the IMEX$k$-LDG$k$ schemes in \cite{peng2019stability} with the weight function $\omega = \exp(-\frac{\varepsilon}{100h})|_{h=h^{(1)}}$.
Numerical results are shown in Figure \ref{fig:two_material_short} and Figure \ref{fig:two_material_long}.  The reference solution is obtained by  the first order forward Euler upwind finite difference scheme applied to the original kinetic equation \eqref{eq:f_model} with $h=\frac{11}{20000}$ and $\Dt=10^{-5}$ for $T=1.5$, and with $h=\frac{11}{2000}$ and $\Dt=10^{-4}$ for $T=20000$.

At $T=1.5$, our proposed schemes capture the solutions very well. The third order IMEX3-LDG3-S scheme has the best result.  We observe that the IMEX$k$-DG$k$-S schemes outperform the IMEX$k$-LDG$k$ schemes with the chosen weight. Unlike for IMEX$k$-LDG$k$ schemes, one does not need to choose a weight function for our proposed methods for this example when both transport dominant and diffusion dominant regions coexist.

At $T=20000$ when the solution reaches its steady state, the numerical solutions by the IMEX$k$-DG$k$-S scheme, $k=1, 2, 3$ match the reference solutions well, and they are comparable with those in \cite{peng2019stability} by  IMEX$k$-LDG$k$ scheme. Higher order schemes lead to better resolution as expected.

%%%%%%%%%%%%%%%%%%%%%%%%%%%%%%%%%%%%%%%%%
\begin{figure}
  \centering
  \subfigure[$\rho$, IMEX$1$-DG$1$-S]
  {
     \includegraphics[width=0.31\linewidth]{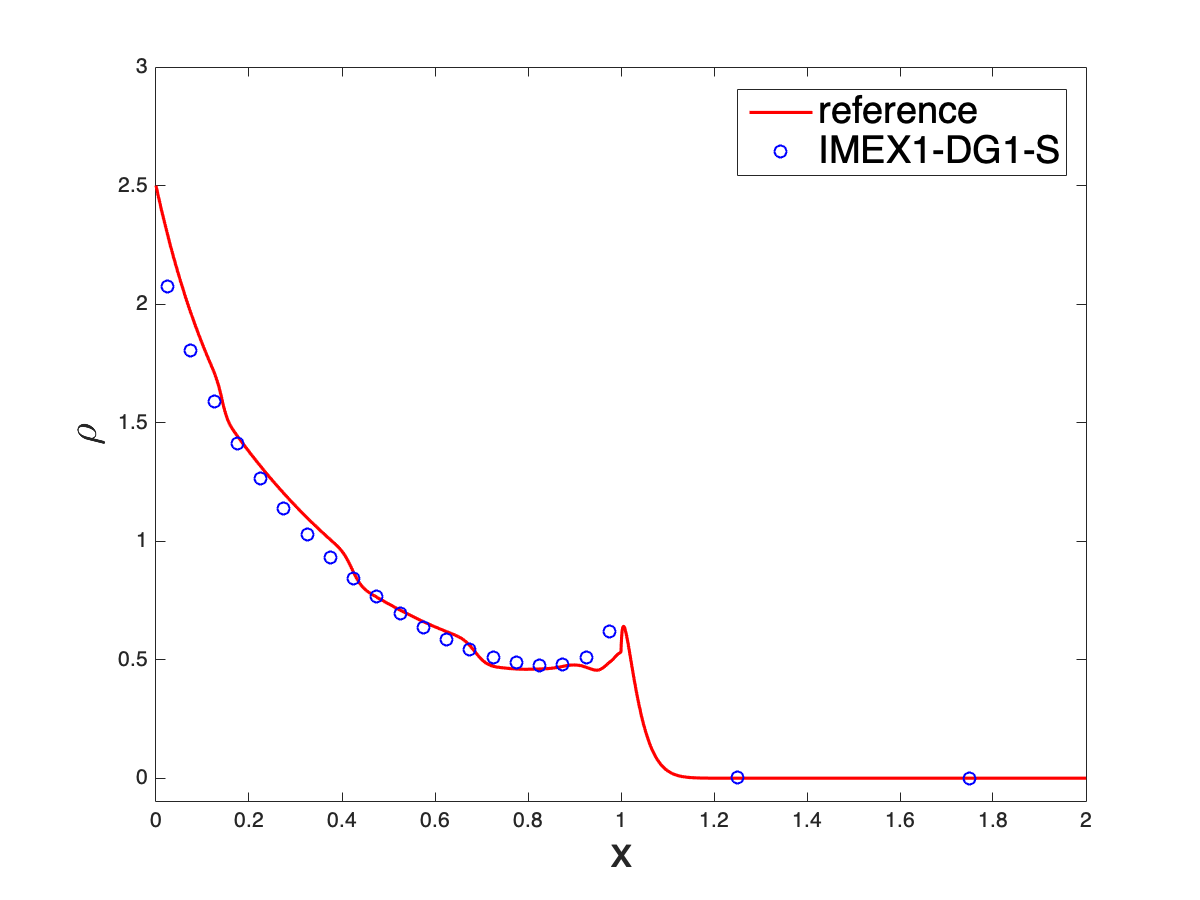}
    }
  \subfigure[$\rho$, IMEX$2$-DG$2$-S]
  {
        \includegraphics[width=0.31\linewidth] {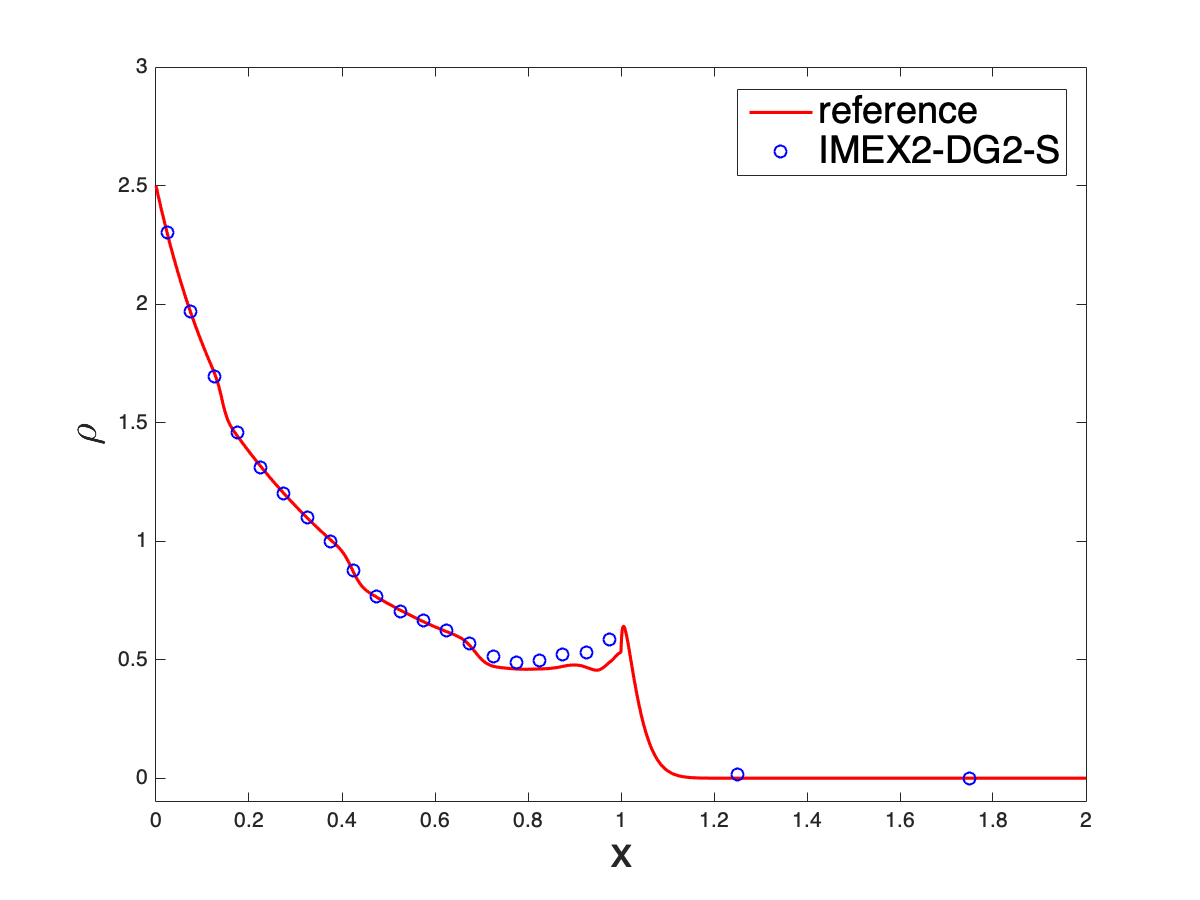}
    }
   \subfigure[$\rho$, IMEX$3$-DG$3$-S]
  {
        \includegraphics[width=0.31\linewidth] {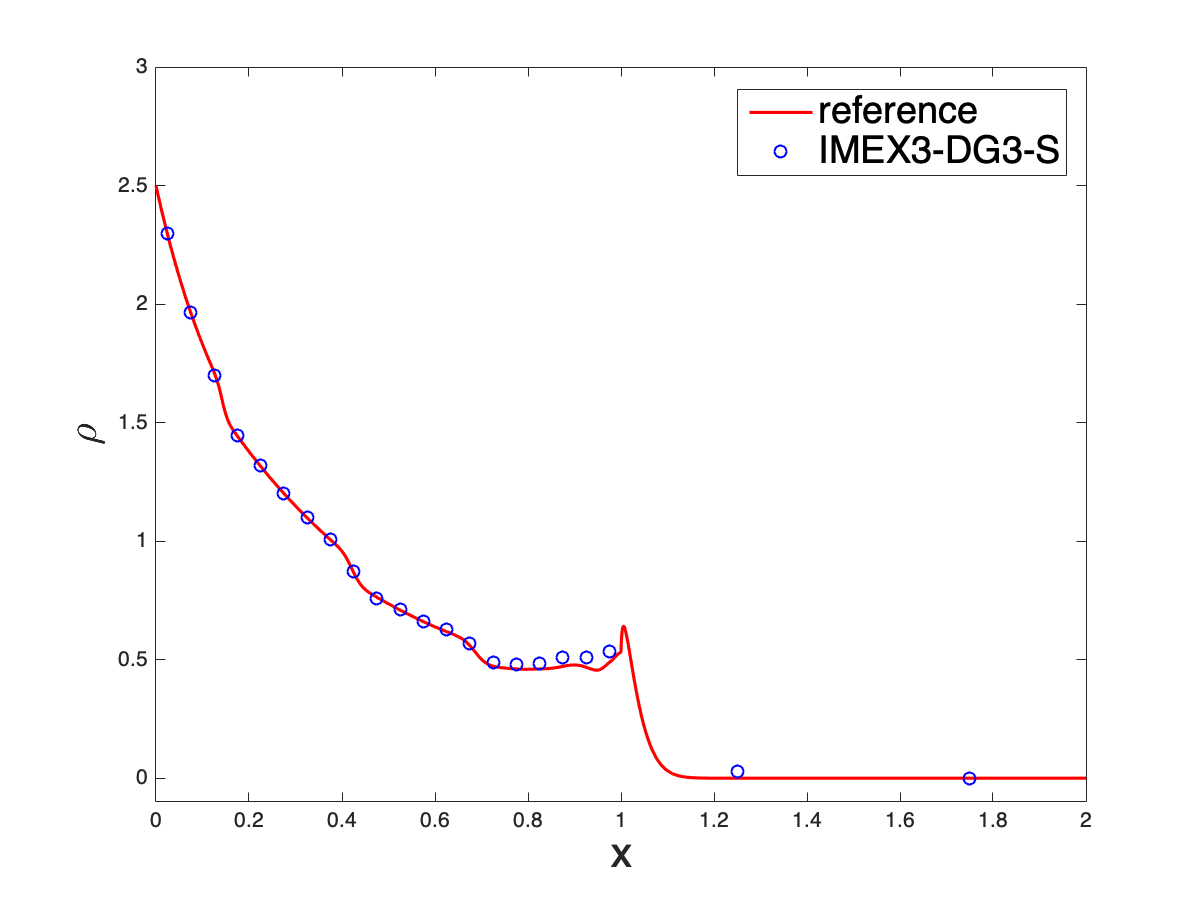}
    }
      \subfigure[$\rho$, IMEX$1$-LDG$1$]
  {
     \includegraphics[width=0.31\linewidth]{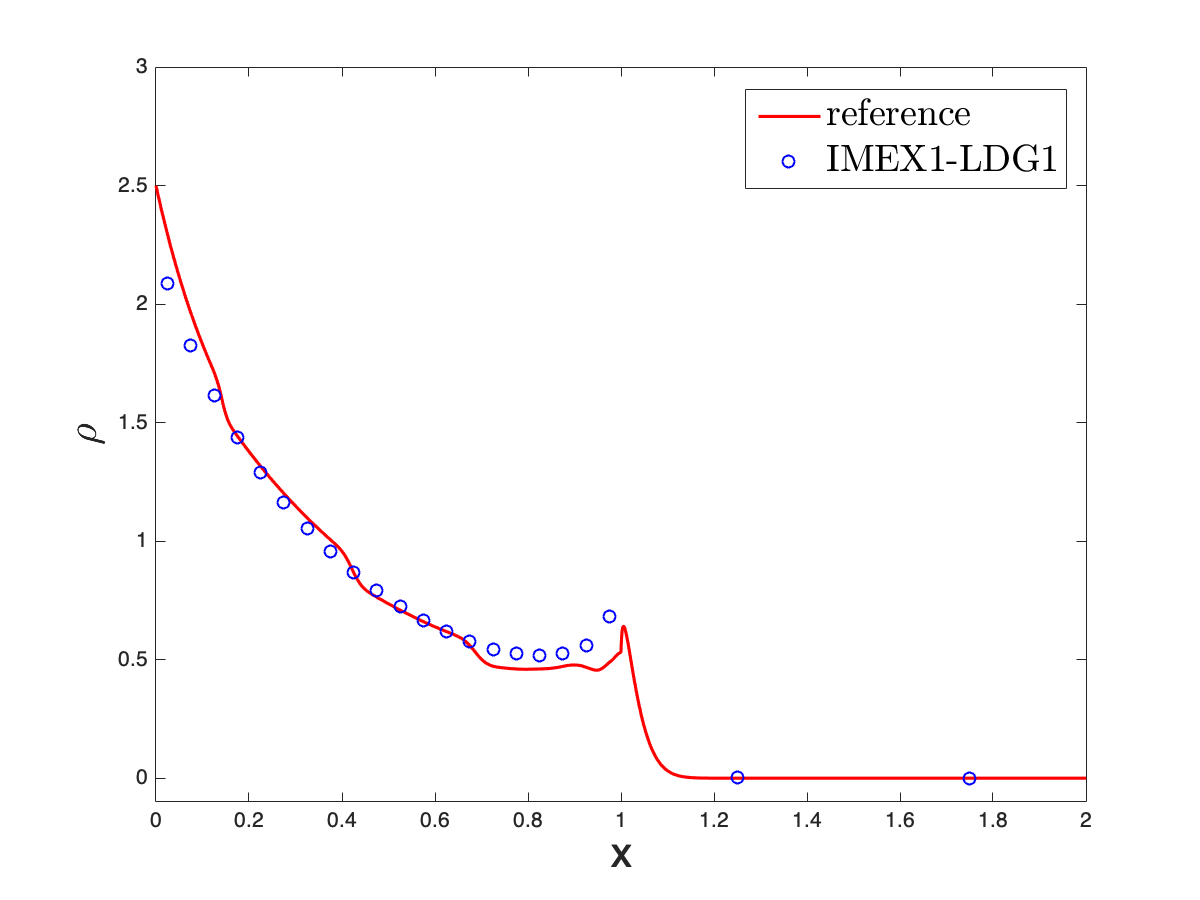}
    }
  \subfigure[$\rho$, IMEX$2$-LDG$2$]
  {
        \includegraphics[width=0.31\linewidth] {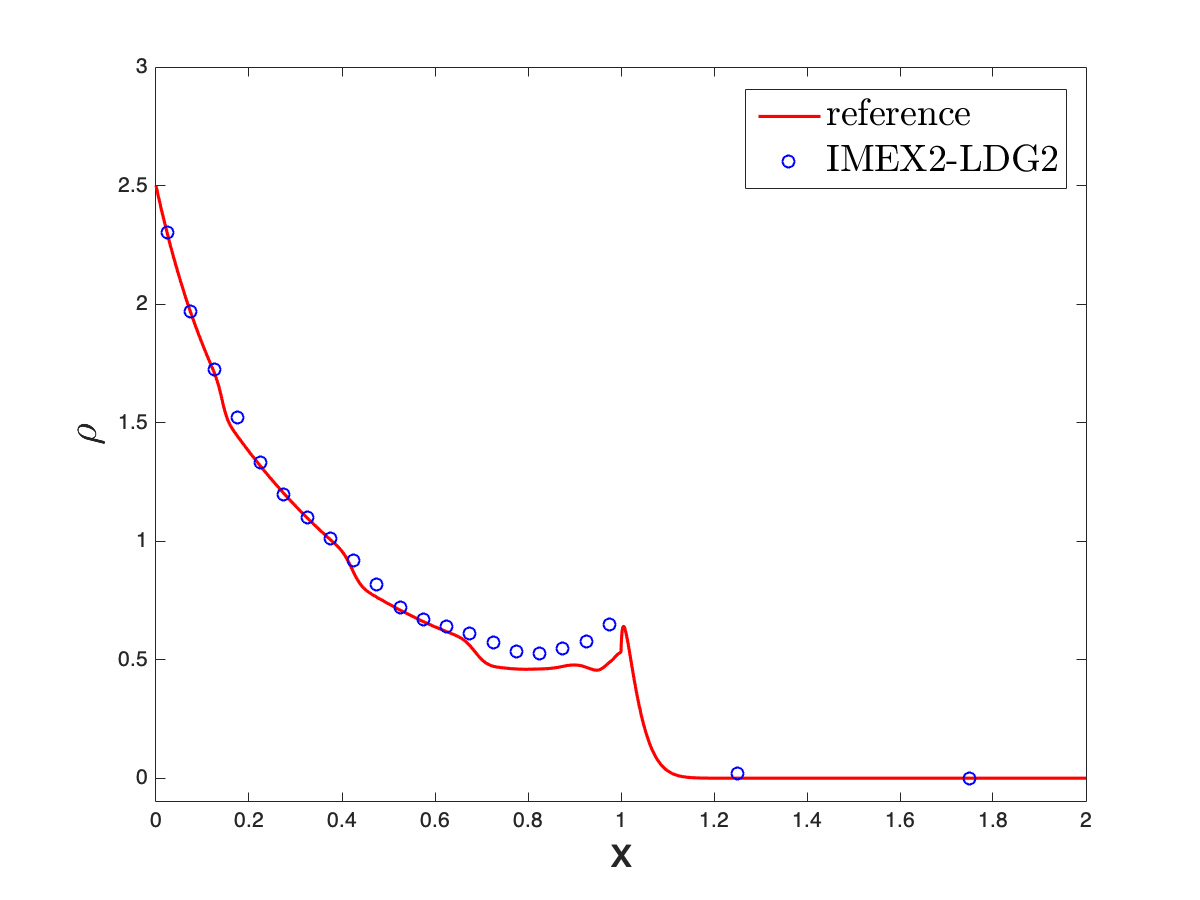}
    }
   \subfigure[$\rho$, IMEX$3$-LDG$3$]
  {
        \includegraphics[width=0.31\linewidth] {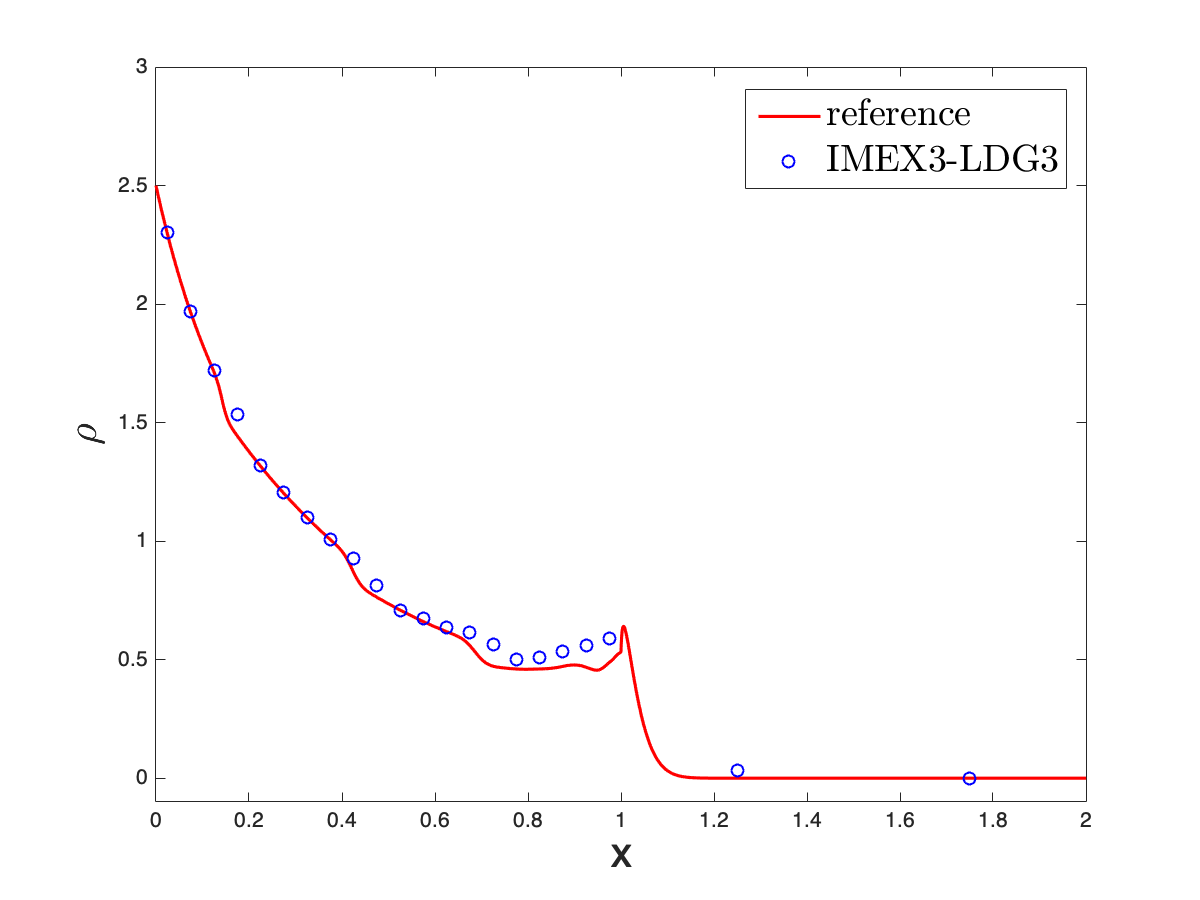}
    }
  \caption{ Example 2: two-material problem
  % for telegraph equation, 
  $T=1.5$, zoomed in with $x\in[0,2]$. 
  		}
  \label{fig:two_material_short}  %% label for entire figure
\end{figure}
%%%%%%%%%%%%%%%%%%%%%%%%%%%%%%%%%%%%%%%%%
\begin{figure}
  \centering
  \subfigure[$\rho$, IMEX$1$-DG$1$-S]
  {
     \includegraphics[width=0.31\linewidth]{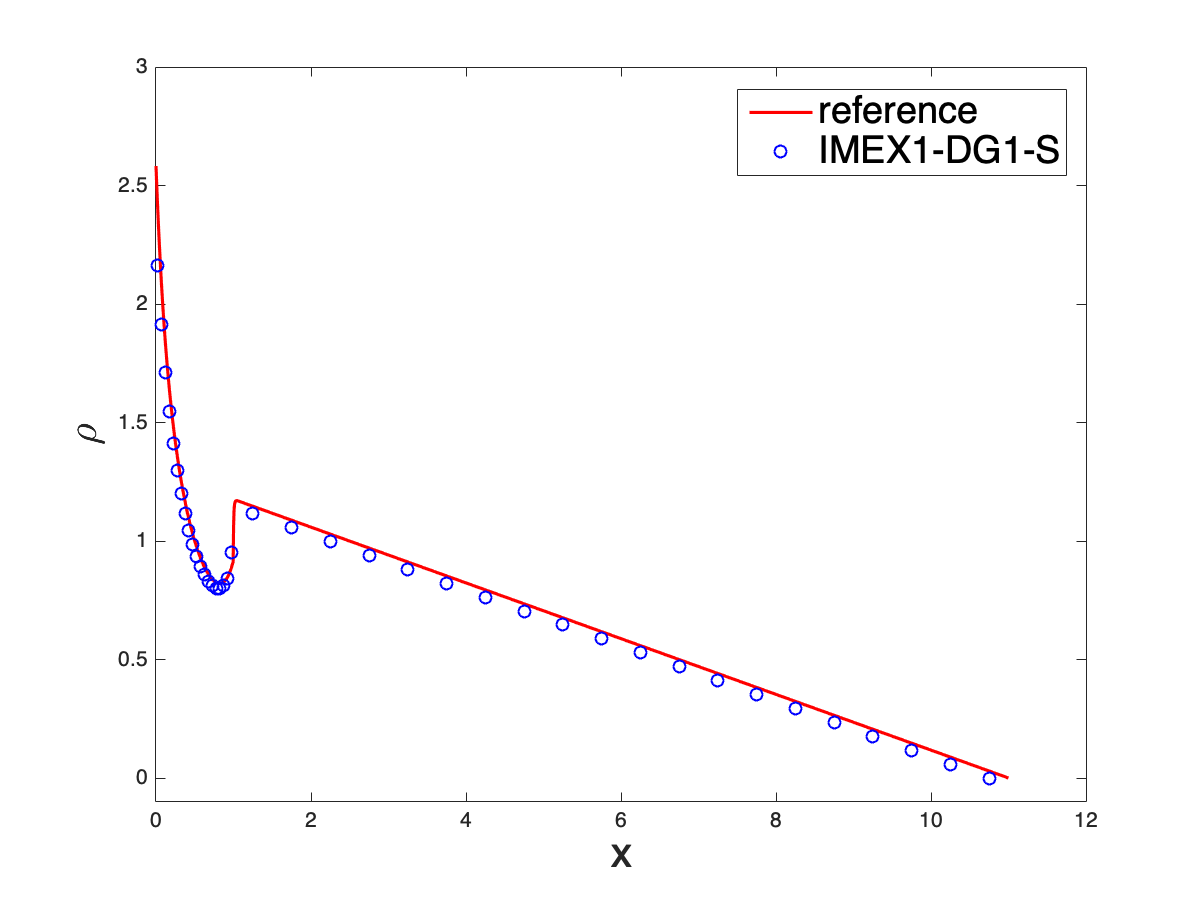}
    }
  \subfigure[$\rho$, IMEX$2$-DG$2$-S]
  {
        \includegraphics[width=0.31\linewidth] {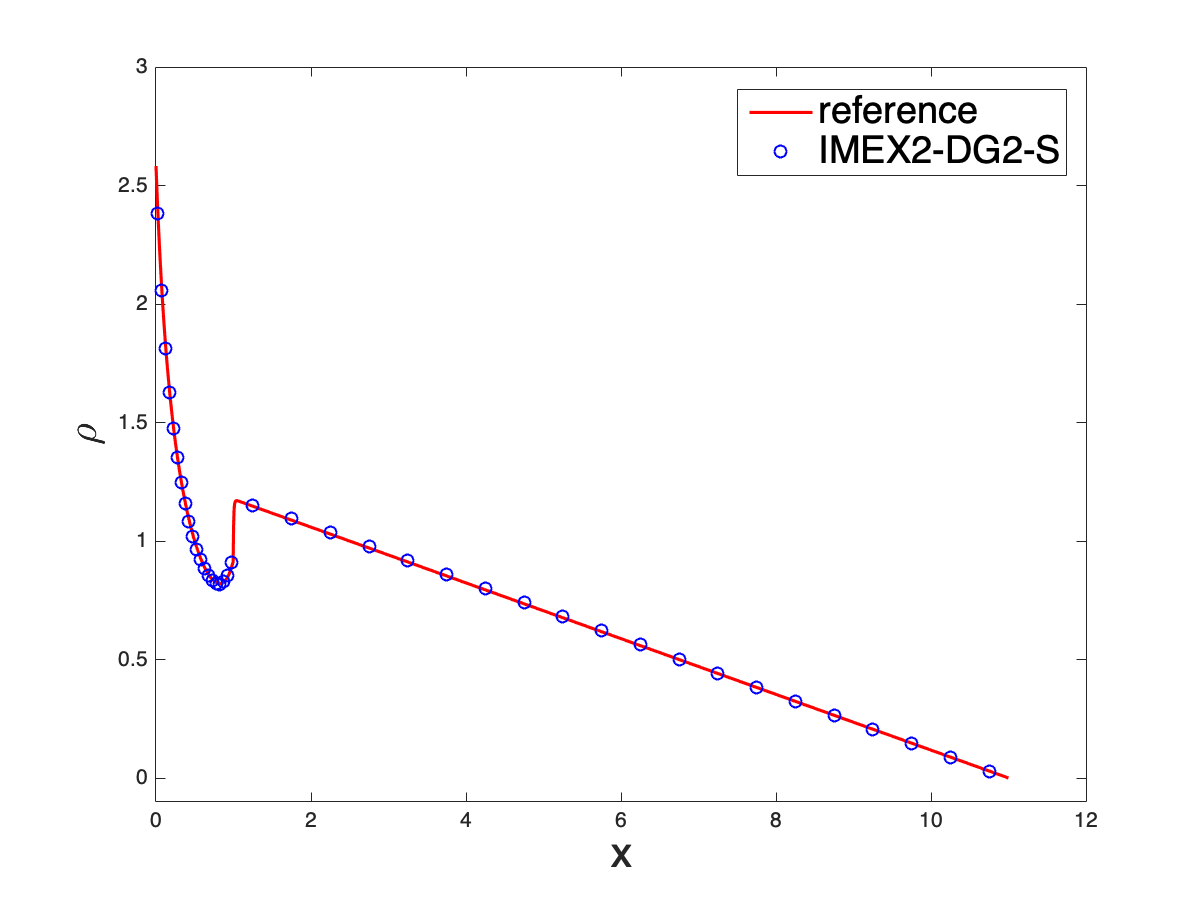}
    }
   \subfigure[$\rho$, IMEX$3$-DG$3$-S]
  {
        \includegraphics[width=0.31\linewidth] {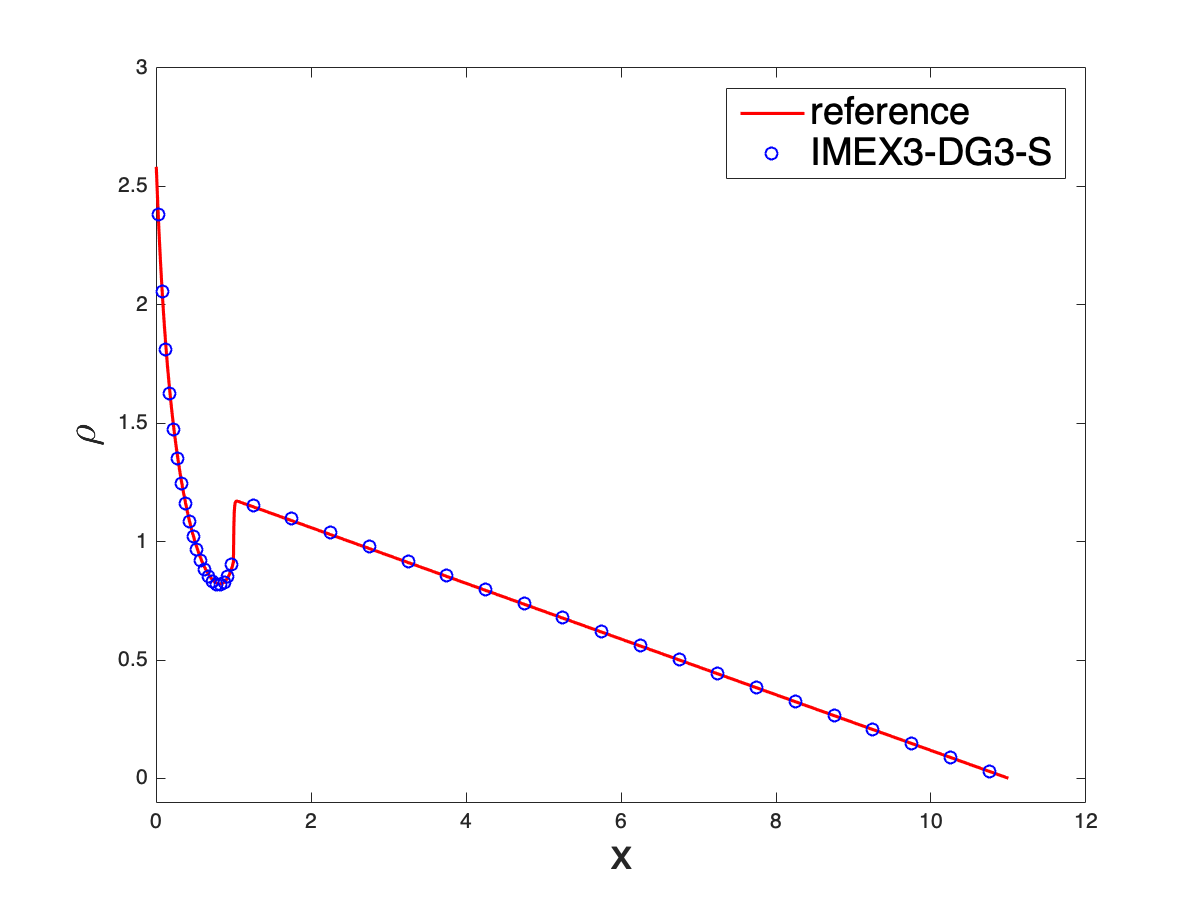}
    }
  \caption{ Example 2: two-material problem,
  % for telegraph equation,
   $T=20000$, $x\in[0,11]$. 
  		}
  \label{fig:two_material_long}  %% label for entire figure
\end{figure}
%%%%%%%%%%%%%%%%%%%%%%%%%%%%%%%%%%%%%%%%%

%%%%%%%%%%%%%%%%%%%%%%%%%%%%%%%%%%%%%%%%%%%%%%%%
% Changing scattering problem
%%%%%%%%%%%%%%%%%%%%%%%%%%%%%%%%%%%%%%%%%%%%%%%%
\smallskip
\textbf{Example 3: problem with varying scattering frequency and constant source term \cite{lemou2008new}.} We consider the one-group transport equation in slab geometry %$\Omega_v=[-1,1]$ 
with a source term $G$:
\begin{align}
\vareps \partial_t f + v\partial_x f = -\frac{\sigma_s}{\vareps}(\lgl f \rgl - f) + \vareps \sigma_a f+\vareps G.
\end{align}
The computational domain is $\Omega_x=[0,1]$, and 
\begin{align}
&\sigma_s(x) = 1+100x^2,\;\;\sigma_a=0,\;\;G=1,\;\;
\notag\\
&f_L(v,t) = 0,\;\; f_R(v,t)=0,\;\; f(x,v,0)=0,\;\;\vareps = 10^{-2}.
\end{align}
The effective scaling is determined by $\frac{\vareps}{\sigma_s(x)}$, hence, it is varying in the computational domain.  

We use a uniform mesh with $h=\frac{1}{40}$ and the source term $G$ % $\vareps G$
 is treated explicitly. Numerical results for $T=0.4$ are presented in Figure \ref{fig:changing_scattering}. The reference solution is obtained by the first order forward Euler upwind finite difference scheme applied to \eqref{eq:f_model} with $h=\frac{1}{20000}$ and $\Dt=0.1\vareps h$. As the value of $\sigma_s(x)$ is larger on the right, the scattering effect is stronger on that side. As a result, sharp feature exists near the right boundary. All schemes match the reference solution well on this relatively coarse mesh, and high order schemes perform better, especially near the right boundary.

%%%%%%%%%%%%%%%%%%%%%%%%%%%%%%%%%%%%%%%%%
\begin{figure}
  \centering
  \subfigure[$\rho$, IMEX$1$-DG$1$-S]
  {
     \includegraphics[width=0.31\linewidth]{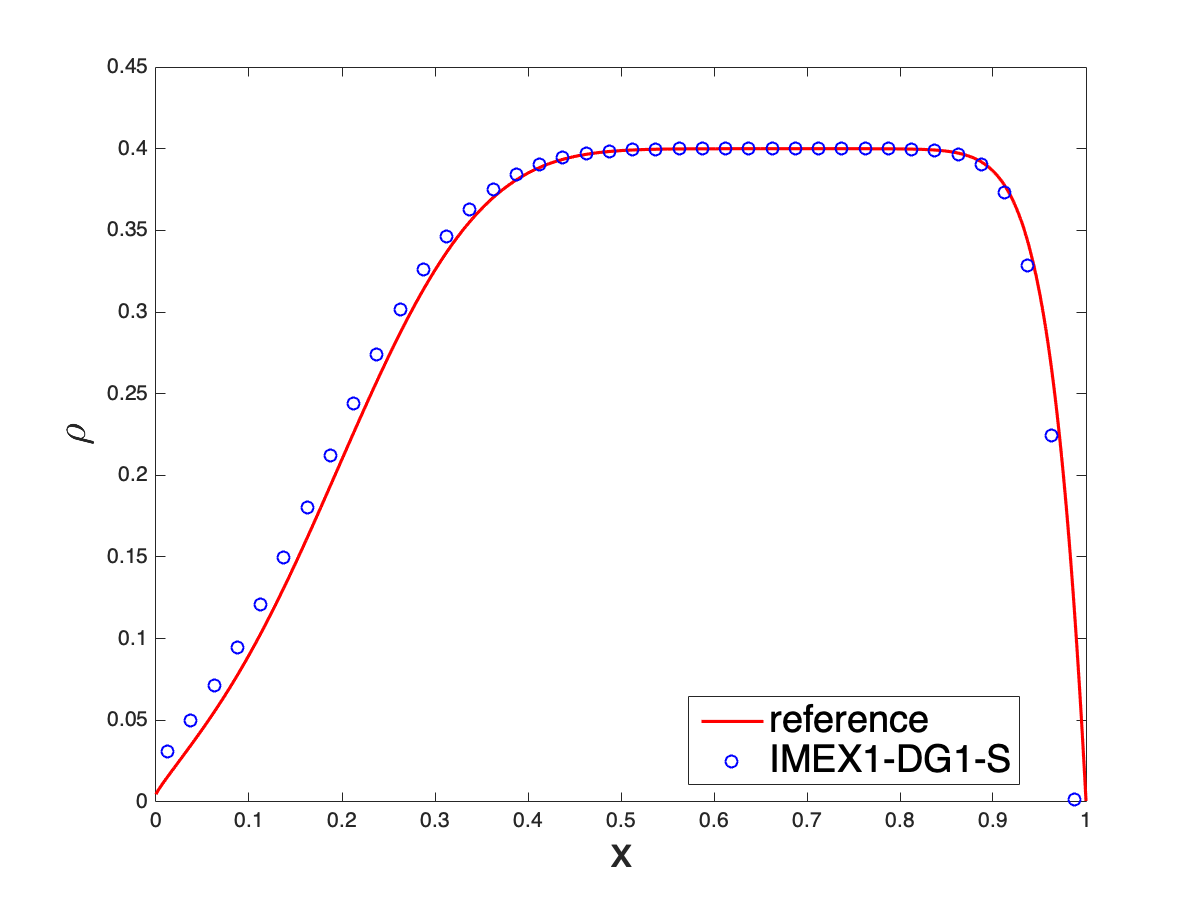}
    }
  \subfigure[$\rho$, IMEX$2$-DG$2$-S]
  {
        \includegraphics[width=0.31\linewidth] {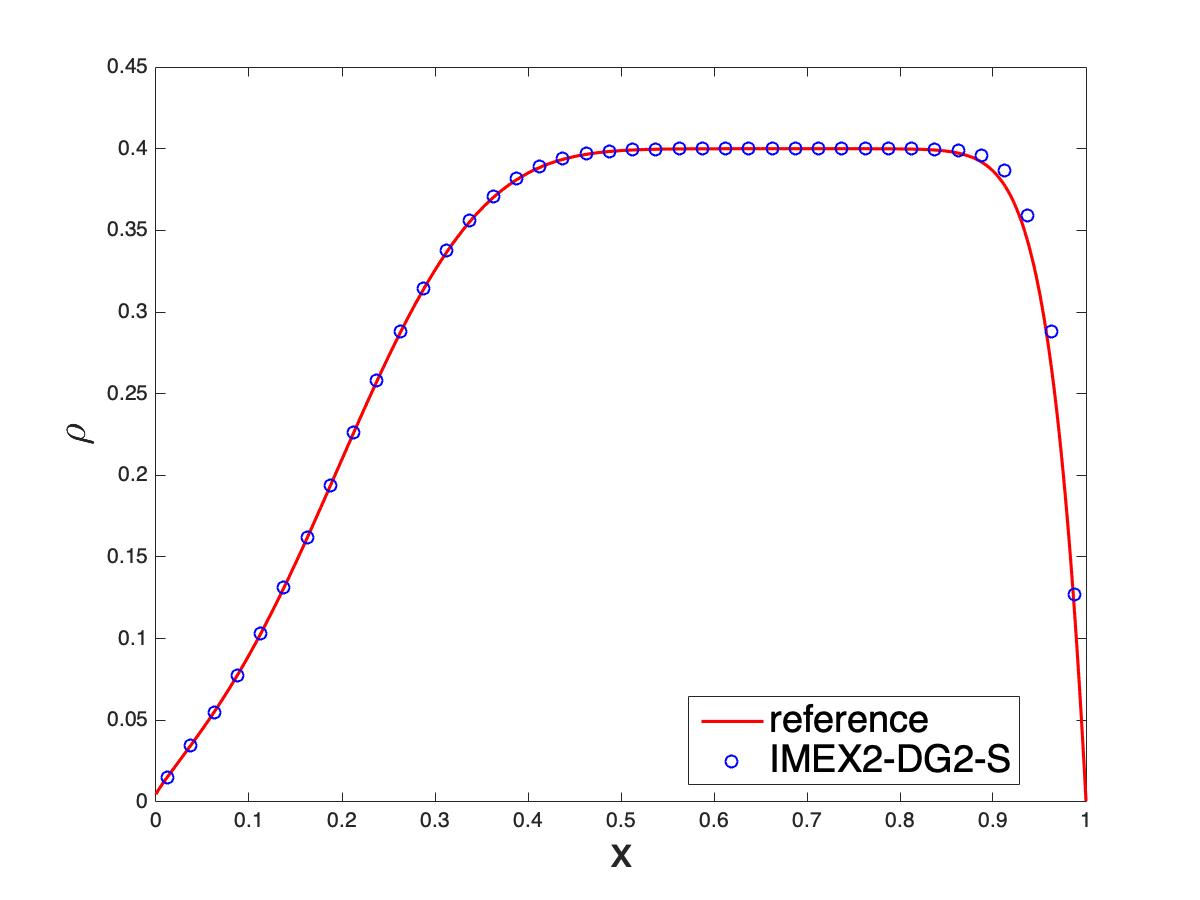}
    }
   \subfigure[$\rho$, IMEX$3$-DG$3$-S]
  {
        \includegraphics[width=0.31\linewidth] {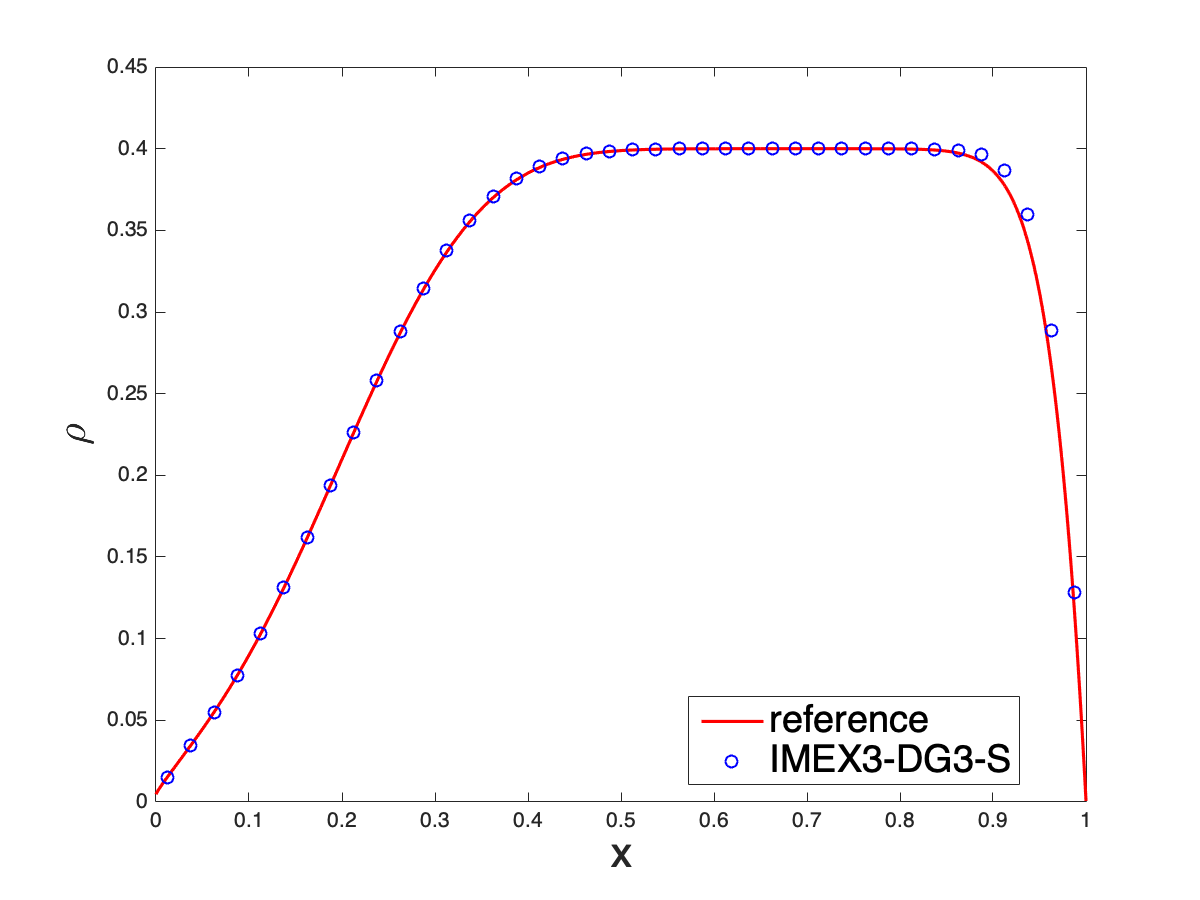}
    }
  \caption{ Example 3: changing scattering problem for one-group transport equation, $T=0.4$. 
  		}
  \label{fig:changing_scattering}  %% label for entire figure
\end{figure}
%%%%%%%%%%%%%%%%%%%%%%%%%%%%%%%%%%%%%%%%%

%%%%%%%%%%%%%%%%%%%%%%%%%%%%%%%%%%%%%%%%%%%%%%%%
% Inconsistent boundary problem
%%%%%%%%%%%%%%%%%%%%%%%%%%%%%%%%%%%%%%%%%%%%%%%%
\smallskip
\textbf{Example 4: diffusive and kinetic regime with isotropic inflow Dirichlet boundary conditions \cite{boscarino2013implicit,lemou2008new}.} In this example, we consider the one-group transport equation in  slab geometry on  
%$\Omega_v=[-1,1]$, 
$\Omega_x=[0,1]$, and 
\begin{align}
\sigma_s=1,\;\;\sigma_a=0,\;\;f_L(v,t)=1,\;\; f_R(v,t)=0,\;\;f(x,v,0)=0,
\end{align}
with $\vareps=1, 10^{-8}$.

In Figure \ref{fig:dirichlet}, we report numerical results on a uniform mesh with $h=\frac{1}{40}$. The reference solution for $\vareps=1$ is obtained by the first order forward Euler upwind finite difference scheme applied to \eqref{eq:f_model} with $h=\frac{1}{2000}$ and $\Dt = 0.5\varepsilon h$, while the reference solution for $\vareps=10^{-8}$ is obtained by a central difference scheme solving the diffusion limit \eqref{eq:diffusion_limit} with $h=\frac{1}{2000}$ and $\Dt = 0.25h^2$.  
For comparison, we also include  in Figure \ref{fig:dirichlet} the numerical results  by  
 the IMEX$k$-LDG$k$ schemes in \cite{peng2019stability}  with the weight function $\omega=\exp(-\frac{\varepsilon}{h})$ and  $\omega=1$, and when $\vareps=1$ .

When the problem is relatively kinetic with $\vareps=1$, it is observed that the numerical solutions by the proposed methods match the reference solutions well. The results are comparable with that by the IMEX$k$-LDG$k$ methods with the weight function $\omega=\exp({-\frac{\varepsilon}{h}})$, and both  are better than that by the IMEX$k$-LDG$k$ methods with the constant  weight function $\omega=1$.
Note that in this example, the initial and boundary conditions at $x=0$ are not compatible, and this introduces a Dirac delta structure in $\partial_x\rho$ at $t=0$ and subsequently sharper features in the solution form   near the left boundary. All these pose challenge to approximate the weighted diffusion term $\omega\partial_{xx}\rho$, unless $\omega$ is chosen to be small to balance the term $\partial_{xx}\rho$. This explains the IMEX$k$-LDG$k$ schemes with the weight function $\omega=\exp(-\frac{\varepsilon}{h})$ outperform that with $\omega=1$. Our IMEX-DG-S schemes on the other hand do not have a weight function to tune for this example.
 
When the problem is relatively diffusive with $\vareps=10^{-8}$, 
%for better resolution,
 we take $\Dt=0.25h$ in the diffusive regime, instead of the original $\Dt=0.75h$ in \eqref{eqn:CFL} (still stable), for both the IMEX1-DG1-S and IMEX3-DG3-S schemes.
The numerical solutions by the IMEX$k$-DG$k$-S scheme, $k=1, 2, 3$, match the reference solutions well.
The higher order schemes lead to better resolution.

%%%%%%%%%%%%%%%%%%%%%%%%%%%%%%%%%%%%%%%%%
\begin{figure}
  \centering
    \subfigure[$\rho$ for $\vareps=1$, IMEX$1$-DG$1$-S]
  {
     \includegraphics[width=0.31\linewidth]{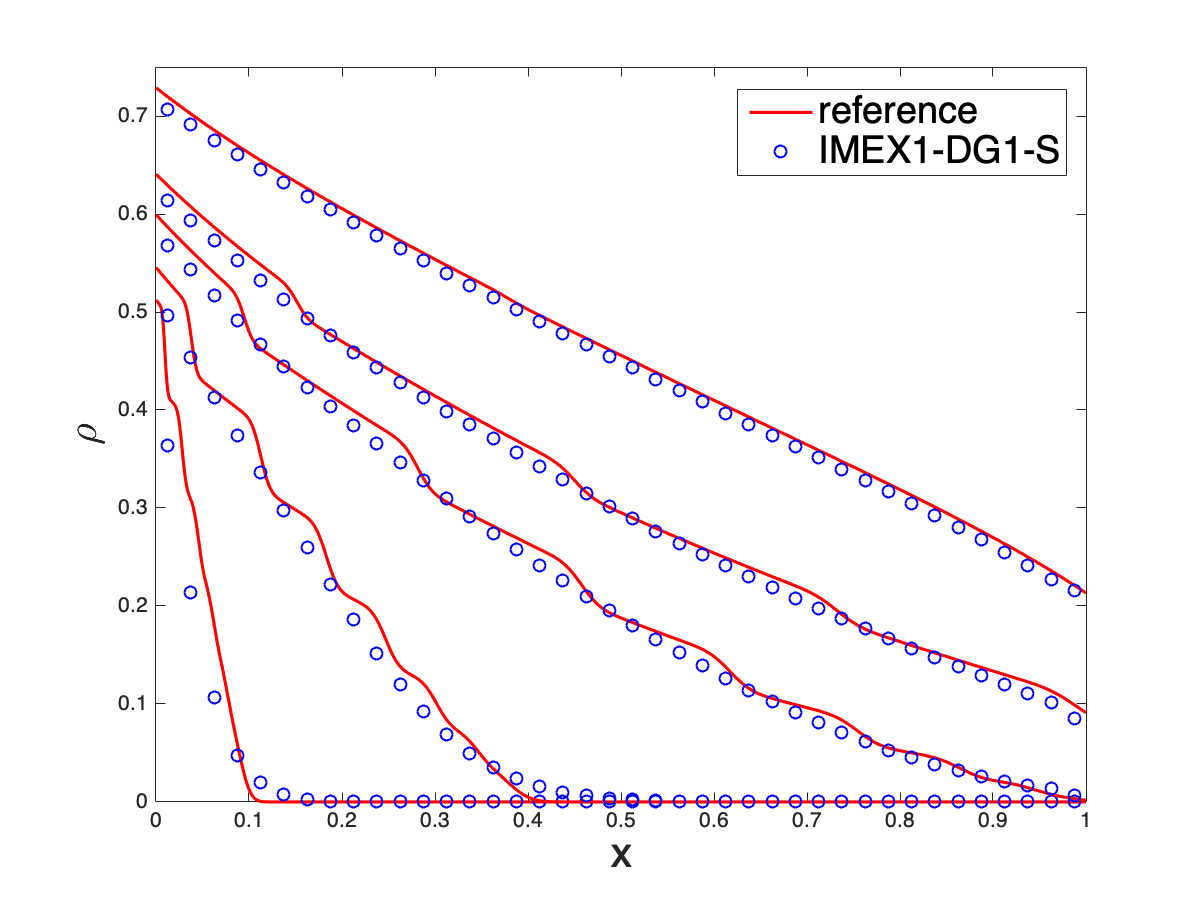}
    }
  \subfigure[$\rho$ for $\vareps=1$, IMEX$2$-DG$2$-S]
  {
        \includegraphics[width=0.31\linewidth] {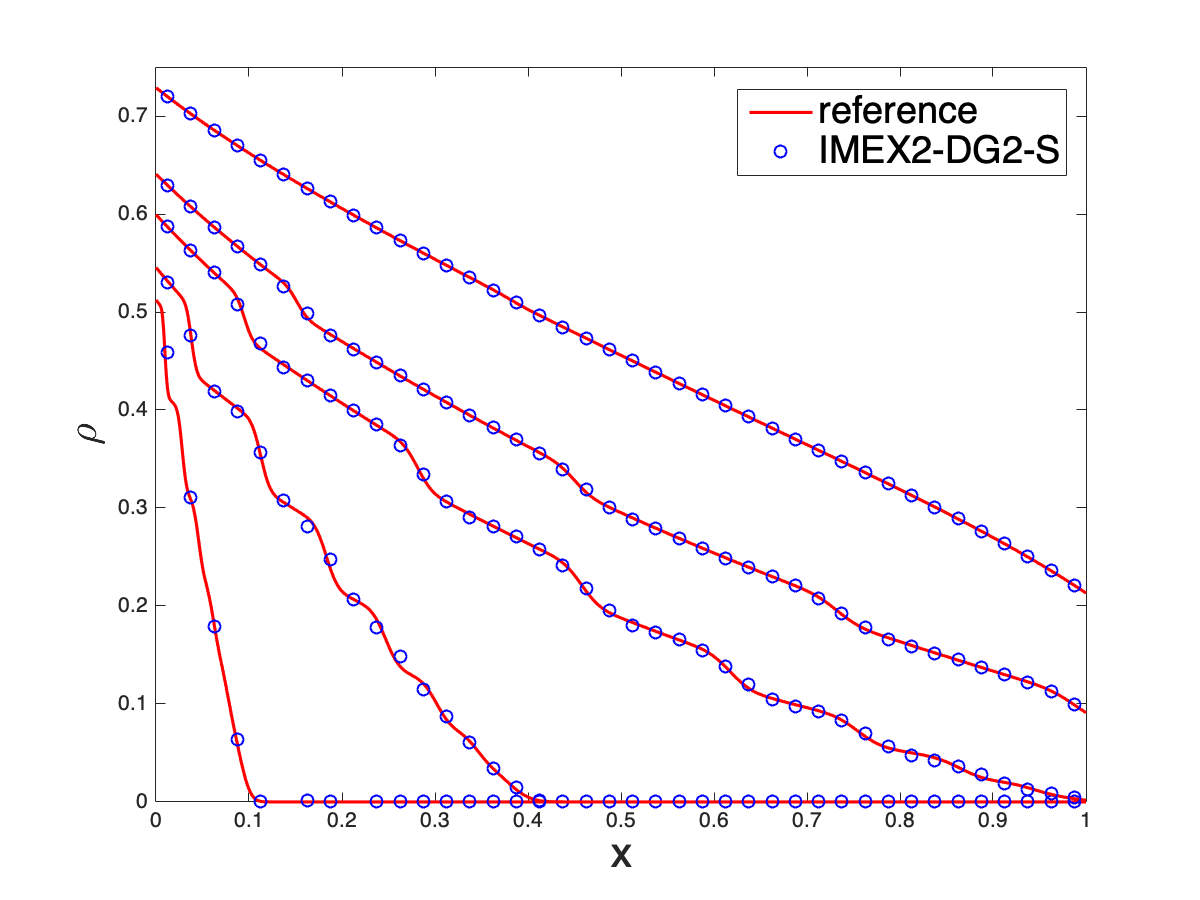}
    }
   \subfigure[$\rho$ for $\vareps=1$, IMEX$3$-DG$3$-S]
  {
        \includegraphics[width=0.31\linewidth] {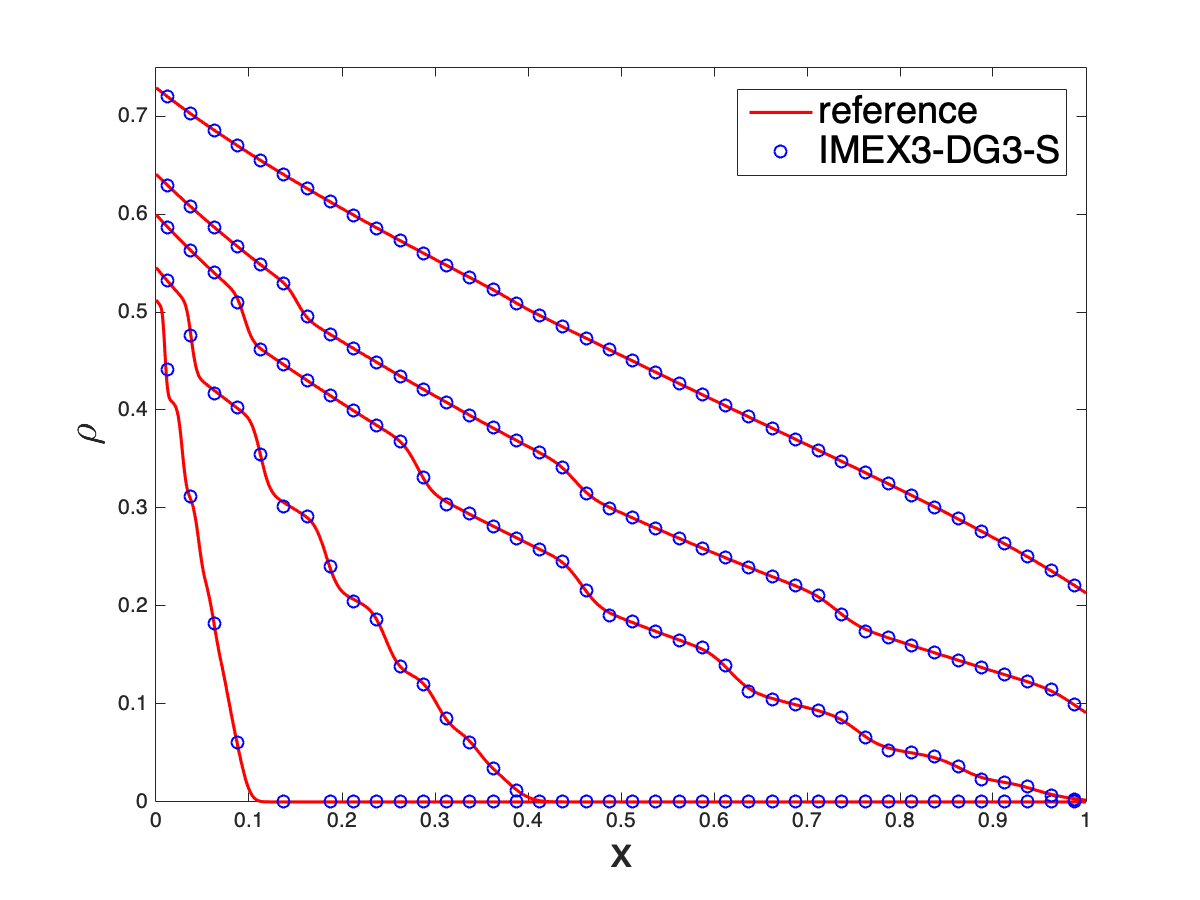}
    }
    \\
    \subfigure[$\rho$ for $\vareps=1$, IMEX$1$-LDG$1$ with $\omega=\exp(-\frac{\varepsilon}{h})$]
  {
     \includegraphics[width=0.31\linewidth]{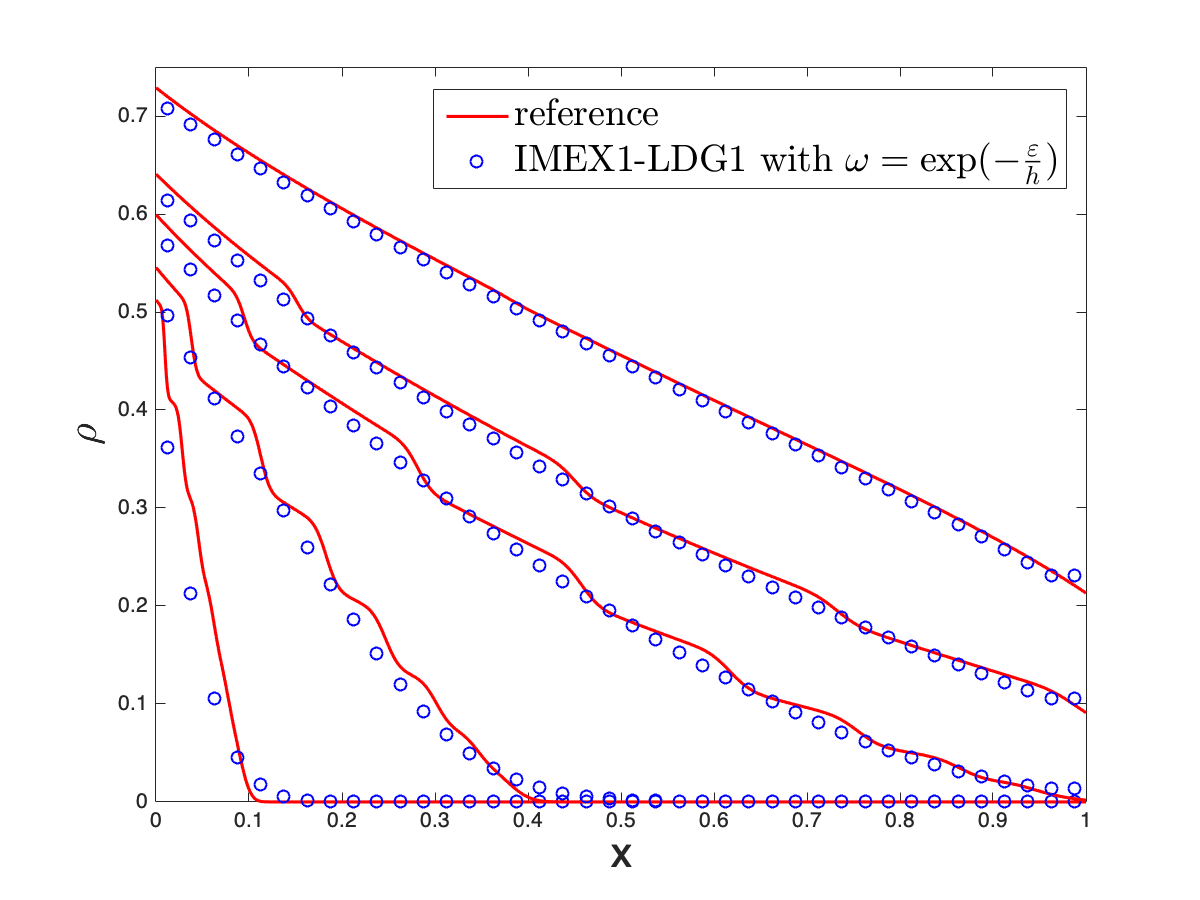}
    }
  \subfigure[$\rho$ for $\vareps=1$, IMEX$2$-LDG$2$ with $\omega=\exp(-\frac{\varepsilon}{h})$]
  {
        \includegraphics[width=0.31\linewidth] {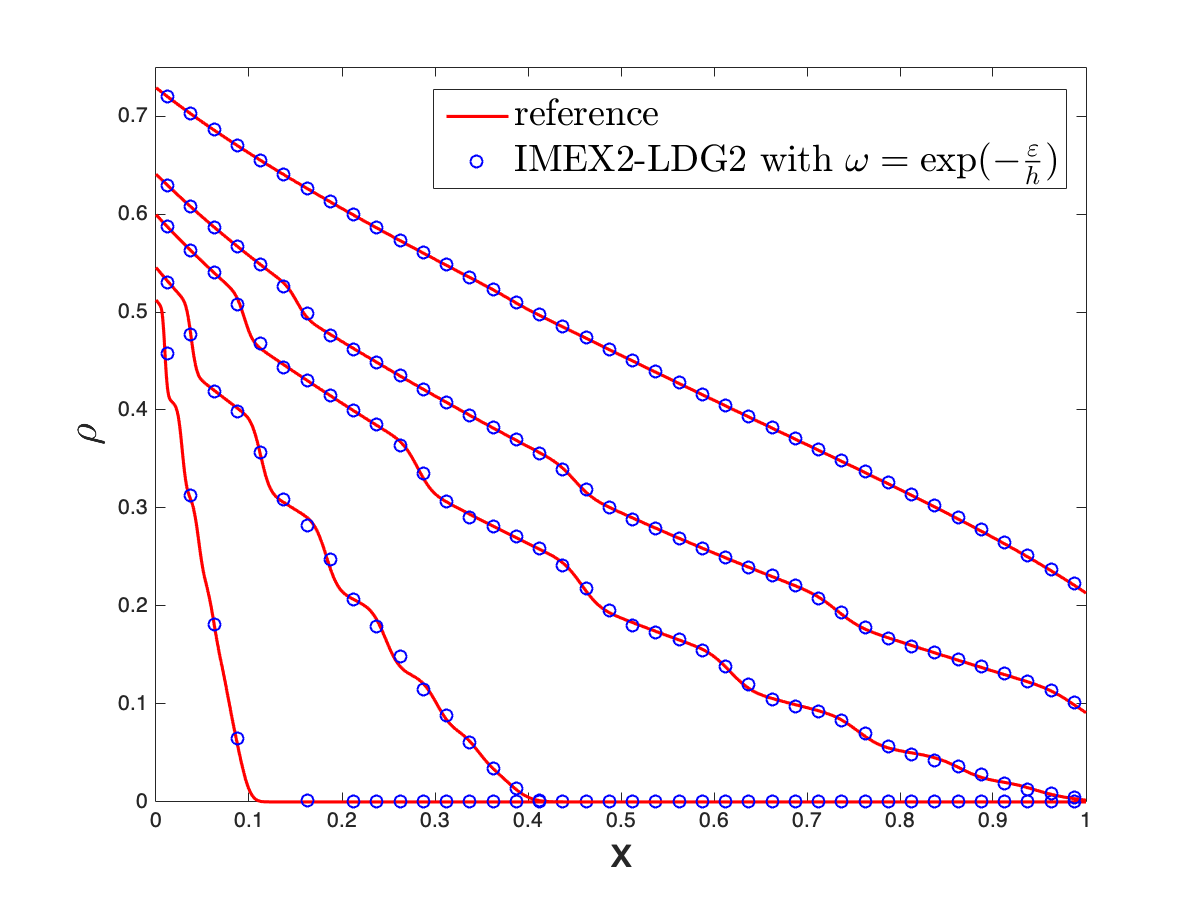}
    }
   \subfigure[$\rho$ for $\vareps=1$, IMEX$3$-LDG$3$ with $\omega=\exp(-\frac{\varepsilon}{h})$]
  {
        \includegraphics[width=0.31\linewidth] {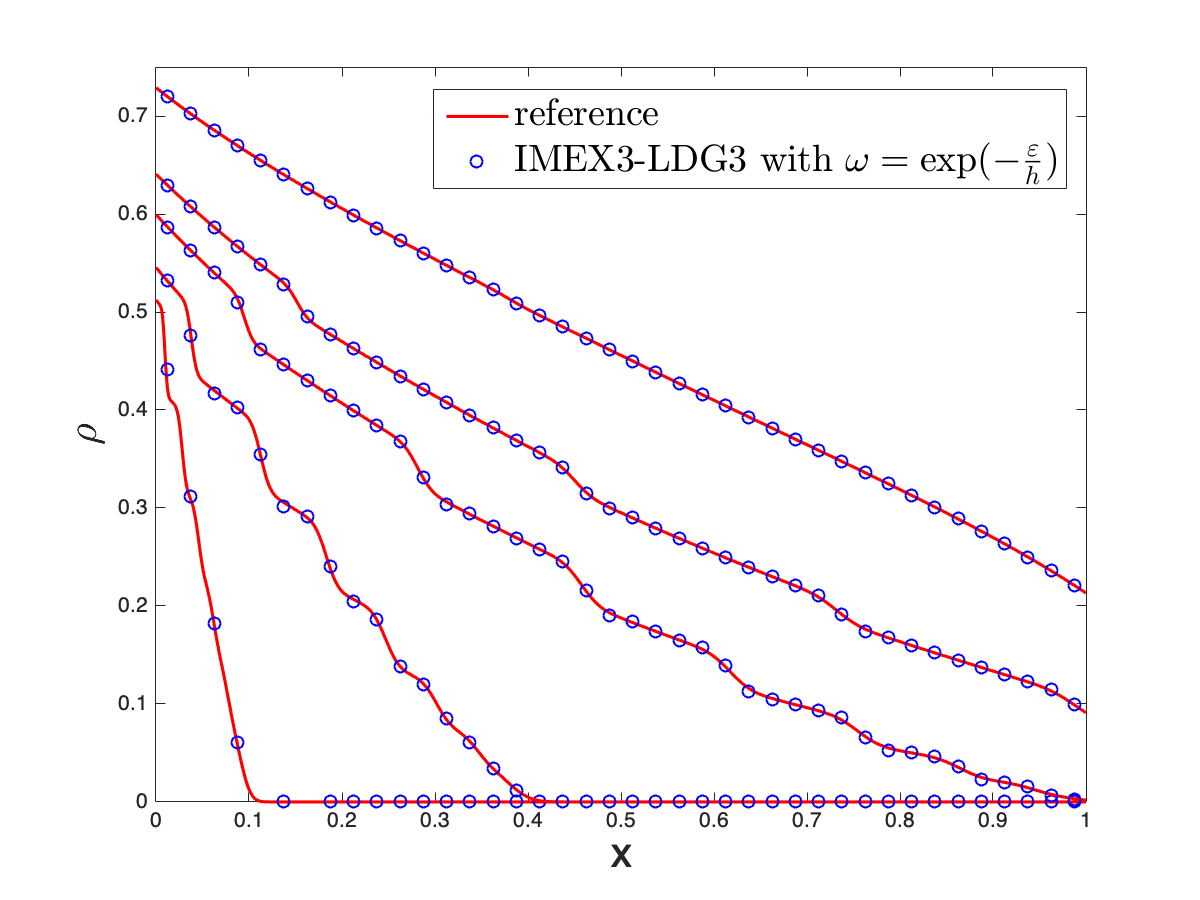}
    }
    \\
      \subfigure[$\rho$ for $\vareps=1$, IMEX$1$-LDG$1$ with $\omega=1$]
  {
     \includegraphics[width=0.31\linewidth]{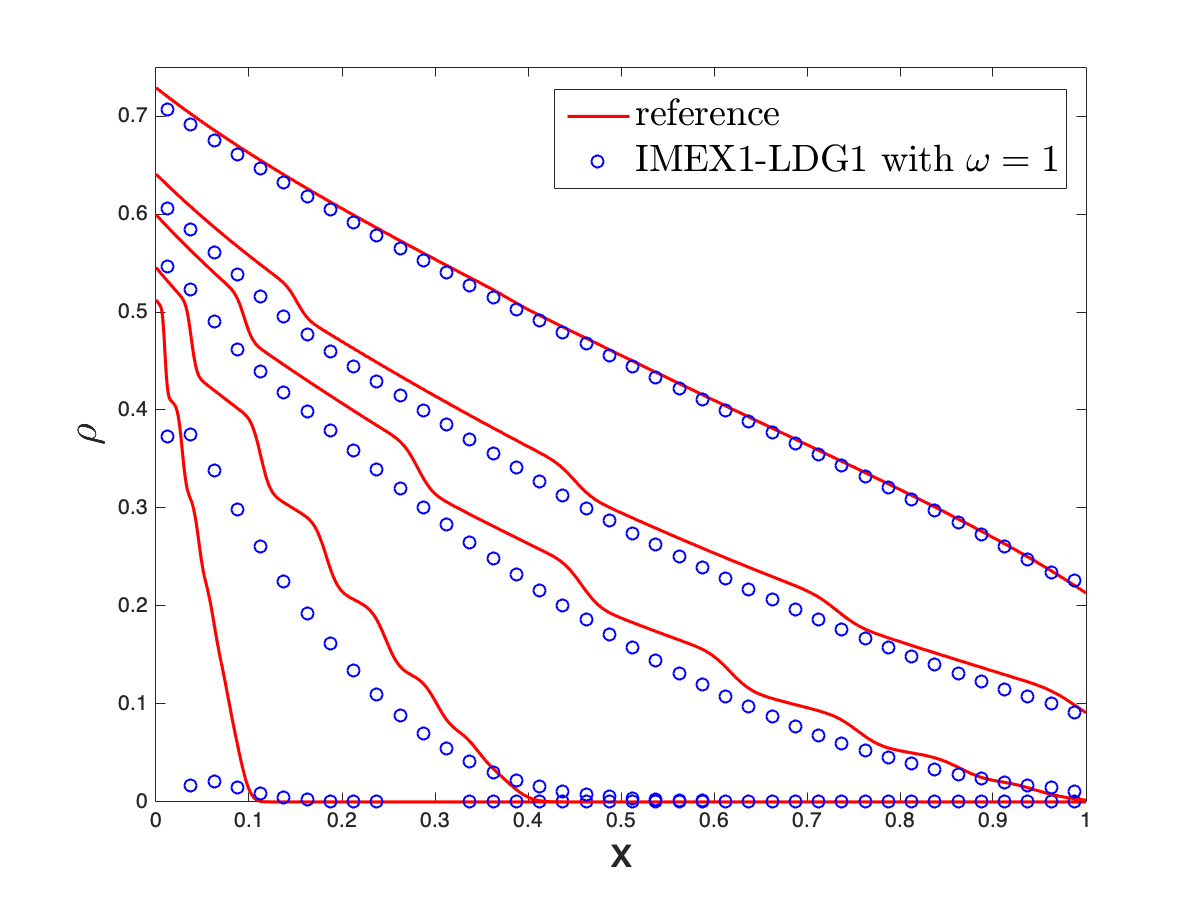}
    }
  \subfigure[$\rho$ for $\vareps=1$, IMEX$2$-LDG$2$ with $\omega=1$]
  {
        \includegraphics[width=0.31\linewidth] {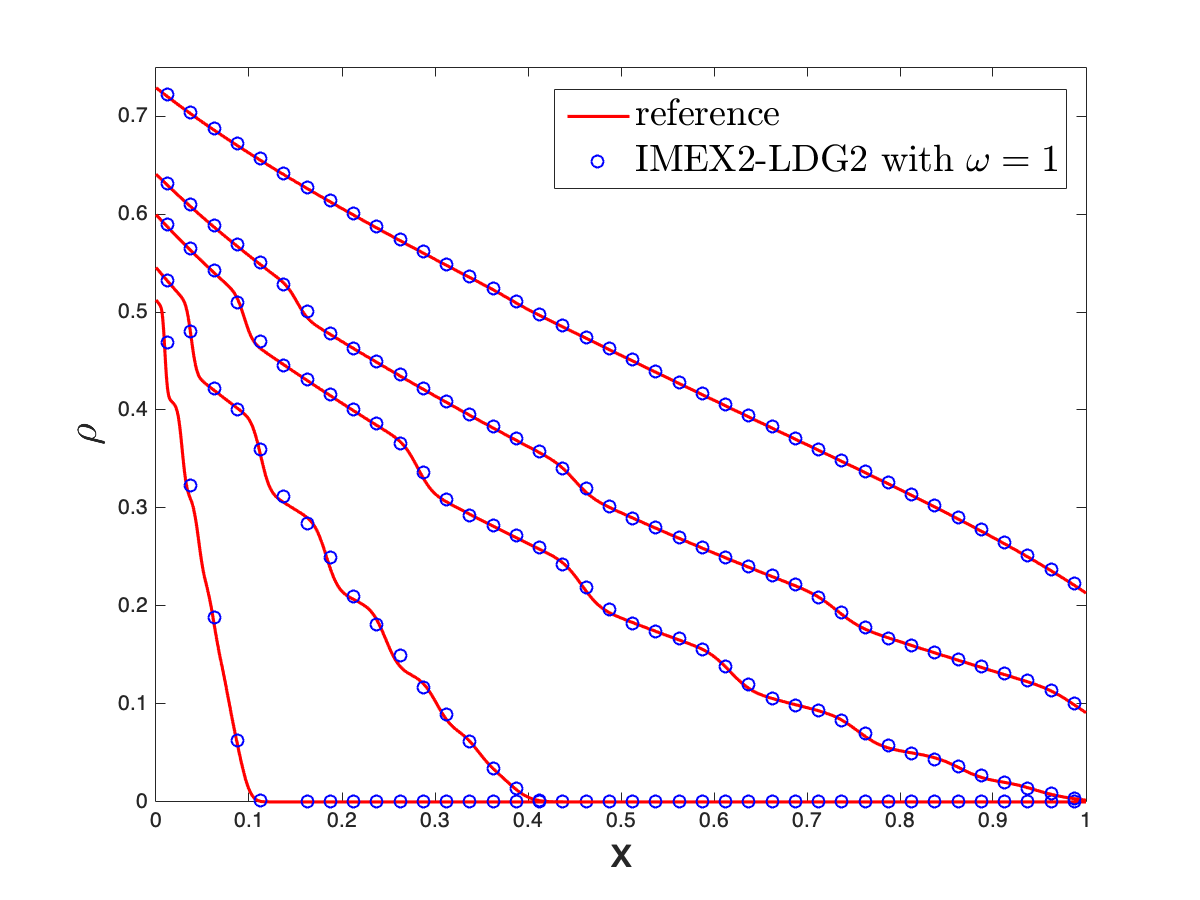}
    }
   \subfigure[$\rho$ for $\vareps=1$, IMEX$3$-LDG$3$ with $\omega=1$]
  {
        \includegraphics[width=0.31\linewidth] {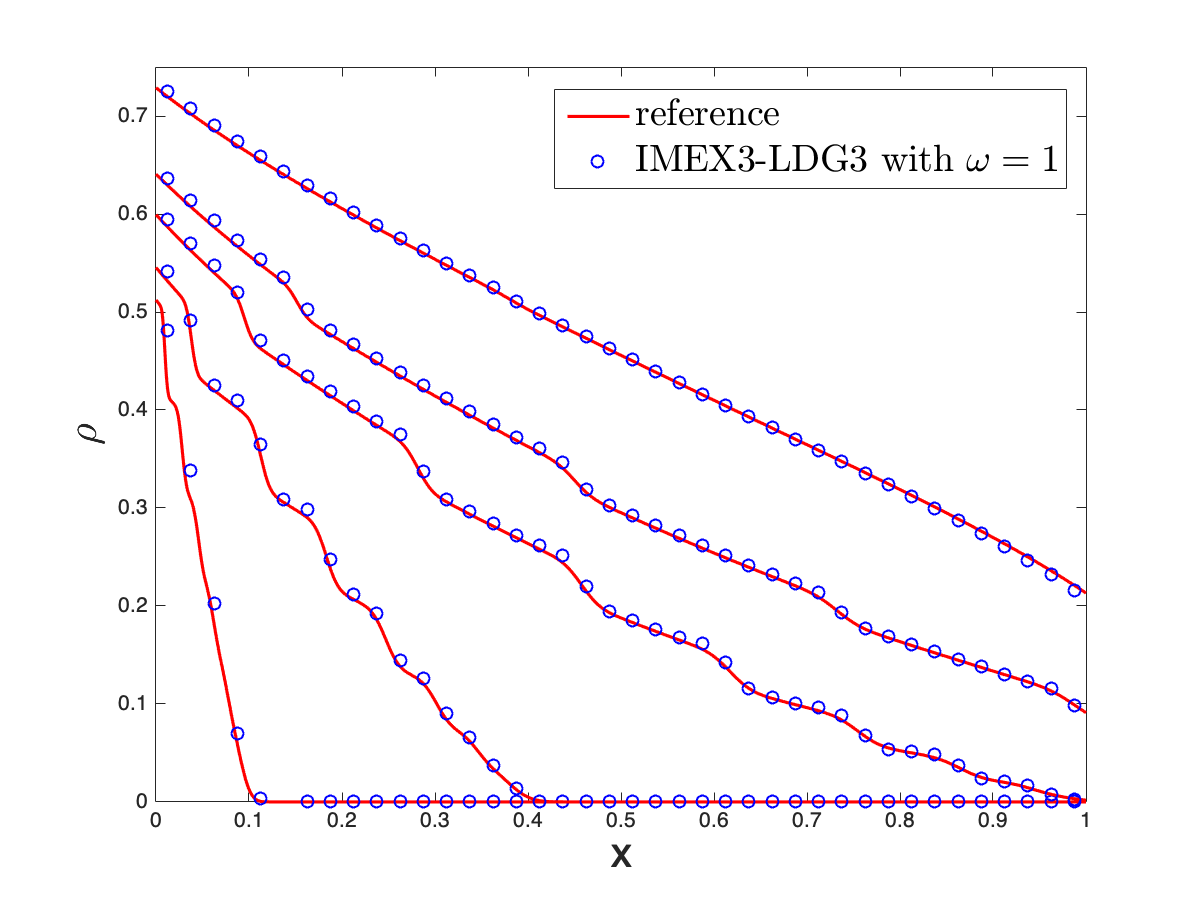}
    }
    \\
     \subfigure[$\rho$ for $\vareps=10^{-8}$, IMEX$1$-DG$1$-S]
  {
     \includegraphics[width=0.31\linewidth]{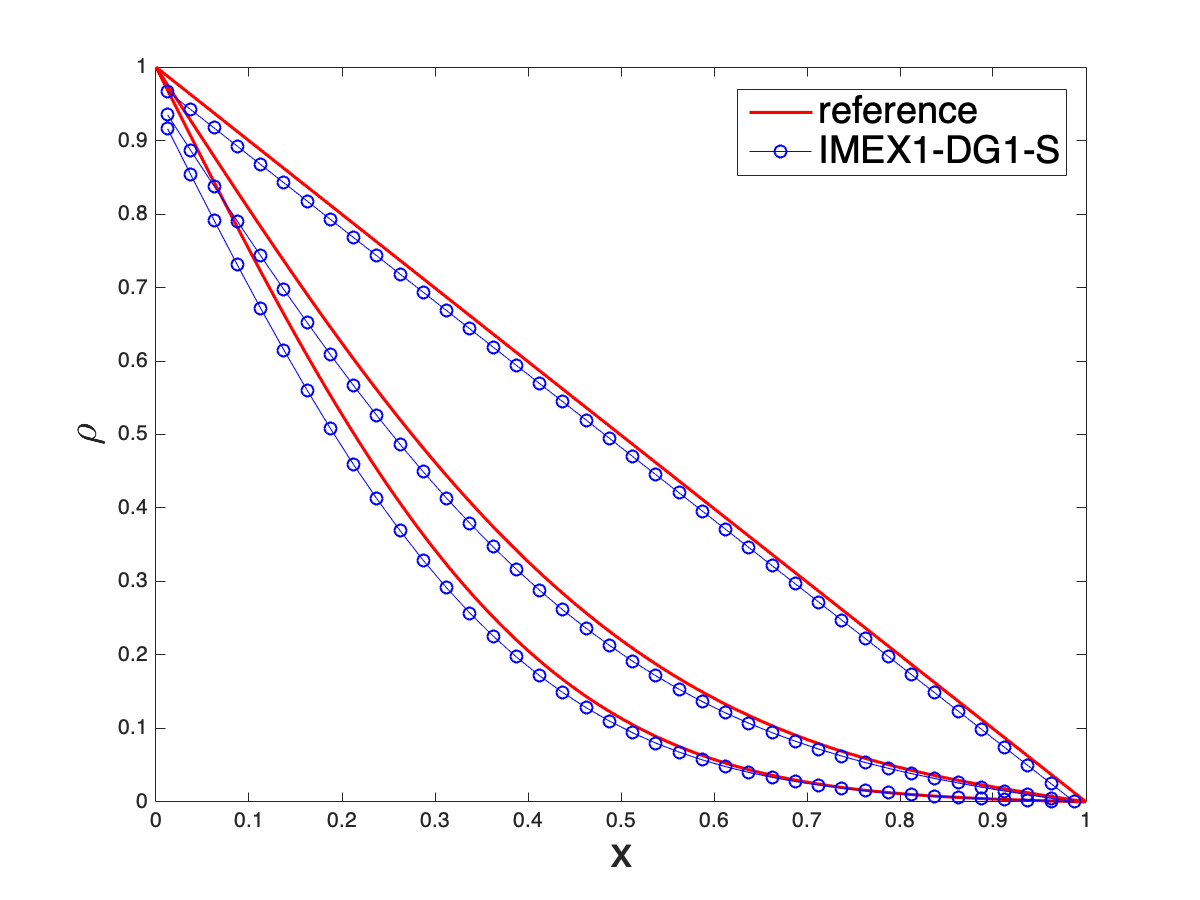}
    }
  \subfigure[$\rho$ for $\vareps=10^{-8}$, IMEX$2$-DG$2$-S]
  {
        \includegraphics[width=0.31\linewidth] {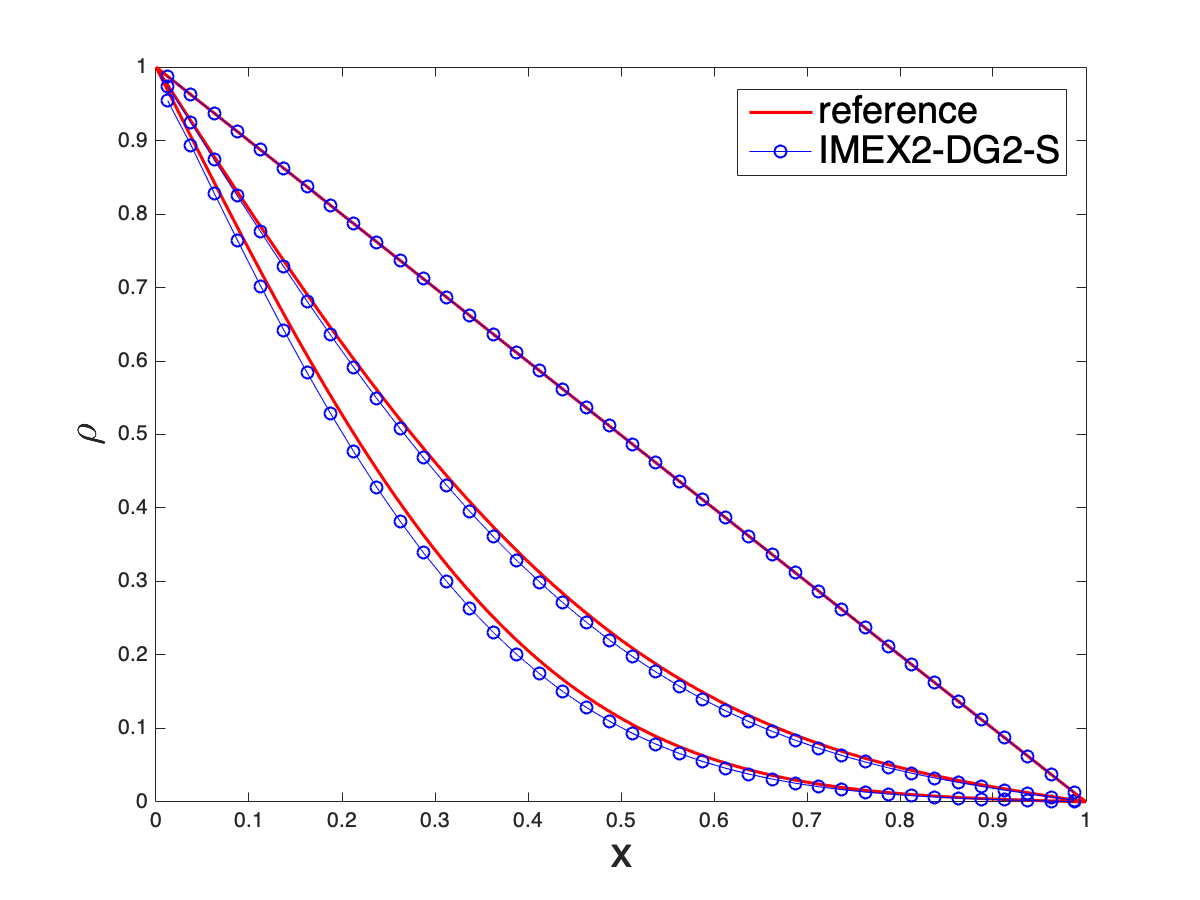}
    }
   \subfigure[$\rho$ for $\vareps=10^{-8}$, IMEX$3$-DG$3$-S]
  {
        \includegraphics[width=0.31\linewidth] {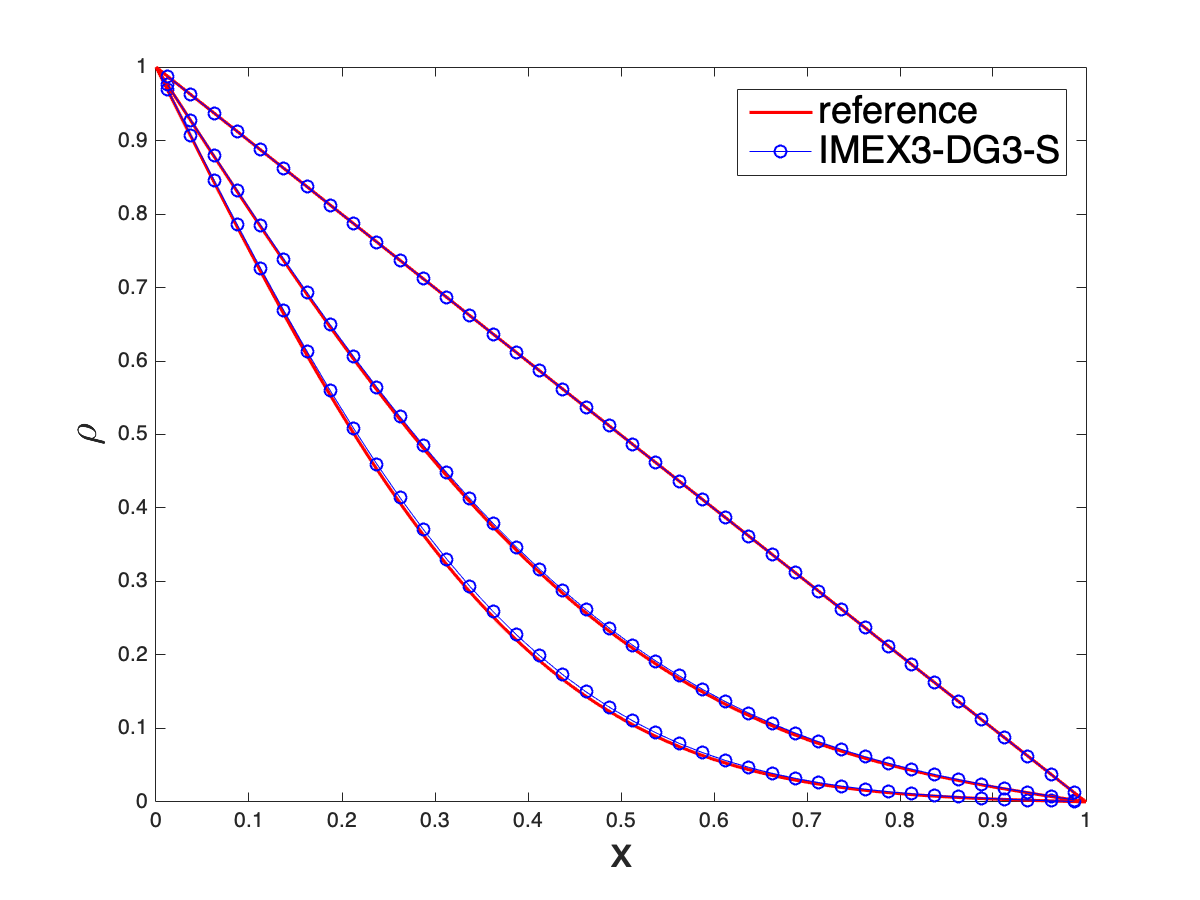}
    }
  \caption{ Example 4: diffusive and kinetic regime with isotropic inflow Dirichlet boundary conditions for one-group transport equation. Top three rows: $\rho$ for $\vareps=1$ and $T=0.1,\;0.4,\;1.0,\;1.6,\;4.0$; Bottom row: $\rho$ for $\vareps=10^{-8}$ and $T=0.15,\;0.25,\;2.0$.
  		}
  \label{fig:dirichlet}  
\end{figure}
%%%%%%%%%%%%%%%%%%%%%%%%%%%%%%%%%%%%%%%%%

%%%%%%%%%%%%%%%%%%%%%%%%%%%%%%%%%%%%%%%%%%%%%%%%
% Riemann problem
%%%%%%%%%%%%%%%%%%%%%%%%%%%%%%%%%%%%%%%%%%%%%%%%
%\begin{comment}
%%%%%%%%%%%%%%%%%%%%%%%%%%%%%%%%%%%%%%%%%%%%%%%%
\smallskip
\textbf{Example 5: Riemann problem for telegraph equation \cite{boscarino2013implicit,jang2015high}.} We consider a Riemann problem with $\Omega_v=\{-1,1\}$, $\sigma_s=1$, $\sigma_a=0$ and the initial data
\begin{align}
\begin{cases}
	\rho(x,0)=2,\; g(x,v,0)=0,\; x\leq 0,\\
	\rho(x,0)=1,\; g(x,v,0)=0,\; x>0.
\end{cases}
\end{align}
Two different cases are considered: the more kinetic case with $\vareps=0.7$ and  $\Omega_x=[-1,1]$, and the more diffusive case with
$\vareps=10^{-6}$ and $\Omega_x=[-2,2]$. For both,  a uniform partition of $\Omega_x$ with $h=\frac{1}{40}$ is used, and the final time is set as $T=0.15$. Numerical results for $\rho$ and $j(x,t)=\lgl vg \rgl$ are presented in Figure \ref{fig:riemann1} and Figure \ref{fig:riemann2}. The reference solution for $\vareps=0.7$ is obtained by the first order forward Euler upwind finite difference scheme solving \eqref{eq:f_model} with a uniform mesh $h=\frac{1}{1000}$ and $\Dt=0.05\varepsilon h$. The reference solution for $\vareps=10^{-6}$ is calculated by a central difference scheme solving the diffusion limit \eqref{eq:diffusion_limit} with $h=\frac{1}{1000}$ and $\Dt = 0.25h^2$. 

For the kinetic case 
%regime
with $\vareps=0.7$, results from all schemes match the reference solution well. 
Compared with the first order scheme,
 the second and the third order schemes give less dissipative results and capture the sharp features better. Small oscillation near discontinuity can be further reduced 
 %eliminated 
 by  applying nonlinear limiters. With the discontinuity present in the solution, when IMEX-LDG schemes are applied to this example (see Section 6.1.2 in \cite{peng2019stability}), the quality of the computed solutions really depends on the choice of the weight function.  For the diffusive case with  $\vareps=10^{-6}$, all schemes capture the  solution well, and high order schemes show better resolutions.

%%%%%%%%%%%%%%%%%%%%%%%%%%%%%%%%%%%%%%%%%
\begin{figure}
  \centering
  \subfigure[$\rho$, IMEX$1$-DG$1$-S]
  {
     \includegraphics[width=0.31\linewidth]{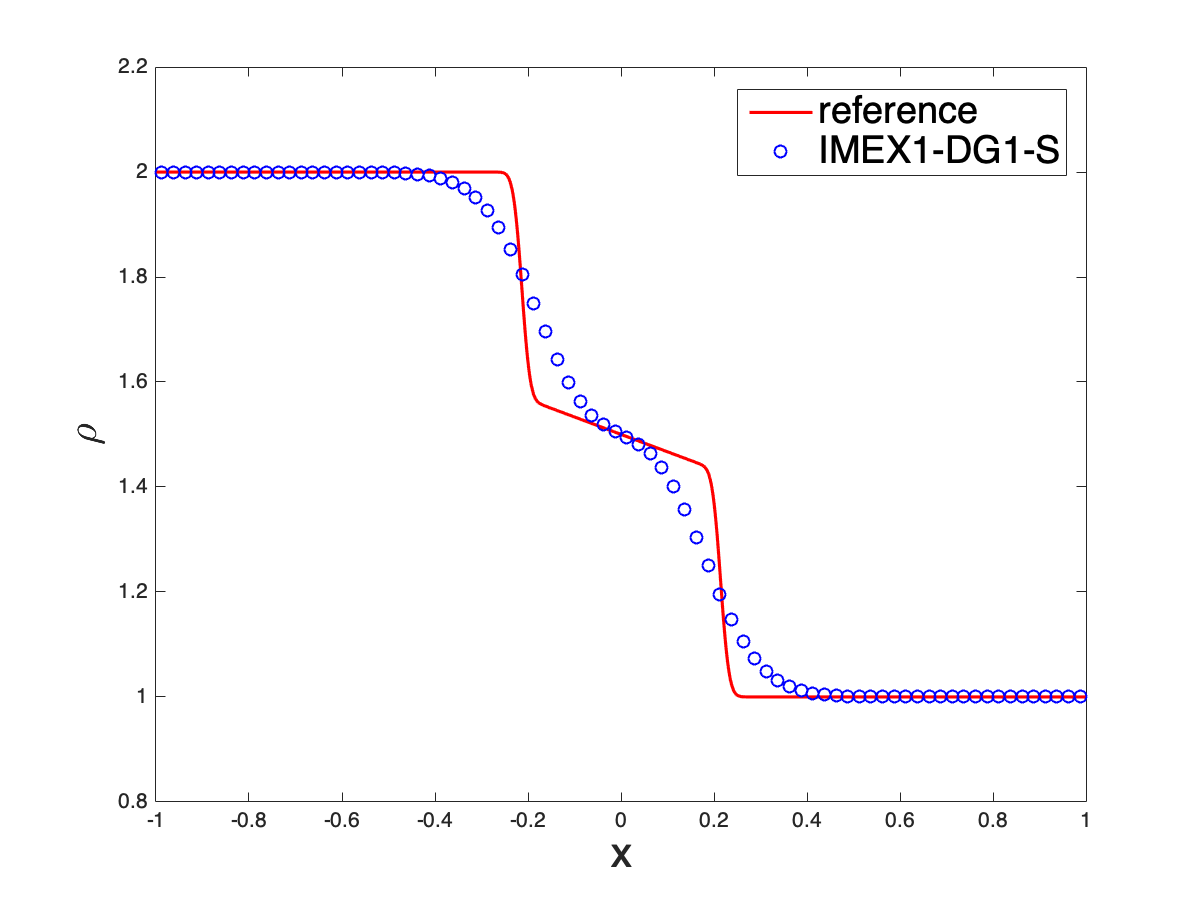}
    }
  \subfigure[$\rho$, IMEX$2$-DG$2$-S]
  {
        \includegraphics[width=0.31\linewidth] {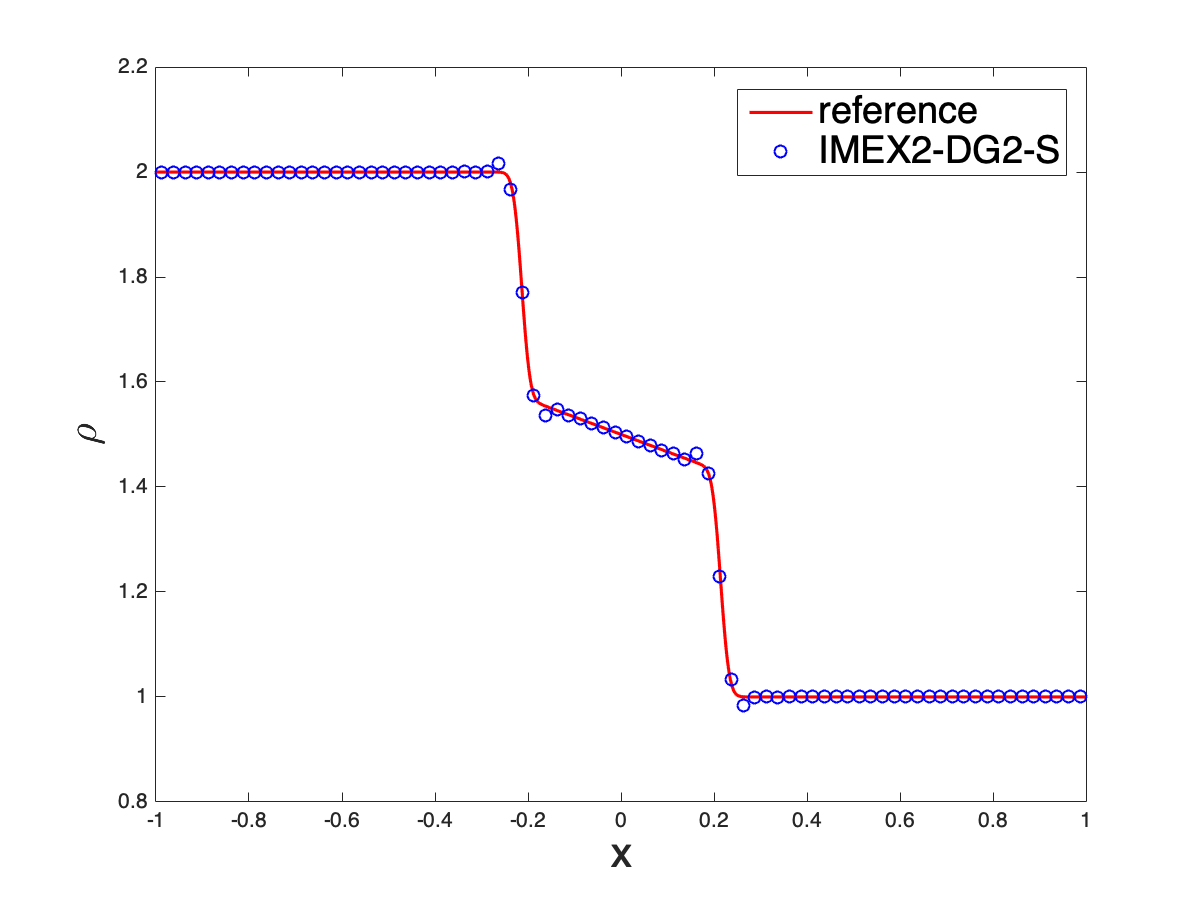}
    }
   \subfigure[$\rho$, IMEX$3$-DG$3$-S]
  {
        \includegraphics[width=0.31\linewidth] {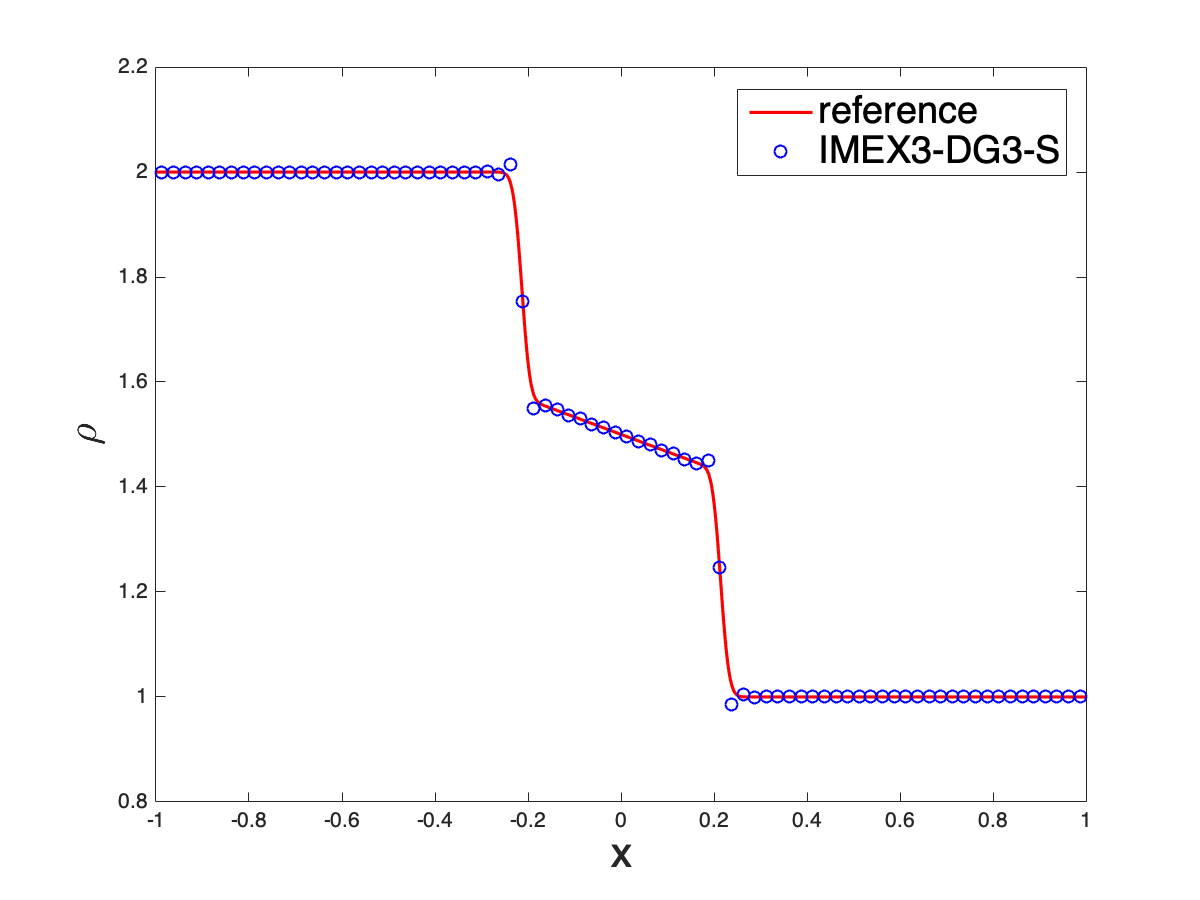}
    }
    \\
      \subfigure[$j$, IMEX$1$-DG$1$-S]
  {
     \includegraphics[width=0.31\linewidth]{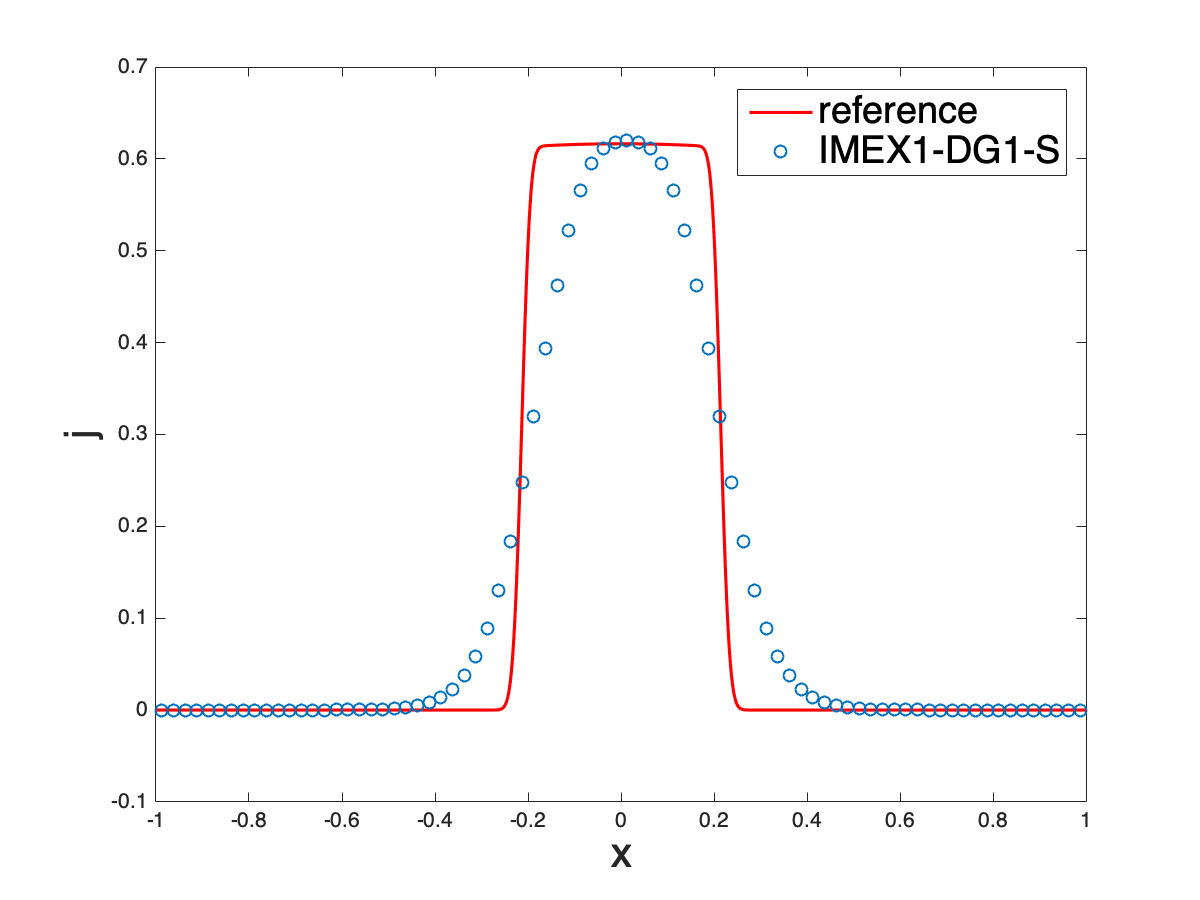}
    }
  \subfigure[$j$, IMEX$2$-DG$2$-S]
  {
        \includegraphics[width=0.31\linewidth] {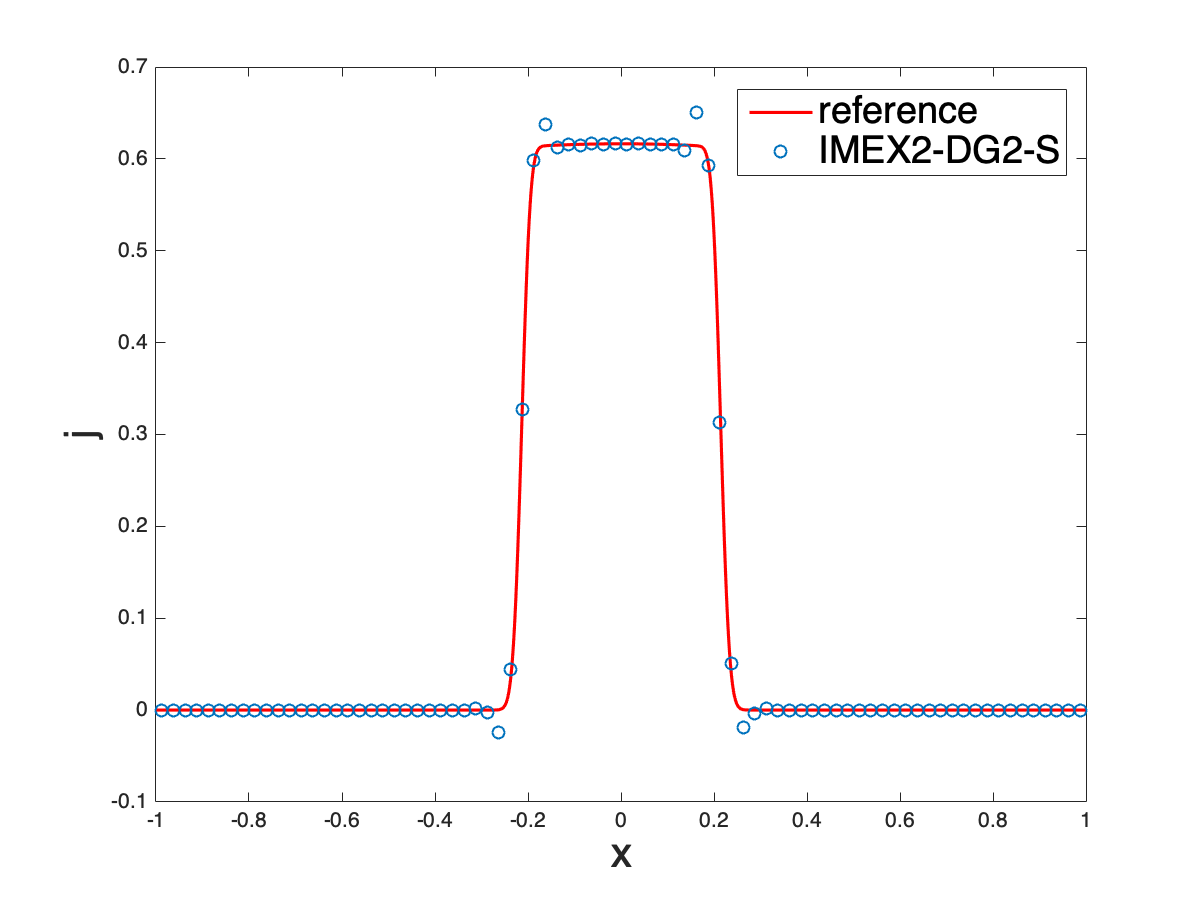}
    }
   \subfigure[$j$, IMEX$3$-DG$3$-S]
  {
        \includegraphics[width=0.31\linewidth] {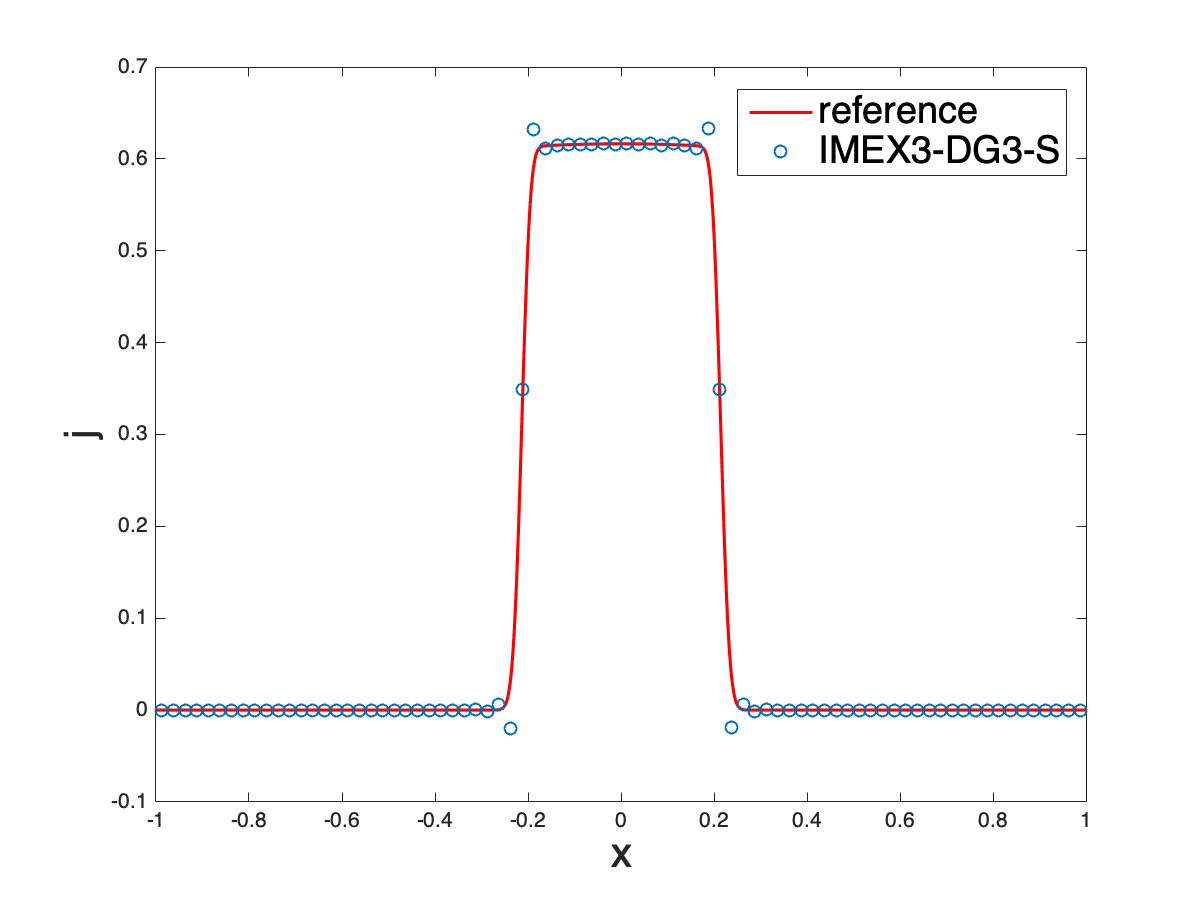}
    }
  \caption{ Example 5: Riemann problem for the telegraph equation. $\vareps=0.7$ and $T=0.15$.
  		}
  \label{fig:riemann1}  %% label for entire figure
\end{figure}

%%%%%%%%%%%%%%%%%%%%%%%%%%%%%%%%%%%%%%%%%
\begin{figure}
  \centering
  \subfigure[$\rho$, IMEX$1$-DG$1$-S]
  {
     \includegraphics[width=0.31\linewidth]{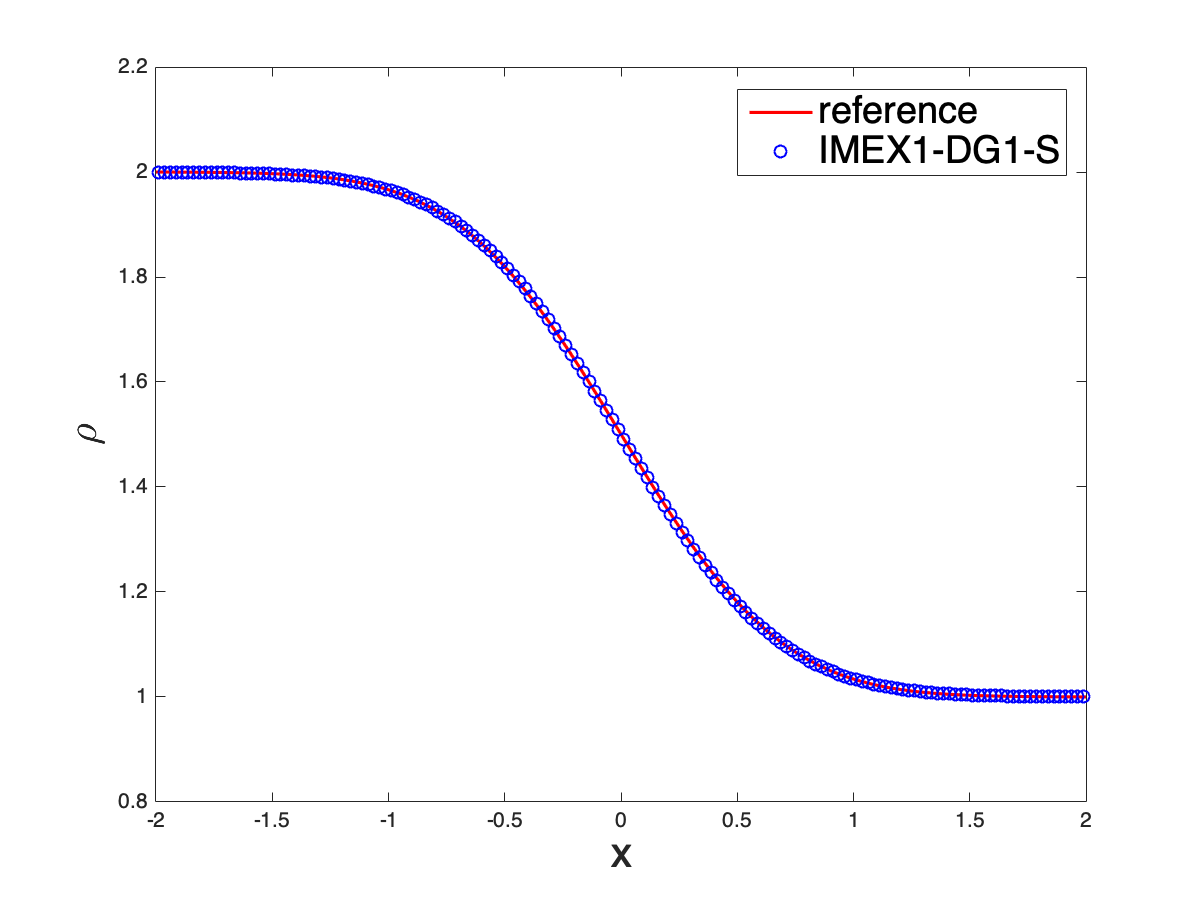}
    }
  \subfigure[$\rho$, IMEX$2$-DG$2$-S]
  {
        \includegraphics[width=0.31\linewidth] {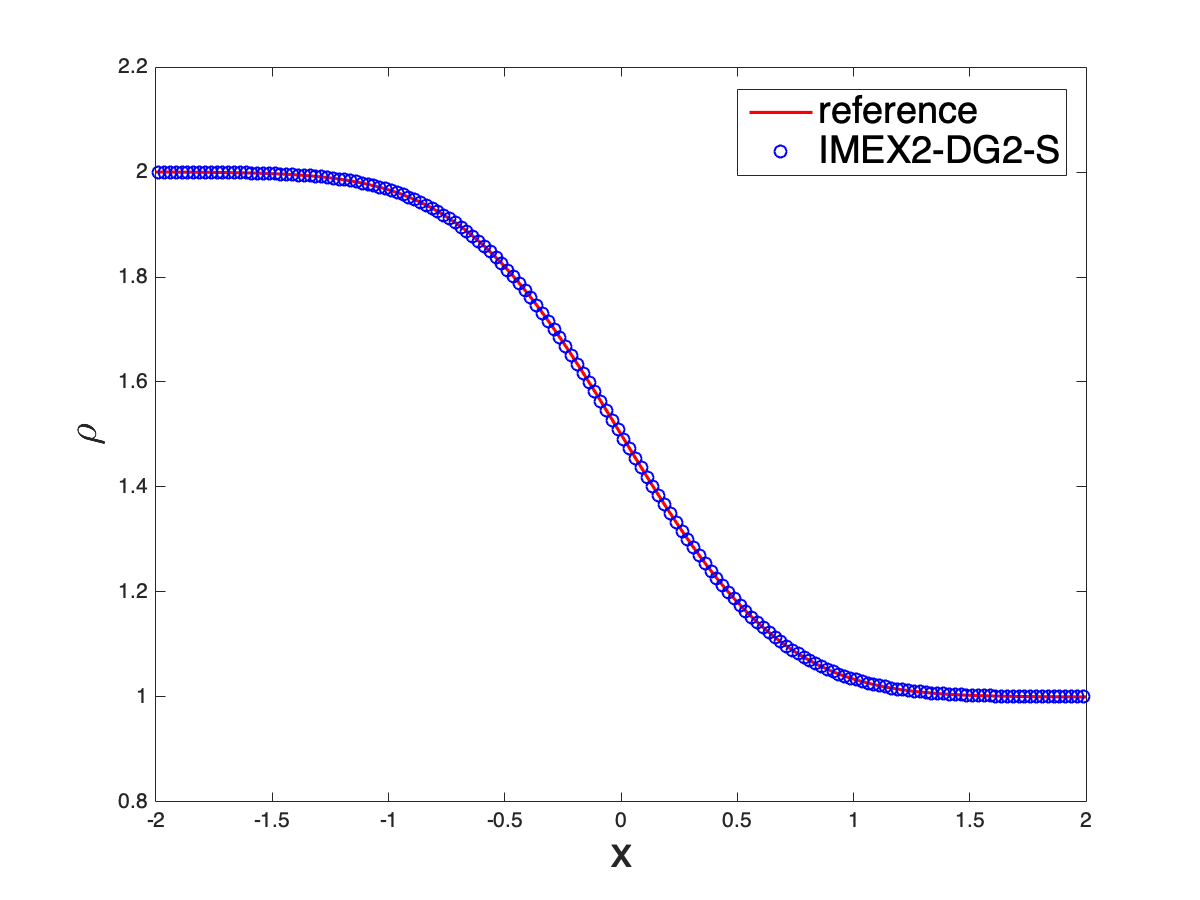}
    }
   \subfigure[$\rho$, IMEX$3$-DG$3$-S]
  {
        \includegraphics[width=0.31\linewidth] {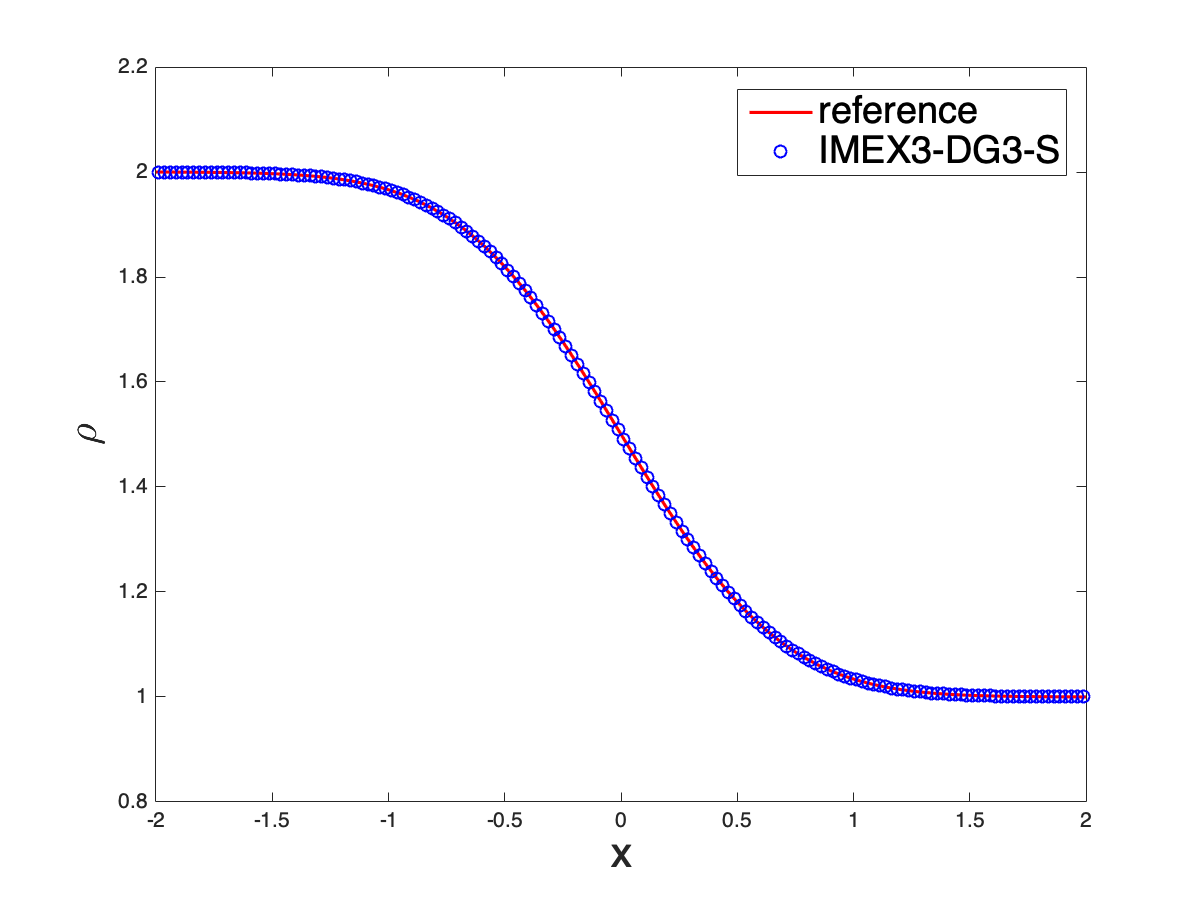}
    }
    \\
      \subfigure[$j$, IMEX$1$-DG$1$-S]
  {
     \includegraphics[width=0.31\linewidth]{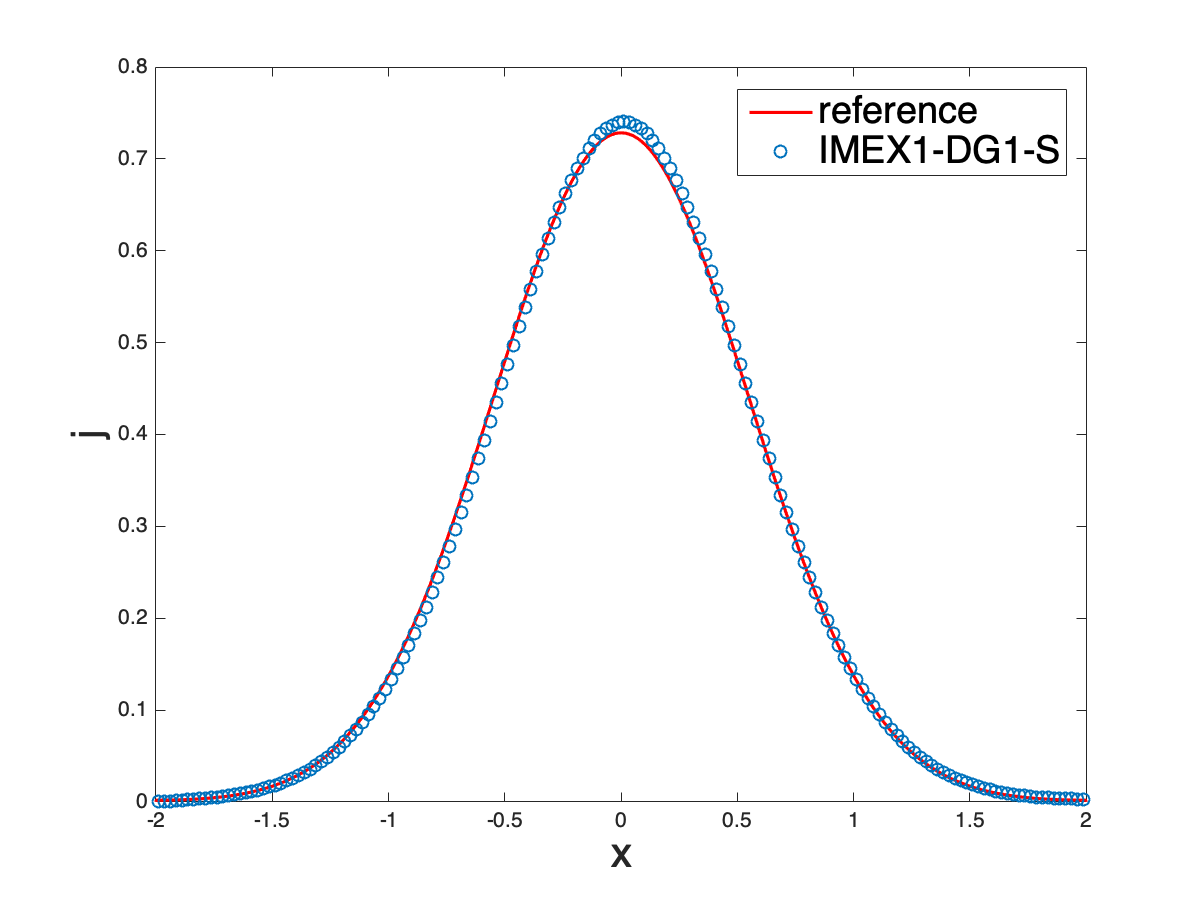}
    }
  \subfigure[$j$, IMEX$2$-DG$2$-S]
  {
        \includegraphics[width=0.31\linewidth] {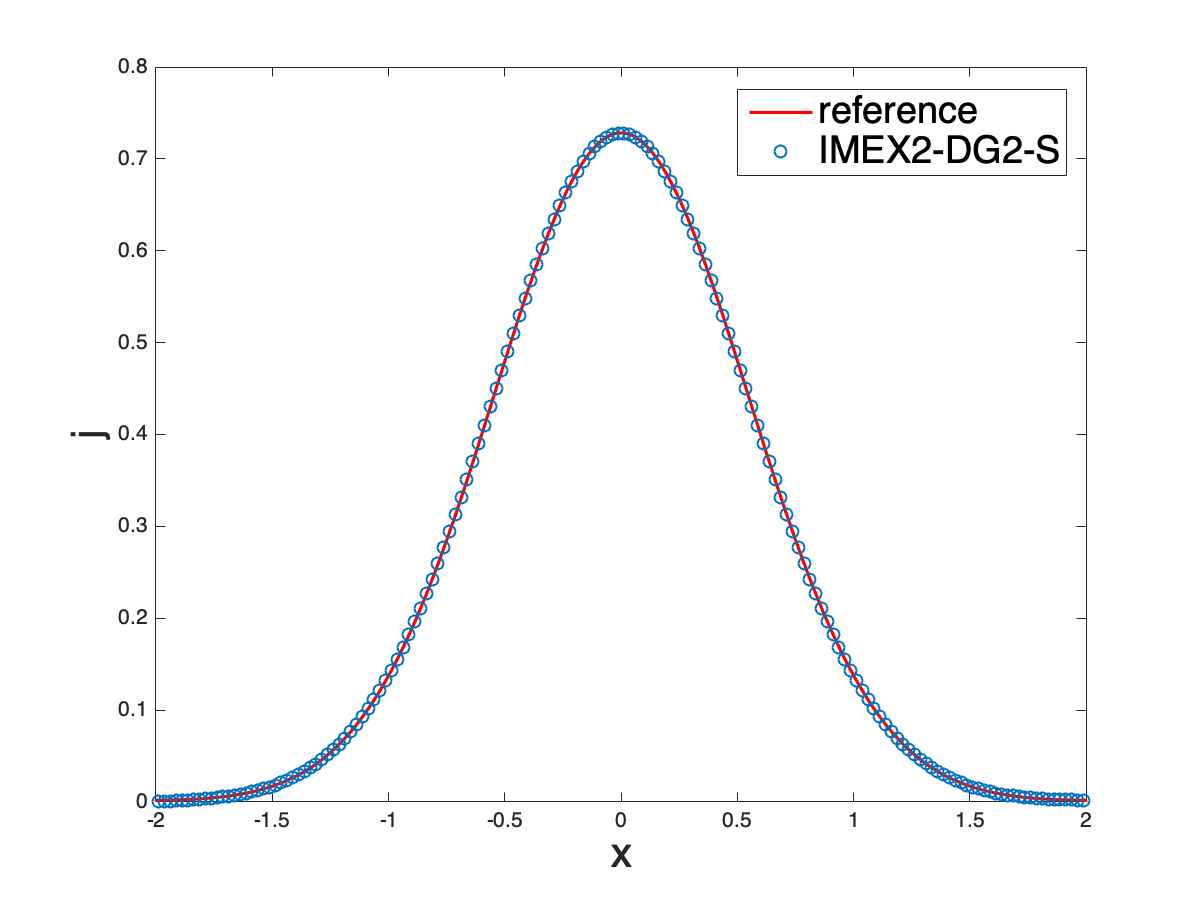}
    }
   \subfigure[$j$, IMEX$3$-DG$3$-S]
  {
        \includegraphics[width=0.31\linewidth] {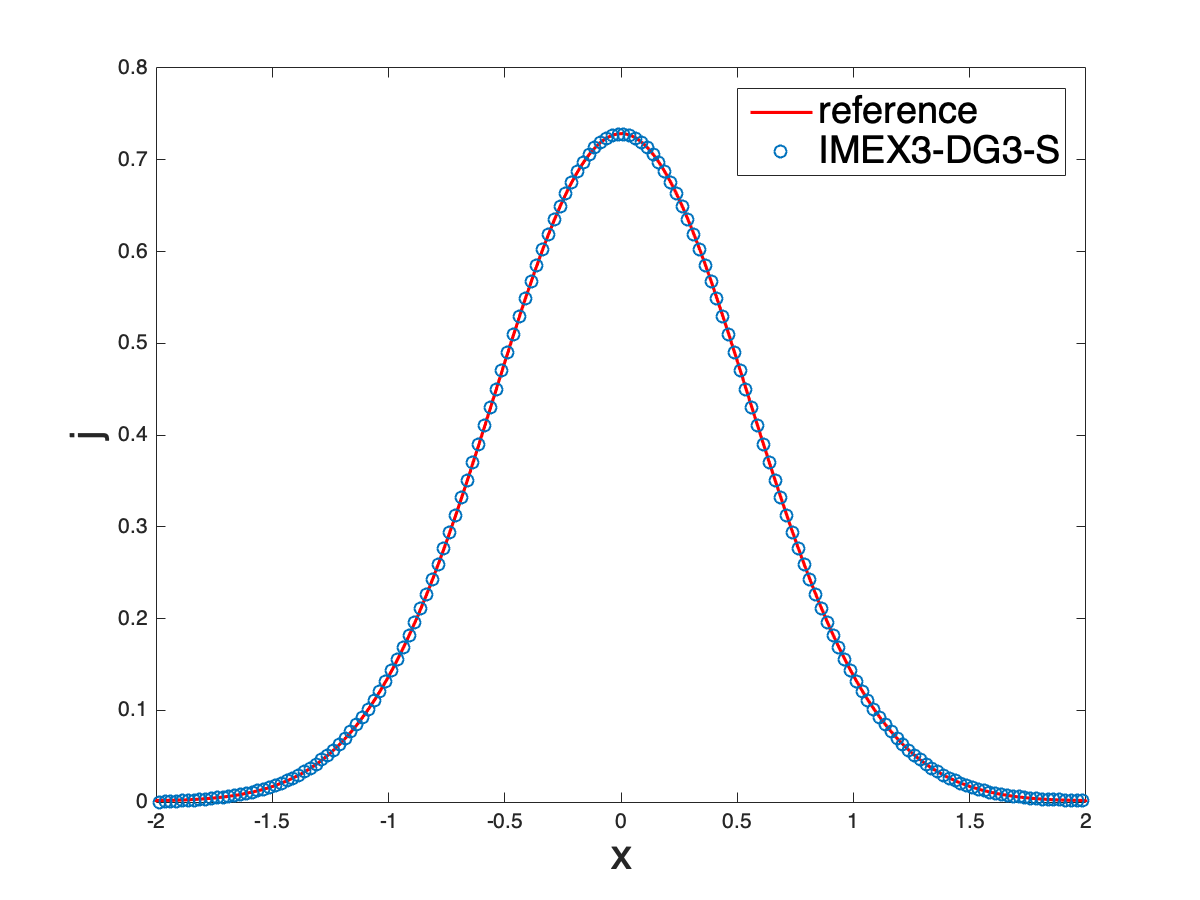}
    }
  \caption{ Example 5: Riemann problem for telegraph equation. $\vareps=10^{-6}$ and $T=0.15$.
  		}
  \label{fig:riemann2}  %% label for entire figure
\end{figure}
%%%%%%%%%%%%%%%%%%%%%%%%%%%%%%%%%%%%%%%%%
%\end{comment}

\section{Conclusions}
\label{sec:conclusion}

To design AP schemes with unconditional stability in the diffusive regime, numerical schemes are developed in \cite{boscarino2013implicit,peng2019stability} based on an additional reformulation to the decomposed system. The key of the additional reformulation is to introduce a weighted diffusive term. In this paper, to avoid issues related to the ad-hoc choice of the weight function, we design   IMEX-DG-S schemes by applying a new implicit-explicit temporal strategy. Asymptotic analysis confirms the AP property of the proposed schemes. Energy type stability analysis for the IMEX1-DG1-S scheme and Fourier type stability analysis for the IMEX$k$-DG$k$-S scheme, $k=1, 2, 3$, are presented. These analyses verify uniform stability of the schemes with respect to $\varepsilon$ and  unconditional stability in the diffusive regime. To achieve these AP and stability properties with computational cost similar to the IMEX-LDG schemes in \cite{peng2019stability}, the Schur complement is applied on the linear solver level. Numerical examples are presented to demonstrate the performance of the IMEX-DG-S schemes and their advantages over the weight-dependent IMEX-LDG schemes in \cite{peng2019stability}.

%\input{./source/appendix}

%%%%%%%%%%%%%%%%%%%

\renewcommand\refname{References}
\bibliographystyle{plain}
\bibliography{bib_peng}
\end{document}